\documentclass[reqno, 11pt]{amsart}%
\usepackage{amsaddr} 
\usepackage{amssymb}
\usepackage[nesting]{hyperref}
\usepackage[pdftex]{graphicx}
\usepackage{listings}
\usepackage{multirow}
\usepackage{placeins}
\usepackage{color}
\usepackage{lscape}
\usepackage{dsfont}
\usepackage{amsmath}
\usepackage{amsfonts}
\usepackage{verbatim}
\usepackage{accents} 
\usepackage{todonotes}
\usepackage{mathtools}
\usepackage{bbm}
\usepackage{natbib}
\usepackage{lipsum}
\usepackage{caption}
\usepackage{subcaption}
\usepackage{afterpage}
\usepackage{booktabs}

\usepackage[ruled,vlined]{algorithm2e}

\usepackage{xcolor}
\usepackage{upgreek}
\usepackage{float}

\newcommand\cA{\mathcal A}

\newcommand\cF{\mathcal F}

\newcommand\cL{\mathcal L}

\newcommand\cN{\mathcal N}

\newcommand\cS{\mathcal S}

\newcommand\EE{\mathbb E}

\newcommand\PP{\mathbb P}
\newcommand\NN{\mathbb N}
\newcommand\RR{\mathbb R}
\newcommand\TT{\mathbb T}

\newcommand\ZZ{\mathbb Z}

\newtheorem{theorem}{Theorem}

\newtheorem{lemma}[theorem]{Lemma}

\theoremstyle{definition}

\newtheorem{remark}[theorem]{Remark}

\newtheorem{prob}[theorem]{Problem}

\newtheorem{assumption}{Assumption}  
\newcommand{\BlackBoxes}{\global\overfullrule5pt}

\BlackBoxes
\usepackage{lastpage}

\begin{document}

\title[MDP of the Third Kind]{Markov Decision Processes of the Third Kind: Learning Distributions by Policy Gradient Descent}

\author[N. \smash{B\"auerle}]{Nicole B\"auerle}
\address[N. B\"auerle]{Department of Mathematics,
Karlsruhe Institute of Technology (KIT), D-76128 Karlsruhe, Germany, \href{mailto:nicole.baeuerle@kit.edu}{nicole.baeuerle@kit.edu}}

\author[A. \smash{Vasileiadis}]{Athanasios Vasileiadis}
\address[A. Vasileiadis]{Department of Mathematics,
Karlsruhe Institute of Technology (KIT), D-76128 Karlsruhe, Germany, \href{mailto:nicole.baeuerle@kit.edu}{athanasios.vasileiadis@kit.edu}, 
}



\begin{abstract}
The goal of this paper is to analyze distributional Markov Decision Processes as a class of control problems in which the objective is to learn policies that steer the distribution of a cumulative reward toward a prescribed target law, rather than optimizing an expected value or a risk functional. To solve the resulting distributional control problem, we propose a policy-gradient algorithm based on neural-network parameterizations of randomized Markov policies, defined on an augmented state space and a sample-based evaluation of the characteristic-function loss. Under mild regularity and growth assumptions, we prove convergence of the algorithm to stationary points using stochastic approximation techniques. Several numerical experiments illustrate the ability of the method to match complex target distributions, recover classical optimal policies when they exist, and reveal intrinsic non-uniqueness phenomena specific to distributional control.
\end{abstract}

\maketitle

\makeatletter \providecommand\@dotsep{5} \makeatother

\vspace{0.5cm}
\begin{minipage}{14cm}
{\small
\begin{description}
\item[\rm \textsc{ Key words}]
{\small Distributional RL, Distributional MDPs, Policy Gradient, Randomized controls, Characteristic Functions  }

\end{description}
}
\end{minipage}\\

\section{Introduction}
Since the introduction of Reinforcement Learning (RL), tremendous progress has been achieved in designing algorithms that allow agents to learn optimal behavior through interaction with uncertain environments. Markov Decision Processes (MDPs) have so far provided the theoretical framework to model these problems, \cite{Puterman1994MDP, Bertsekas-Shreve,Powell2011ADP,BauerleRieder}. When all the involved functions, i.e., the dynamics and rewards/costs, are known and the spaces are finite, we can use exact dynamic programming methods usually in a tabular form. When the spaces are continuous, we need some form of function approximation. For the purpose of this article, RL is what we do when we only observe samples from the states, actions and rewards, without access to any part of the theoretical model.     
 
 Classical RL methods focus on estimating the maximal expected cumulative return, the classical value function of control, which has led to impressive successes across domains ranging from notoriously hard games like chess and Go to modern robotics. We call this mathematical formulation \emph{MDPs of the first kind}. However, not all problems we  face can be solved by classical dynamic programming, for example portfolio selection under model uncertainty, robust control and in general risk-aware optimization, \cite{ContModelUncertainty,bauerleglaunerDistributionallyRobustMDPs,PrashanthFu2022RiskSensitiveRL,CoacheJaimungalRLconvexriskmeasures,coache2025robustreinforcementlearningdynamic,BauerleJaskiewiczMDPRiskSensitiveOverview,Guin_2026} call for a different formulation. These problems can often be solved by a state space augmentation, see \cite{bauerle2011markov}.
  We  collectively call this mathematical formulation \emph{MDPs of the second kind}. Finally, motivated by recent advances in finance, \cite{LassanceVrins2023TargetDistribution, ColaneriEisenbergSalterini2023TerminalDist}  and AI (especially generative AI), we are interested in optimal controls that steer the whole distribution of (discounted) cumulative rewards toward a certain target distribution. We call these (Distributional) MDPs,  \emph{MDPs of the third kind}. In this article we restrict to problems with a finite time horizon.
 
 In complete analogy, Distributional Reinforcement Learning (DRL) addresses the problem in a practical way, using only samples by modeling the entire distribution of returns, providing a richer representation that can enhance stability, performance, and risk-sensitive decision making. 

\subsection{Motivation}
\label{se:motivation}
To motivate our results we start by giving an answer to a simple variant of our problem that we can actually solve analytically, so we get a flavor of what to expect in more complicated and general cases that we treat in the rest of the article. 

\begin{prob}
\label{prob:informal}
	Find optimal controls that shift the distribution of a terminal cumulative (discounted) reward (cost) such that it matches a predefined target. 
\end{prob}

On a compact space, the Fourier representation of a probability measure is a sequence of coefficients and matching this  after truncating the sequence is a finite-dimensional problem. In particular, we can match directly by identifying the coefficients of the sequence.  

To illustrate this, consider  $\mathbb{T}^{1} = \mathbb{R}/(2\pi\mathbb{Z})$, the one-dimensional torus endowed with its Borel $\sigma$-algebra.
For a finite horizon $T \in \mathbb{N}$, consider the Markov Decision Process
\begin{equation}\label{eq:torusMDP}
s_{t+1} = \bigl(s_{t} + a_{t} + \varepsilon_{t+1}\bigr) \bmod 2\pi,
\qquad t = 0,\dots,T-1,
\end{equation}
where $s_0 \in \mathbb{T}^{1}$ is an 
initial state.
The reward is given by $r_t(s_t,a_t)=0$ for $t=0,1,\ldots,T-1$ and $r_T(s_T,a_T)=s_T.$  Next,  $a_{t}\in\mathcal A\subset\mathbb{T}^{1}$ is the action chosen by the decision maker at stage $t$ which may depend on the observed history of the process. We consider randomized policies here.
Finally, $(\varepsilon_{t})_{t\ge 1}$ is an i.i.d. Gaussian sequence on $\RR$ with zero expectation and variance $\sigma^2$, independent of $s_0$. All sums are taken in $\mathbb{R}$ and reduced modulo $2\pi$ to lie again in $\mathbb{T}^{1}$.

The objective is to control the process in such a way that the distribution $\mu_T$ of the accumulated reward $ \sum_{t=0}^T r_t(s_t,a_t)=s_T$ is as close as possible to a given target distribution $\mu^*$ on $\mathbb{T}^1$. The distance is measured as a distance of characteristic functions. We will be more precise later.
For the purpose of illustration let us assume that $T=1$. In this case we can compute the characteristic function of $s_1$ explicitly. 
 First note that since $s_1\in [0,2\pi]$, the characteristic function of $s_1$ is determined by its frequencies $\widehat \varphi_{s_1}(k) = \EE[e^{iks_1}]$ for $k\in\ZZ$, see for example \cite[Chapter 1]{rudin1962fourier}. Next, for $k\in\ZZ$ and $x\in\RR$ we have $e^{ik(x \bmod 2\pi)}=e^{ikx}$. Thus, when we assume that $a_0$  is independent of $\varepsilon_1$ we obtain
$$ \EE[e^{ik s_1}]=\EE[e^{ik (s_0+a_0+\varepsilon_1)}]= e^{iks_0} \widehat\varphi_{a_0}(k) e^{-\frac12 \sigma^2 k^2}$$
We want to match $\EE[e^{ik s_1}]$ to the characteristic function of the target distribution $\mu^*$ by choosing the distribution of $a_0$ accordingly. More precisely, we match $\widehat \varphi_{s_1}(k) = \widehat{\mu}^*(k) :=  \int_{\TT^1} e^{-i k u} \, d\mu^*(u)$ for $k \in \mathbb{Z}.$ Thus, we obtain
$$\widehat \varphi_{a_0}(k) =\widehat\mu^*(k) \cdot e^{i k s_0}\; e^{\frac{1}{2} \sigma^2 k^2} \quad \forall\, k\in \ZZ. $$
for the characteristic function of the policy we are looking for. Inverting this yields the optimal distribution of $a_0.$

It is relatively obvious from the discussion above that when we are considering only deterministic controls we just shift the initial condition without any hope to match any nontrivial target distribution. Further, in general we cannot expect to match the target exactly.

\subsection{Extension to Continuous Noncompact Spaces} 
Of course, as illustrative and appealing as it is, this example is of limited use and we present it only for pedagogical reasons. The main question is \emph{how to extend to continuous non compact Borel spaces with nonlinear dynamics} when we cannot trace the functional form of the distribution of controls since  convolutions are no longer available and the Fourier representation cannot be broken down analytically? 

To motivate our choices in Section \ref{sec:Problem+Ass} we first notice that to extend the problem in case of continuous noncompact spaces we need to define a distance between the characteristic functions of the reward $R_T=\sum_{k=0}^{T-1} r_k(s_k,a_k)$ and a target characteristic function $\varphi^*$: 
\begin{align}\label{eq:1}
    \int_{\mathbb{R}} \big|\varphi^*(u) - \EE[e^{iuR_T}]\big|^2 \, w(u)\, du,
\end{align} 
with some weight function $w$ (e.g.\ Gaussian) to make the integral finite and emphasize relevant frequencies. A natural choice for working directly with probability measures could also be the Wasserstein distance but with the serious caveat that it is not differentiable out of the box and thus not easy to use as a loss function for our purposes.  Distances like \eqref{eq:1} are popular in statistics for testing hypothesis about distributions. For special choices of the weight function $w$ we obtain for example the energy distance. The theoretical foundations of problems like this have been laid in \cite{bauerle2025distributionalbellmanequation}. The authors there give a dynamic programming equation for the solution. However, the question still remains how these problems can numerically be solved efficiently.
Thus, in this paper we use artificial neural networks (ANNs) to parameterize the controls where we condition them on the current state and reward and using some external noise in order to learn to sample optimally from the unknown action distribution $\nu$ when we have no access to, or assumptions about it. The ANN parameters are trained by a gradient descent scheme. 

To wrap up, we design a \emph{model dependent} policy gradient algorithm using only samples from the dynamics and rewards to minimize a characteristic function loss. The dependence comes from differentiation of the rewards, but essentially only needs sufficiently smooth rewards and dynamics. Discussion about how to construct a good model is out of the scope of this article. The same algorithm could be considered in a model-free setting where instead of differentiation we use zero-order optimization, but at the price of increased computational cost.

\subsection{Choosing the Target}
There are several ways to choose the target distribution. This may either be a desired risk-return distribution when we think about investment problems where a certain risk profile is targeted. Another setting could be the problem of imitating others: Suppose you can observe the rewards/outcomes of an opponent but you do not know her/his actions. Then you can easily obtain their target characteristic function of cumulative rewards and use this to find  a control which is consistent with this behavior. A third application is to design a Markov model such that the  cumulative rewards behave as observed in nature.

\subsection{Literature Review}
Distributional Reinforcement Learning has attracted  considerable attention in recent years with several important contributions.  We distinguish between DRL as a representation where we are interested in learning the distribution of the accumulated reward $R_T^\sigma$ for policy $\sigma$ while at the same time trying to maximize $\EE[R_T ^\sigma]$, and distributional optimization where we choose the policy to maximize a functional of the return law. 

We begin by summarizing work that belongs to the first category. 
One of the most influential advances is the quantile-regression formulation of  \cite{dabney2018distributional}, which replaces the fixed-support categorical approximation of C51 from the seminal paper \cite{bellemare2017distributional} with a trainable quantile parameterization. This representation yields a theoretically consistent approximation of value-return distributions in the 1-Wasserstein metric and resolves several limitations of earlier approaches, most notably by providing unbiased stochastic gradient estimates via asymmetric quantile losses and by ensuring that the projected Bellman operator is a contraction.  

Further along this line of research on \emph{value-based control} the article by \cite{bellemare2020distributional} introduces the first convergence guarantees for a DRL algorithm combined with function approximation, addressing a central theoretical gap left open by earlier works such as C51 and its analysis by \cite{bellemare2017distributional}. Its key contribution is the construction of a fully Cram\'er-based loss free of softmax and KL divergence which extends the Cram\'er distance to arbitrary real vectors and augments it with a normalization penalty, enabling a new algorithm (S51) whose updates are mathematically tractable. Using this generalized loss, the authors prove that the projected distributional Bellman operator converges under linear function approximation and quantify the resulting approximation error, showing, perhaps unexpectedly, that distributional methods can yield worse expected-value accuracy than classical TD under function approximation. 

In \cite{achab2023onestepdistributionalreinforcementlearning}, the authors introduce a simplified distributional framework that resolves one of the central theoretical difficulties of classical DRL: the instability of the optimality operator in the control case. By restricting attention to the randomness induced by the first transition only, the authors construct one-step distributional Bellman operators that, unlike the original distributional operators, are $\gamma$-contractions in all Wasserstein metrics for both policy evaluation and control, leading to a unified almost-sure convergence theory even without assuming uniqueness of the optimal policy. Using these operators, they develop new Cram\'er-projected algorithms (tabular and deep variants such as OS-C51) whose categorical updates are simpler, cheaper, and provably convergent.
 The second stream of literature tries to maximize expressions of the form 
\begin{align}\label{objective:HF}
    H\left( \EE\left[ f(R_T^\sigma)\right]\right)
\end{align}
for suitable functionals $f,H.$ This leads, for example, to quantile optimization, CVaR optimization or even more general risk measures.
Though these problems can be solved with classical MDP techniques and state space augmentation, it is here even more evident to work with the underlying distribution, \cite{lim2022distributional,li2022quantile,tamar2015optimizing,chow2014algorithms}.

Next, \cite{pires2025optimizing} introduce distributional dynamic programming, a general DP framework for \emph{optimizing  specific statistical functionals} of the 
return distribution of an infinite-horizon discounted MDP.  Their approach is the closest to ours in this article. We agree on the definition of the distributional MDP which augments states with cumulative reward (stock) and combines with a distributional Bellman recursion defined over this augmented MDP, enabling principled optimization of objectives characterized of being of the form \eqref{objective:HF}. In this way, distributional DP broadens the scope of DRL by integrating distributional objectives directly into the dynamic programming principle, complementing earlier work on distributional estimation in both discrete and continuous time. In contrast to \cite{pires2025optimizing}, as in \cite{bauerle2025distributionalbellmanequation} we consider the problem of matching given distributions where we measure the distance by the characteristic functions, a criterion not covered in \cite{pires2025optimizing}. Moreover, we use a finite time horizon and  resort to a lifted distributional MDP (see Section \ref{sec:liftedMDP})  to formally describe the Bellman recursion and optimize over history-dependent policies.

To sum up: The problem studied in this article is of a different nature: instead of estimating an endogenous return distribution or optimizing risk functionals, we aim to \emph{steer} the distribution of the accumulated reward random variable toward a prescribed target law by directly optimizing the policy. Whereas quantile-based DRL methods rely on Bellman recursions, additive rewards, and the value-function framework, our approach operates on a lifted state space of probability measures and optimizes a characteristic-function divergence akin to a policy iteration method. We show in Example \ref{se:forzen_lake} that our framework is general and flexible enough that it can accommodate solutions to some classical RL problems with minimal modifications. 

Let us also mention that recent work has extended DRL to the continuous-time setting by characterizing the evolution of return distributions through a distributional analogue of the backward Kolmogorov equation. In \cite{wiltzer2022distributional} the authors derived what they call the distributional HJB equation, but which is in fact a linear backward Kolmogorov PDE for the state-conditioned return CDF under a fixed policy, thereby providing the first continuous-time description of return distributions beyond expectations. To make this infinite-dimensional PDE numerically tractable, they introduced the statistical HJB (SHJB) loss, which replaces the full distribution by a finite set of statistics (e.g., quantiles) and measures the discrepancy between the approximate CDF and the PDE operator. However, that work did not establish whether minimizing the SHJB loss actually produces consistent approximations of the true return distribution. This gap is filled by \cite{alhosh2025tractablerepresentationsconvergentapproximation}, who prove that if the chosen imputation strategy (mapping finite statistics to a distribution) satisfies a mild topological condition, namely convergence in the sense of tempered distributions then SHJB minimization is provably convergent: as the number of statistics increases, the approximate solutions converge to the true solution of the Kolmogorov PDE. Moreover, they show that the widely used quantile representation satisfies this condition and achieves an $\mathcal{O}(\frac{1}{N})$ approximation error for the CDF, thereby giving the first rigorous justification that quantile-based continuous-time DRL is a sound and convergent discretization scheme for the distributional Kolmogorov equation.

Finally, we mention two works that connect stochastic control with optimal transport and are  close in spirit to ours but methodologically different. In \cite{Alouadi2025SBTS} the authors study time series generation through the Schr\"odinger Bridge (SB) framework, which formulates generative modeling as a stochastic control problem on path space: one selects a drift process minimizing a quadratic control (relative-entropy) cost with respect to a Brownian reference measure, subject to matching prescribed marginal distributions of the state process at given time points. The main conceptual difference lies in the location of the distributional objective: while the SB approach enforces distributional constraints directly on the state process itself, we consider controlled dynamics augmented with a cumulative reward variable and impose the distribution-matching objective on this derived random quantity rather than on the state trajectory. In this sense, rewards act as an additional layer on top of the dynamics, transforming path distributions into outcome distributions. Unlike the SB formulation, we are not aware of any optimal transport or entropic regularization interpretation on path space for our loss. Nevertheless, \cite{Alouadi2025SBTS} is very interesting in conjunction with \cite{wiltzer2022distributional} for a general bridge interpretation. 

In \cite{TerpinLanzettiDorflers2024DPOT} the authors consider fleets of identical agents whose dynamics are non-interacting and each agent evolves independently according to a common controlled transition map, and the population state evolves as the pushforward of the current distribution through this single-agent dynamics. Any coupling between agents arises exclusively through the objective function, via optimal transport discrepancies that compare the induced population distribution to a prescribed reference measure. As a result, the separation principle identified in that work critically relies on this absence of interaction in the dynamics. Here, we emphasize again that we operate in a lifted distributional MDP and rely on a direct policy method. However, we find the connection with optimal transport very interesting at least in the case of value iteration.

\subsection{Summary and Organization of the Article}

The remainder of the article is organized as follows. In Section~\ref{sec:Problem+Ass}, we formally introduce the distributional control problem, define the characteristic function matching objective, and reformulate it as a lifted distributional Markov Decision Process on an augmented state space. We also specify the class of randomized Markov policies considered and state the standing assumptions on the dynamics, rewards, and neural-network parametrization. Section 3 presents the policy-gradient algorithm for distribution matching, including the discretization of the Fourier domain, the construction of an unbiased (up to finite-sample effects) stochastic gradient estimator, and the full learning procedure. In Section 4, we establish our main theoretical result, proving convergence of the proposed algorithm to a stationary point under standard stochastic approximation conditions. Section~\ref{sec:appl} illustrates the methodology on a series of numerical examples, namely linear quadratic control, an investment problem, a compactly supported distribution and classical benchmark MDPs, highlighting both the flexibility and the intrinsic non-uniqueness of distributional optimal controls. Technical proofs and auxiliary results are collected in the Appendix.

\section{The Problem and Assumptions}\label{sec:Problem+Ass}
\subsection{Problem Formulation}
The main purpose of this section is to give a formal definition of our general problem. To this end, we assume a finite time horizon $T\in\mathbb{N}$ and let us denote the state space and action space  by $\mathcal{S}$ and $\mathcal{A}$, respectively. We assume that both are Borel subsets of $\RR$.   Let $(\Omega, \mathcal{F}, \mathbb{P})$ be a complete probability space and $(\varepsilon_t)_{t=0}^{T-1}$ a sequence of i.i.d. standard normal random variables defined on this space. For ease of notation we consider a stationary model here, but the extension to non-stationary systems is immediate.
The system evolves according to the stochastic dynamics:
\begin{equation}
\label{eq:def:nonlinear_dynamics}
    s_{t+1} = F(s_t, a_t, \varepsilon_{t+1}), \quad \varepsilon_{t+1} \sim \mathcal{N}(0,1),
\end{equation}
with the initial state $s_0$ either deterministic or $\mathcal{F}_0$-measurable. Furthermore, $F:\cS \times \cA \times \RR \to \cS$ is a measurable transition function and $a_{t}\in\mathcal A$ is the action chosen at time point $t$. We assume that the decision maker can choose   a  policy $\upsigma=(\upsigma_0,\ldots,\upsigma_{T-1})$ where $\upsigma_t$ is the randomized, history-dependent decision rule at time $t$, i.e. $a_t\sim \upsigma_t(\cdot|s_0,a_0,\ldots,s_t)$.  

Assume a measurable one-stage reward $r:\cS \times \cA \to \RR$. The cumulative reward for a trajectory up to time $t$ is:   

\begin{equation}
    R_t = \sum_{k=0}^{t-1} r(s_k, a_k). 
\end{equation}

For this work, we regard $R_T$ as a random variable and we are interested in its distribution. The distribution depends on the chosen policy $\upsigma$, hence we denote it by $\text{Law}^\upsigma(R_T)$. 

When we are given a target distribution $\mu^\star$ on $\RR$, our goal is to find a policy $\upsigma$  such that 
\begin{equation}
\label{eq:matching equation}
    \text{Law}^\upsigma(R_T) \approx \mu^\star.
\end{equation}

To formalize the matching, we introduce a divergence between probability distributions: the weighted squared $L^2$ distance between the characteristic functions:
\begin{equation}
\label{eq:def: general Loss} 
    \mathcal{L}(\upsigma) := \int_{\mathbb{R}} \left| \varphi^\star(u) - \varphi_{\upsigma}(u) \right|^2 w(u) \, du,
\end{equation}
where $\varphi_{\upsigma}(u) := \mathbb{E}^\upsigma[e^{iu R_T}]$ is the characteristic function of the real-valued random variable $R_T$ when policy $\upsigma$ is chosen, $\varphi^\star$ is the characteristic function of $\mu^*$ and $w: \mathbb{R} \to \mathbb{R}_+$ is a suitable weighting function satisfying the weight moment assumptions 
\begin{equation*}
  \int_\RR  w(u) du < \infty, \quad   \int_\RR |u| \cdot w(u) du < \infty, \quad  \int_\RR u^2 \cdot w(u) du < \infty.
\end{equation*}
Distances like these are popular in statistics for testing distributional hypotheses, \cite{epps1983test,baringhaus1988consistent}.
The problem becomes:
\begin{equation}
\label{eq:general learning problem}
    \min_{\upsigma} \mathcal{L}(\upsigma).
\end{equation}

\begin{remark}
	Equation \eqref{eq:matching equation} should be understood in the sense of the closest controlled distribution $\text{Law}^\upsigma(R_T)$ to $\mu^*$. Of course, when $\mu^*$ belongs to the reachable set of the transition operator $T$ of the lifted state $(s_t,R_t)$ (for details see the following section) then the matching can be exact. We stress that the problem of minimizing the distance can be solved whether we have an exact matching or an approximate one. In the second case we have a positive distance as a solution of the minimization problem. 
\end{remark}

Now, we reformulate Problem \eqref{eq:general learning problem} according to \cite{bauerle2025distributionalbellmanequation} as a \emph{lifted MDP} in order to distinguish classes of controls among which we will look for optimal ones. 

 \subsection{Lifted Distributional MDP Formulation} \label{sec:liftedMDP}
 We now recast the optimization problem \eqref{eq:general learning problem} as a Markov Decision Process on an \emph{augmented state space} that captures the distribution of outcomes. 

  Before we continue with the formal description of the distributional MDP, let us explain the main idea with the simple torus example from the introduction. According to \eqref{eq:torusMDP} we may write the terminal state as
 \begin{align}
     s_T &= (\ldots(((s_0+a_0+\varepsilon_1) \bmod 2\pi + a_1+\varepsilon_2) \bmod 2\pi +a_2+\varepsilon_3\bmod 2\pi\ldots \bmod 2\pi)
 \end{align}
 where $(\varepsilon_1,\ldots,\varepsilon_T)$ are i.i.d.\ and the decision variable satisfies $a_t \sim \sigma_t(\cdot | s_0,a_0,s_1,a_1,\ldots ,s_t).$ Thus, $a_t$ depends in particular on $(\varepsilon_1,\ldots,\varepsilon_{t-1}).$ The aim is to choose the laws $\sigma_0,\ldots,\sigma_{T-1}$ such that
 \begin{align}\label{prob:torus1}
     \inf_\sigma \int_{\mathbb{R}} \left| \EE^{\upsigma}[e^{ius_T}] - \varphi^*(u) \right|^2 w(u) \, du,
 \end{align}
 is minimized for a target characteristic function $\varphi^*.$ But this means we have to keep track of the complete history of the process. We can avoid this  by working on the state space of distributions. That is, we instead consider the distribution function $F_t$ of $s_t$ where the sequence $(F_t)$ satisfies the following recursion
 \begin{align*}
     F_{t+1}(x) &= \PP(s_{t+1}\le x) = \PP( (s_t+a_t+\varepsilon_{t+1})\bmod 2\pi \le x) \\
     &= \int\int\int \PP_\varepsilon( (s+a+\varepsilon)\bmod 2\pi \le x) \pi_t(da|s) F_t(ds) F_\varepsilon(d\varepsilon).
 \end{align*}
 Note that the decision here is to choose $\pi_t(\cdot|s)$ which is based on the knowledge of $F_t.$ The objective can now equivalently be written as
 \begin{align}
     \inf_\pi \int_{\mathbb{R}} \left| \int e^{ius}F_T(ds) - \varphi^*(u) \right|^2 w(u) \, du.
 \end{align}
 This has  been shown in \cite{bauerle2025distributionalbellmanequation}  and has the advantage that the history of the process is already encoded in its distribution. Further, this approach allows for a dynamic programming principle even if we consider the more general case of the distribution of $R_T.$

To formulate the general \emph{distributional MDP}, we retain the original notation and introduce an \emph{augmented state variable} $(s_t, R_t)$ taking values in the product space $\mathcal{S} \times \mathcal{R}$, where $\mathcal{R} \subseteq \mathbb{R}$ is the set of possible cumulative rewards. The law of $(s_t, R_t)$ under a policy $\upsigma$ (and fixed initial distribution) is denoted $\mathsf F_t^\upsigma \in \mathcal{P}(\mathcal{S} \times \mathcal{R})$, i.e., the joint distribution of $(s_t, R_t)$. In particular, $\mathsf F_0 = \mathcal L(s_0) \otimes \delta_{ 0}$ is the initial distribution of $(s_0, R_0)$. The terminal law $\mathsf F_T^\upsigma$ encodes the distribution of both the terminal state $s_T$ and the cumulative reward $R_T$. Note that the marginal of $\mathsf F_T^\upsigma$ on the $R$-coordinate is precisely $\text{Law}(R_T)$. 

We can now define a \emph{lifted distributional MDP} whose state at each time is the joint distribution $\mathsf F_t^\upsigma$ of  $(s_t, R_t)$. The lifted model is specified as follows:

\begin{enumerate}
\item \textbf{State space:} $\mathcal{X} := \mathcal{P}(\mathcal{S} \times \mathcal{R})$. An element $\mathsf F \in \mathcal{X}$ is a joint distribution of the state and accumulated reward. At time $t$, the lifted state is $\mathsf F_t = \text{Law}^\sigma(s_t, R_t)$ under the current policy $\sigma$. This lifted state encodes all the information necessary to describe the process.

\item \textbf{Action space:} $\Pi^M := \{\pi: \mathcal{S} \times \mathcal{R} \to \mathcal{P}(\mathcal{A})\}$, the set of Markov (possibly randomized) decision rules on the augmented state. That is, $\pi \in \Pi^M$ is a conditional distribution $\pi(da|s, R)$ prescribing an action distribution given state $s$ and cumulative reward $R$. A lifted policy is a sequence $\pi_0, \pi_1, \ldots, \pi_{T-1}$ with $\pi_t \in \Pi^M$.

\item \textbf{One-stage reward:} $0$ for each $t < T$. Intermediate transitions are assigned zero reward, as only the terminal outcome matters.

\item \textbf{Terminal reward functional:} $H: \mathcal{P}(\mathcal{S} \times \mathcal{R}) \to \mathbb{R}$ is defined by
\[
H(\mathsf F) := \int_\RR \Big|\varphi^*(u)- \int e^{iux} \mathsf F(\mathcal{S},dx)\Big|^2 w(u)du.
\]

\item \textbf{State-transition operator:} $T: \mathcal{P}(\mathcal{S} \times \mathcal{R}) \times \Pi^M \to \mathcal{P}(\mathcal{S} \times \mathcal{R})$ describes the evolution of the distribution under a policy. Given $\mathsf F \in \mathcal{P}(\mathcal{S} \times \mathcal{R})$ and $\pi \in \Pi^M$, the next distribution $\mathsf F' = T^{\pi}(\mathsf F)$ is the law of $(s_{t+1}, R_{t+1})$ where:
\[
\begin{aligned}
(s_t, R_t) &\sim \mathsf F, \\
a_t &\sim \pi(\cdot | s_t, R_t), \\
s_{t+1} &= F(s_t, a_t, \varepsilon_{t+1}), \\
R_{t+1} &= R_t + r(s_t, a_t).
\end{aligned}
\]
Formally, for any measurable set $B \subseteq \mathcal{S} \times \mathcal{R}$,
\[
T^{\pi}(\mathsf F)(B) = \int_{\mathcal{S} \times \mathcal{R}} \int_{\mathcal{A}} \mathbb{P}\left\{ (F(s,a,\varepsilon),\; R + r(s,a)) \in B \right\} \pi(da|s,R)\,F(ds,dR).
\]
\end{enumerate}

This defines a finite-horizon deterministic dynamic program on the space of probability measures, with state variable $\mathsf F_t \in \mathcal{P}(\mathcal{S} \times \mathcal{R})$. Any policy $\upsigma$ in the original model induces a sequence $\mathsf F_0 \xrightarrow{\sigma_0} \mathsf F_1 \xrightarrow{\sigma_1} \dots \xrightarrow{\sigma_{T-1}} \mathsf F_T$ with terminal objective $H(F_T)$ and similarly any sequence $(\pi_0, \ldots, \pi_{T-1})$ of kernels in the lifted MDP defines a feasible path $\mathsf F_0 \xrightarrow{\pi_0} \mathsf F_1 \xrightarrow{\pi_1} \cdots \xrightarrow{\pi_{T-1}} \mathsf F_T$. It is possible to construct the sequence  $(\pi_0, \ldots, \pi_{T-1})$ such that distribution sequences coincide (see \cite{bauerle2025distributionalbellmanequation}, Prop. 2.2).
Hence, the original problem \eqref{eq:general learning problem} is equivalent to:
\begin{equation} \label{eq:equivalent_value_function}
\min_{(\pi_0,\ldots,\pi_{T-1})} H\left(T^{\pi_{T-1}} \circ \cdots \circ T^{\pi_0}(\mathsf F_0)\right).
\end{equation}

For completeness, we define the value functional $V_t: \mathcal{P}(\mathcal{S} \times \mathcal{R}) \to \mathbb{R}$ by:
\begin{equation}
	\label{eq:value_function}
V_T(\mathsf F) := H(\mathsf F), \quad
V_t(\mathsf F) := \inf_{\pi \in \Pi^M} V_{t+1}\left(T^{\pi}(\mathsf F)\right), \quad t = T-1, \dots, 0.	
\end{equation}

Then $V_0(\mathsf F_0)$ gives the optimal value of the lifted problem and equals the value in \eqref{eq:general learning problem}.
Finally, a key consequence of this formulation is that the search for an optimal policy may be restricted \emph{without loss of generality} to actions $\pi_t \in \Pi^M$ on  $(s_t, R_t)$. Any history-dependent policy can be equivalently represented as a sequence $(\pi_0, \dots, \pi_{T-1})$ producing the same law of outcomes, see \cite[Remark 2.2]{bauerle2025distributionalbellmanequation}.  
However, due to the complicated state and action space in the lifted MDP it is computationally very hard to perform the value iteration algorithm \eqref{eq:value_function}. In order to circumvent this problem we define a parametrized policy class and use gradient descent to obtain the optimal policy.

\subsection{Controls}
Now, in contrast to value iteration methods based on $V_T(\mathsf F)$ we choose to use $H(\mathsf F)$ as a target function for a parametric policy iteration in line with Motivation \ref{se:motivation}. As we already explained, the lack of analytical traceability for the distribution of the actions calls for function approximation. 
We will choose artificial neural nets (ANNs) to approximate the decision rules. Their use in our work is very close in the spirit to \cite{Goodfellow2014GAN}, in particular, we learn an optimal sampling mechanism from the unknown distribution of controls, while making no assumption on their distribution. We stress that our method is likelihood free. Closely related approaches  with  direct parameterization are those of \cite{han2016deep,hure2021deep}. In our work, we make use of Lipschitz assumptions and properties of the MDP to prove convergence of the algorithm. This has also been done in \cite{pirotta2015policy} where, however, the authors parametrize the probability law of the action. In \cite{fatkhullin2023stochasticpolicygradientmethods}, the authors also consider  an implicit method but assuming a density.

More precisely,  we further assume that on our  probability space $(\Omega, \mathcal{F},\mathbb{P})$ there exists another  sequence of i.i.d.\  standard normal random variables $(z_t)$ which is independent of $(\varepsilon_t)$.
We define the action $a_t$ at time $t$, sampled from the Markovian policy $\pi_t(\cdot|s_t,R_t)$  as 
\begin{equation}
    a_t = f(\theta,s_t,R_t,z_t,t), 
\end{equation}
where $f: \Theta\times \mathcal{S}\times \mathcal R \times \mathbb{R} \times \{0,1,\ldots,T-1\}\to \mathcal{A}$ is a neural network parameterized by $\theta\in\Theta$. Note that we need $z_t$ to generate the randomness. W.l.o.g.\ we could also use a random variable uniformly distributed over $(0,1)$ although other choices are also possible. 

To summarize we have the following evolution for $t=0,1,\ldots ,T-1$ of the parametrized system where $s_0$ is a given state and $R_0=0$:
\begin{eqnarray*}
a_t^\theta &=& f(\theta,s_t^\theta,R_t^\theta,z_t,t),\\
s_{t+1}^\theta &=& F(s_t^\theta,a_t^\theta, \varepsilon_{t+1}),
\end{eqnarray*}
and for $t=1,\ldots ,T$
$$R_t^\theta = \sum_{k=0}^{t-1} r(s_k^\theta,a_k^\theta).$$ 

In general, we consider a feedforward neural network with $K$ hidden layers of width $P$, taking as input $(s,R,z,t) \in \mathcal{S}\times \mathcal R \times \mathbb{R} \times \{0,1,\ldots,T-1\}$,\footnote{In practice we can either take one "big" neural network including time as a separate input variable or have $T$ different neural networks each with a set of parameters $\theta_t$; the two specifications give identical results} applying the activation function $\sigma$ component-wise and producing a scalar output.
While our results hold for arbitrary $K, P$, for the sake of readability we assume a two-layer neural network  

\begin{equation}
	\label{eq:definition_NN_formula}
	f(\theta,s,R,z,t) =  w^2\sigma(w^1_s s + w^1_R R+ w^1_z z + w^1_t t+ b^1)+b^2. 
\end{equation}
where $\theta=(w^1_s,w_R^1,w_z^1,w_t^1,w^2,b^1,b^2)\in \RR^7$  and $\sigma:\RR\to\RR.$ For convenience we will write
\begin{equation}
	\label{eq:definition_NN_formula2}
	f(\theta,x) =  w^2\sigma(W^1 x + b^1)+b^2 
\end{equation}
for $W^1=(w^1_s,w_R^1,w_z^1,w_t^1)\in \RR^4$ and $x=(s,R,z,t)\in \mathcal{S}\times \mathcal R \times \mathbb{R} \times \{0,1,\ldots,T-1\}.$
 The convergence result for the algorithm proved below also holds for multi-layer networks. We only need to adjust the constants to transfer the regularity via the different layers of the ANN.

\subsection{Assumptions and Some Implications}
Throughout we need several assumptions on the model data such that we can prove convergence of our algorithm.
\subsubsection{Assumptions on the neural net data}
We assume that the parameters $\theta$ of the neural net are taken from a compact set $\Theta\subset \RR^7$ and that the activation function $\sigma$ shows a high regularity in the sense that the following assumption holds.

\begin{assumption}
	\label{as:1} The parameter set $\Theta$ is compact and convex and $\sigma \in C^{1,1}(\RR)$, i.e.\  $\sigma$ is continuously differentiable and $\sigma'$ is Lipschitz-continuous and bounded. 
\end{assumption}

Note that for example the activation functions $\sigma(x)=\tanh(x)$ or Softplus satisfy the requirements. 

\subsubsection{Assumptions on the growth of the state process}
We do not want to bound the state process because this would rule out a number of interesting applications. Instead we assume that the Markov Decision Process possesses a bounding function, a concept which is often used to bound the value function (see e.g. \cite{BauerleRieder}, Sec. 7.3).

\begin{assumption}
	\label{as:2} There exists a linear bounding function  for the MDP, i.e.\ there exist $c_b,d_b\in\RR_+$ such that for all $s\in \mathcal{S}$, $\varepsilon \sim \cN(0,1)$ and $m(s,a):=\EE\left[F(s,a,\varepsilon)|s,a \right]$ we have
 $$|m(s,a)| \le c_b |s| + d_b \mbox{ for all } a\in \mathcal{A}$$
\end{assumption}
The assumption of a bounding function will be crucial in the definition of the good event later.

In the next subsection we state a number of Lipschitz continuity properties of our data. 

\subsubsection{Further Lipschitz Continuity Assumptions on the Data}
For the convergence of the algorithm we have to impose further Lipschitz continuity properties of the data.

\begin{assumption}
	\label{as:3} The transition function $F:\cS \times \cA \times \RR \to \cS$ is Lipschitz-continuous in the third component, uniformly in $s,a,$ i.e. there exists a constant $L_F>0$ such that
    $$ |F(s,a,\varepsilon)-F(s,a,\tilde\varepsilon)| \le L_F |\varepsilon-\tilde\varepsilon|, \quad \mbox{for all } s\in \cS,  a\in \cA, \varepsilon, \tilde\varepsilon \in\RR.$$
\end{assumption}

\begin{assumption} \label{ass:rFLip}
We assume
\begin{itemize}
\item[(i)]$r$ is Lipschitz continuous, i.e. for all $s,\tilde s\in \cS,a,\tilde a \in \cA $ there exists a $L_r>0$ such that
$$|r(s,a)-r(\tilde s,\tilde a)| \le L_{r}\big( |s-\tilde s| + |a-\tilde a| \big) $$
\item[(ii)] $F$ is  Lipschitz continuous, i.e. for all $s,\tilde s\in \cS,a, \tilde a \in \cA $ and $\varepsilon \in\RR$ there exists a $\tilde L_F>0$ such that
$$|F(s,a,\varepsilon)-F(\tilde s,\tilde a,\varepsilon)| \le \tilde L_{F} \big( |s-\tilde s| + |a-\tilde a|\big). $$
\end{itemize}
\end{assumption}

We denote by $\partial_s r, \partial_a r$ and $\partial_s  F, \partial_a F $ the derivative w.r.t.\ the corresponding component and assume the following:

\begin{assumption} \label{ass:rFLipgrad}
We assume
\begin{itemize}
\item[(i)]$\partial_s r, \partial_a r$ exist and are  Lipschitz continuous, i.e. for all $s,\tilde s\in \cS,a, \tilde a \in \cA $ and for $i=s,a$ there exist $L_{\nabla}^{r,i} >0$ such that
$$|\partial_i r(s,a)-\partial_ir(\tilde s,\tilde a)| \le L_{\nabla}^{r,i} \big( |s-\tilde s| + |a-\tilde a|  \big)$$
\item[(ii)] $\partial_s  F, \partial_a F$ exist and are Lipschitz continuous, i.e. for all $s,\tilde s\in \cS,a, \tilde a \in \cA,\varepsilon\in\RR $ for $ i=s,a$ there exist $L_{\nabla}^{F,i} >0$ such that
$$|\partial_iF(s,a,\varepsilon)-\partial_i F(\tilde s,\tilde a,\varepsilon)| \le L_{\nabla}^{F,i} \big( |s-\tilde s| + |a-\tilde a|\big). $$
\end{itemize}
\end{assumption}

\begin{remark}
    First note that if the state and action space are compact, Assumption \ref{as:2} would be satisfied. Further, in this case the condition that reward $r$ and transition function $F$
 are continuously differentiable would directly imply Assumption \ref{ass:rFLip}.
 \end{remark}

\subsubsection{Bounds on the evolution of the process}
Under the assumptions above, we can get bounds on the state evolution. This is important for the Lipschitz properties that we prove later.
 Define the `good event' $G\subset \Omega$ by setting
\begin{align*}
    G := \bigcap_{t=1}^T \{ |\varepsilon_t| \le B_\varepsilon\} \cap  \bigcap_{t=0}^{T-1} \{ |z_t| \le B_z\} 
\end{align*}
and choose the constants $B_\varepsilon$ and $B_z$ such that $\PP(G) \ge 1-\gamma$ for a given small $\gamma>0.$ On $G$ we derive the following explicit bound on the state process:
Indeed, on $G$ we obtain with Assumption \ref{as:2} and the Lipschitz properties of $F$ when we define $\hat L_F :=  L_F (B_\varepsilon+\EE[|\varepsilon_t|])$:
\begin{eqnarray} \nonumber
    |s_t| &=& |s_t-m(s_{t-1},a_{t-1})+m(s_{t-1},a_{t-1})| \\  \nonumber
    &\le& \EE\Big[|F(s_{t-1},a_{t-1},\varepsilon_t)- F(s_{t-1},a_{t-1},\varepsilon)| \Big| s_{t-1},a_{t-1},\varepsilon_t\Big]+ |m(s_{t-1},a_{t-1})| \\\nonumber &\le &
    L_F (B_\varepsilon+\EE[|\varepsilon|])+ |m(s_{t-1},a_{t-1})| \\ 
   &\le&\hat L_F + c_b|s_{t-1}| +d_b\le \ldots \le 
  c_b^t |s_0|+  (\hat L_F+d_b)\sum_{k=1}^{t}c_b^{t-k} =: B^s_t. \label{eq:s_bound} 
\end{eqnarray}

 Instead of assuming a compact state  space we work with the good event $G$. As long as $\omega\in G$ we have a bounded state and action process  and thus also bounded rewards. 

\section{Gradient Descent Algorithm}
In order to define a gradient descent scheme for the parameters of the ANN, we first reformulate the objective \eqref{eq:general learning problem} given our ANN parameterization as  

\begin{equation}
	\label{eq: parametrised Loss}
	\min_{\theta \in \RR^7} \cL(\theta) = \min_{\theta \in \RR^7} \int_{\mathbb{R}} \left| \varphi^\star(u) - \EE[e^{iu R_T^\theta}]  \right|^2 w(u) \, du.
\end{equation}
For the gradient descent we need the gradient of $\cL$ with respect to the parameters of the ANN. It can be computed by a pathwise derivative:
\begin{align}
\label{eq:def:Gradient_Loss}
   \nabla_\theta \mathcal{L}(\theta) =
   2 \int_{\mathbb{R}} \operatorname{Re} \left[ \overline{\left( \EE[e^{iu R_T^\theta}] -\varphi^\star(u)  \right)}\cdot \mathbb{E}\left[i u e^{i u R_T^\theta} \cdot \nabla_\theta R_T^\theta \right] \right] w(u) du
\end{align}

In order to proceed we use three approximations for the gradient. 

\subsection{Conditioning on the Good Event}
We first condition the expectation on the good event $G,$ i.e.\ instead of $\EE[e^{iu R_T^\theta}]$ etc.\ we consider
$$ \varphi_\theta(u):=\EE_G[e^{iu R_T^\theta}]:= \EE[e^{iu R_T^\theta}|G].$$

We can control the error introduced by this conditioning irrespective of $\theta.$ We discuss this in more detail in the next remark.

\begin{remark}[Error induced by conditioning on $G$]
Suppose $G_n$ is an increasing sequence of subsets of $\Omega$ such that $G_n \uparrow \Omega$ and $\PP(G_n)=1-\gamma_n$ where necessarily $\gamma_n\downarrow 0.$ W.l.o.g.\ we assume $\gamma_0=1/2$. Denote $v_n(\theta) :=  \int_{\mathbb{R}} \left| \varphi^\star(u) - \EE_{G_n}[e^{iu R_T^\theta}]  \right|^2 w(u) \, du$ and  note that  $|e^{iu R_T^\theta}| \le 1$. In Lemma  \ref{lem:LR} we show that $\theta\mapsto R_T^\theta$ is continuous. Thus,  by the dominated convergence theorem $\theta\mapsto  v_n(\theta)$ is also continuous. Hence minimizers $\theta_n^*$ of $v_n$ exist for all $n\in\NN$ because $\Theta$ is compact. Moreover, we obtain for $n\ge m$ and $\theta\in \Theta$:
\begin{eqnarray*}
    |v_n(\theta)-v_m(\theta)| &\le & 4  \int \Big| \EE_{G_n}[e^{iu R_T^\theta}] -  \EE_{G_m}[e^{iu R_T^\theta}]  \Big|  w(u) du \\
    &\le &  16 \gamma_m \int_\RR  w(u) du < \infty
\end{eqnarray*}
  The right-hand side obviously does not depend on $\theta$ and converges to zero for $m\to\infty.$ Thus, Theorem A.1.5 in \cite{BauerleRieder} implies that any sequence of minimizers $(\theta_n^*)$ has converging subsequences whose limits are optimal for $v(\theta)= \cL(\theta).$  
\end{remark}

\subsection{Discretization of the Integral}
\label{subse:discretization of the integral}
In order to implement the algorithm we have to discretize the integral, unless we are in a special case which we discuss in Remark \ref{rem:Epps} below. We fix $K\in(0,\infty)$ and a quadrature rule on $[-K,K]$
given by nodes $(u_\ell)_{\ell=1}^L$ and nonnegative weights
$( \beta_\ell)_{\ell=1}^L$ such that, for integrable $\mathrm h$,
\[
\int_{-K}^{K} \mathrm h(u)\,w(u)\,du
\;\approx\;
\sum_{\ell=1}^L  \beta_\ell\,\mathrm h(u_\ell).
\]
Here $w(u)$ is the weight from the definition of $\cL$ and $ {\beta}_\ell  $ the  weights. $K$ should be chosen large enough  that the omitted tail integral term is small. Thus, we consider
\begin{equation}\label{eq:LL}
  \cL_{L}(\theta)
:=\sum_{\ell=1}^{L} \beta_\ell\,\big|\varphi^\star(u_\ell)- \varphi_\theta(u_\ell) ]\big|^2.  
\end{equation}

\subsection{Grid-based Stochastic Gradient}
\label{subse:grid-based stochastic gradient}
Finally, by coupling $(\varepsilon_{t+1},z_t)$ we draw $M\in\NN$ i.i.d. samples   of 
$R_T^{\theta,1},\ldots,R_T^{\theta,M}$ on the set $ G$ and define the estimator for the gradient $\nabla_\theta \mathcal{L}_L(\theta)$ as follows. First define the empirical characteristic function and its derivative by
$$ \hat\varphi_\theta(u) = \frac1M \sum_{j=1}^M e^{iu R_T^{\theta,j} }, \quad \widehat{\nabla_\theta \varphi_{\theta}}(u)= \frac1M \sum_{j=1}^M i u e^{iu R_T^{\theta,j}} \cdot \nabla_\theta R_T^{\theta,j}  $$
 Thus, in the end we consider
 \[
\hat\cL_{L}(\theta)
:=\sum_{\ell=1}^{L} \beta_\ell\,\big|\varphi^\star(u_\ell)-\hat\varphi_\theta(u_\ell)\big|^2.
\]
with gradient
 \begin{equation}\label{eq:approxgradient}
   g(\theta)
:= \nabla_\theta \hat\cL_{L}(\theta)= 2\sum_{\ell=1}^{L} \beta_\ell
\mathrm{Re}\Big(
 \overline{(\hat\varphi_\theta(u_\ell)-\varphi^*(u_\ell))} \cdot \widehat{\nabla_\theta \varphi_{\theta}}(u_\ell)
\Big).  
 \end{equation}

\begin{remark}\label{rem:Epps}
A particularly convenient  case appears when we choose a standard normal as target and weight function. That is we have
$$ \varphi^*(u) = e^{-1/2 u^2}, \quad w(u)= 1/\sqrt{2\pi} e^{-1/2 u^2}, \quad u\in\RR.$$
In this case we can skip the  procedure of integral approximation  since the empirical loss function has a very explicit representation. More precisely
the empirical target function has the following explicit representation (see e.g. \cite{epps1983test,baringhaus1988consistent}):
\begin{eqnarray*}
   \int_{\mathbb{R}} \left|      \varphi^\star(u)\!-\!\hat\varphi_\theta(u) \right|^2 w(u) \, du
= \frac{1}{M^2} \!\sum_{j,k=1}^M \! \exp(-\frac12 |R_T^{\theta,j}\!-\!R_T^{\theta,k}|^2) - \frac{\sqrt{2} }M  \sum_{j=1}^M \exp(-\frac14 |R_T^{\theta,j}|^2) \!+\! \frac{1}{\sqrt{3}}.
\end{eqnarray*} 
\end{remark}

\subsection{Algorithm}
We conclude this section with the complete algorithm for our learning scheme. We want to stress again that the policy defines a sampling procedure, making it a likelihood-free generative model. The samples have to be truncated to stay in $G.$

\begin{algorithm}[H]
\DontPrintSemicolon
\caption{Policy Gradient for distribution matching}
\label{alg:theoretical}

\KwIn{Target cumulative rewards  (or target law); policy $a_t=f(\theta,s_t,R_t,z_t,t)$; horizon $T$; learning rate $(\alpha_k)_{k=0}^\infty$;  frequency grid $(u_\ell)_{\ell=1}^{L}$ with weights $\beta_\ell$; batch size $M$.}
\KwOut{Trained parameters $\theta$.}

Compute target empirical CF, $\varphi^*$  from target data or given distribution.\\

Initialize $\theta_0$\;

\For{$k=0,1,2,\dots$}{
  
  \For{$\ell=1$ \KwTo $L$}{
  Simulate $M$ \emph{independent} trajectories for this frequency\\
    \For{$j=1$ \KwTo $M$}{
      $s_0^{\theta_k,j} \leftarrow s_0$, \ $R_0^{\theta_k,j} \leftarrow 0$\;
      \For{$t=0$ \KwTo $T-1$}{
        Sample $z_t^{j} \sim \mathcal{N}(0,1)$, \mbox{truncate},
        $a_t^{\theta_k,j} \leftarrow f\big(\theta_k,s_t^{\theta_k,j},R_t^{\theta_k,j},z_t^{j},t\big)$\;
        Sample $\varepsilon_{t+1}^{j} \sim \mathcal{N}(0,1)$, \mbox{truncate},
        $s_{t+1}^{\theta_k,j} \leftarrow F\!\big(s_t^{\theta_k,j},a_t^{\theta_k,j},\varepsilon_{t+1}^{j}\big)$\;
        $R_{t+1}^{\theta_k,j} \leftarrow R_t^{\theta_k,j} + r\!\big(s_t^{\theta_k,j},a_t^{\theta_k,j}\big)$\;
      }
    }
    
    $\displaystyle \widehat{\varphi}_{\theta_k}(u_\ell) \leftarrow \frac{1}{M}\sum_{j=1}^M e^{\,i\,u_\ell R^{\theta_k,j}_T}$\;

  }
Compute loss:
  \[
  \hat\cL_{L}(\theta_k) = \sum_{\ell=1}^{L} \beta_\ell\,
  \big| \varphi^{*}(u_\ell)-\hat{\varphi}_{\theta_k}(u_\ell) \big|^2
  \]
  Update parameters: $\theta_{k+1} \leftarrow \theta_k - \alpha_k \, g(\theta_k)$\;
  \If{$\hat \cL_L(\theta_k) < \text{threshold} $}{\textbf{break}}
}

\Return $\theta_{k+1}$\;
\end{algorithm}

\begin{remark}
	While details about the specific choice of the input parameters will be given in Section \ref{sec:appl}, we want to comment that alternatives are possible for the implementation and our goal here is not to  test all of them exhaustively. First, the outer loop over all nodes $u_\ell$ can be avoided using the same noise per node, making the computation cheaper without affecting the bias of the estimator. Second, the variance of the estimator depends on the choice of the weighting function $w(\cdot)$ and the quadrature weights $\beta_\ell$ as it will become apparent in the following section. Of course, we can, in addition, use other techniques to couple the noises across iterations to reduce the variance.  Third, we assume that the parameter update stays in the set $\Theta.$
\end{remark}

The main result of our paper is now the convergence of the algorithm toward a stationary point. 

\section{Convergence of the Algorithm }
In this section we state and prove our main convergence theorem. The problem \eqref{eq: parametrised Loss} is highly non-convex with respect to parameter $\theta$, and thus the best we can hope is convergence to a stationary point. We want also to emphasize that even if the loss was strictly convex this would give a unique global minimum for the loss but no further information on the set of optimal controls. The optimal control does not need to be unique as  already noted in \cite[Chapter 7.4, 7.5]{Bellemare2023DistributionalRL} and we will observe this also in Section \ref{sec:appl}. In order to address questions of uniqueness of optimal controls or even define a notion of uniqueness adapted to our needs we should investigate further the value function as defined in \eqref{eq:value_function}, in the spirit of  \cite{bauerle2025distributionalbellmanequation}.\footnote{We leave this for future work.} 

Our proof follows the lines of the Robbins-Monro stochastic approximation framework, see for instance \cite{Borkar2025StochasticApproxRL} and \cite[Chapter 5 and Exercise 5.30]{Bach2024LearningTheory}. The main difficulty is to prove the Lipschitz continuity of the gradient of \eqref{eq:LL}. To this end, we need to work in steps, transferring the regularity from our control problem to the characteristic functions and finally to the gradients. One main difficulty is that the noise, in both control and dynamics (independently) can accumulate in the system and cause divergence. To account for this case, we need to control the noise with the help of the \emph{good event} on which  the convergence happens.       

We first prove the L-smoothness of the target function on $G$, i.e. the fact that the gradient of the target function is Lipschitz-continuous on $G$. The proof consists of a number of lemmas and is deferred to the appendix. In what follows $\|\cdot\|_2$ is the usual Euclidean norm.

\begin{theorem}\label{lem:Lipschitz_gradient}
Under Assumptions 1-5
the mapping $\theta \mapsto \nabla_\theta \mathcal{L}_{L}(\theta)$  is   Lipschitz continuous on  the compact set $\Theta$ with a constant $\mathbf{L}>0$, i.e.
$$\|  \nabla_\theta \mathcal{L}_L(\theta) -  \nabla_\theta \mathcal{L}_L(\tilde\theta)\|_2 \le \mathbf{L} \|\theta-\tilde\theta\|_2 \mbox{ for all }  \theta,\tilde\theta\in\Theta.$$
\end{theorem}

Before we continue we need some estimates. Note that in the algorithm we have to simulate the $z$ and $\varepsilon$ random variables for each iteration step $k$, exactly $M$ times over the planning period $t=0,\ldots,T.$ Thus, we denote these random variables by $z_t^{k,j}$ and $\varepsilon_t^{k,j}$. In what follows we denote by $\cF_k = \sigma(\varepsilon_t^{1,j},\ldots,\varepsilon_t^{k,j},t=0,\ldots, T-1, j=1,\ldots,M) \vee \sigma(z_t^{1,j},\ldots ,z_t^{k,j}, t=1,\ldots,T, j=1,\ldots,M)$ the information available at iteration time $k$ of the algorithm, i.e.\ after $k$ simulations of the full state trajectory in the algorithm  and we denote by $\cF_k^G =\{ G\cap A: A\in \cF_k\}$ the trace $\sigma$-algebra on the good event. Note that in the construction of $g(\theta_k)$ from \eqref{eq:approxgradient} we use the samples indexed by $k+1.$

\begin{lemma}\label{lem:bias}
    The bias of the estimator for the gradient of the target function in \eqref{eq:approxgradient} at iteration time $k$ is given by
\begin{align}
\label{eq:bias-decomposition}
\EE\!\big[g(\theta_k)\mid \cF_k^G\big]
=
\nabla_\theta \cL_{L}(\theta_k)
+\mathrm{Bias}_M(\theta_k),
\end{align}
where
\begin{align}
\label{eq:bias-term}
\mathrm{Bias}_M(\theta_k)
=
\frac{2}{M}\sum_{\ell=1}^{L}\beta_\ell\,
\text{Re}\Big(
iu_\ell\,\EE[\nabla_\theta R_T^{\theta_k}|\cF_k^G]
-\overline{\varphi_{\theta_k}(u_\ell)}\,\nabla_\theta \varphi_{\theta_k}(u_\ell)
\Big),
\end{align}
and it satisfies the bound ($B_\nabla^R$ is the constant from Lemma \ref{lem:bound1} )
\begin{equation}
\label{eq:bias-bound}
\big\|\mathrm{Bias}_M(\theta_k)\big\|_2
\;\le\;
\frac{\hat{S}}{M} , \quad \hat{S}:=4\,B_\nabla^R\Big(\sum_{\ell=1}^{L}|\beta_\ell u_\ell|\Big).
\end{equation}
Moreover, we have
\begin{equation}
\label{eq:square-bound} \EE\!\big[ \|g(\theta_k)\|_2^2\mid \cF_k^G\big]\le \hat{S}^2.
\end{equation}
\end{lemma}

For the  main convergence result we need some further assumptions on the step sizes (which are the usual ones). 

\begin{assumption} \label{ass:ConvAlgo}
 Let $(\alpha_k), \alpha_k> 0$ be a sequence of step sizes such that $$\sum_{k=0}^\infty \alpha_k = \infty\quad \mbox{ and } \quad \sum_{k=0}^\infty \alpha_k^2 <\infty. $$
\end{assumption}
In what follows we denote $A_K := \sum_{k=0}^{K-1} \alpha_k$ and $\hat{S}$ as in the previous lemma.

\begin{theorem}[Expected average stationarity result]
\label{thm:conv-discrete} Under Assumptions 1-6 we obtain for the iterated sequence of parameters $(\theta_k)$ from our algorithm given by $\theta_{k+1} \leftarrow \theta_k - \alpha_k \, g(\theta_k)$ that 
\begin{eqnarray}\label{eq:prob-bound-good}
\sum_{k=0}^{K-1}\frac{\alpha_k}{A_K}  \EE_G\Big[\|\nabla_\theta \cL_L(\theta_k)\|_2^2 \Big] &\le& \frac{2}{A_K} \cL_L(\theta_0)+ \mathbf{L} \hat{S}^2  \frac{\sum_{k=0}^{K-1} \alpha_k^2}{A_K} + \frac{\hat{S}^2}{M^2}   \\
&=&O\Big(\frac{1}{A_K}\Big)+O\Big(\frac{1}{M^2}\Big)  \mbox{  for }K,M\to \infty,
\end{eqnarray}
\end{theorem}

Thus, when we increase the number of iterations and the number of samples for the empirical gradient, then the average expected gradient tends to zero:
Note that for $\alpha_k=1/k+1$ we have $A_K=O(\log(K))$ i.e. for the convergence it is more important to increase the number of iterations than the number of samples per gradient. See \cite{Liu2023HighProbSGD} for a discussion on how to improve convergence rates. \\

\begin{proof}
Fix $k$ and work conditionally on $\cF_k^G$ so that $\theta_k$ is deterministic. In what follows we assume that the $\theta_k$ stay in $\Theta.$
By $\mathbf{L}$--smoothness of $\cL_{L}$ ,
for $\theta_{k+1}=\theta_k-\alpha_k g(\theta_k)$ we have the standard descent inequality
\begin{equation}
\label{eq:smooth-descent}
\cL_{L}(\theta_{k+1})
\le
\cL_{L}(\theta_k)
-\alpha_k\big\langle\nabla \cL_{L}(\theta_k),\, g(\theta_k)\big\rangle
+\frac{\mathbf{L}}{2}\alpha_k^2\|g(\theta_k)\|_2^2.
\end{equation}
Taking conditional expectation $\EE[\cdot\mid \cF_k^G]$ and using linearity yields
\begin{align}
\label{eq:cond-step}
\EE\!\big[\cL_{L}(\theta_{k+1})\mid \cF_k^G\big]
&\le
\cL_{L}(\theta_k)
-\alpha_k\Big\langle\nabla \cL_{L}(\theta_k),\,\EE[g(\theta_k)\mid \cF_k^G]\Big\rangle
+\frac{\mathbf{L}}{2}\alpha_k^2\EE\!\big[\|g(\theta_k)\|_2^2\mid \cF_k^G\big].
\end{align}

Define the conditional bias vector
\[
b_k
:=
\EE[g(\theta_k)\mid \cF_k^G]-\nabla \cL_{L}(\theta_k).
\]
Then
\[
\Big\langle\nabla \cL_{L}(\theta_k),\,\EE[g(\theta_k)\mid \cF_k^G]\Big\rangle
=
\|\nabla \cL_{L}(\theta_k)\|_2^2
+
\big\langle\nabla \cL_{L}(\theta_k),\, b_k\big\rangle.
\]

Apply Young's inequality $|\langle u,v\rangle|\le \tfrac12\|u\|_2^2+\tfrac12\|v\|_2^2$ to obtain
\[
-\alpha_k\big\langle\nabla \cL_{L}(\theta_k),\, b_k\big\rangle
\le
\frac{\alpha_k}{2}\|\nabla \cL_{L}(\theta_k)\|_2^2
+\frac{\alpha_k}{2}\|b_k\|_2^2.
\]
Plugging this into \eqref{eq:cond-step} gives the one-step inequality
\begin{align}
\label{eq:one-step-final}
\EE\!\big[\cL_{L}(\theta_{k+1})\mid \cF_k^G\big]
&\le
\cL_{L}(\theta_k)
-\frac{\alpha_k}{2}\|\nabla \cL_{L}(\theta_k)\|_2^2
+\frac{\alpha_k}{2}\|b_k\|_2^2
+\frac{\mathbf{L}}{2}\alpha_k^2\EE\!\big[\|g(\theta_k)\|_2^2\mid \cF_k^G\big].
\end{align}

Now use Lemma \ref{lem:bias} to obtain
\[
\frac{\alpha_k}{2}\|b_k\|_2^2
\le
\frac{\alpha_k}{2}\cdot \frac{\hat{S}^2}{M^2}, \quad 
\frac{\mathbf{L}}{2}\alpha_k^2\EE\!\big[\|g(\theta_k)\|_2^2\mid \cF_k^G\big]
\le
\frac{\mathbf{L}}{2} \, \alpha_k^2 \hat{S}^2.
\]
Substitute these into \eqref{eq:one-step-final}, and then take $\EE_G[\cdot]$ to remove the conditioning on $\cF_k$ and rearrange the equation:
\begin{align}
\label{eq:avg-bound-pre}
\frac{\alpha_k}{2}\EE_G\!\Big[\|\nabla \cL_{L}(\theta_k)\|_2^2\ \Big]
&\le
\EE_G\!\Big[\cL_{L}(\theta_k)-\cL_{L}(\theta_{k+1})\ \Big]
+\frac{\mathbf{L}}{2}  \, \alpha_k^2 \hat{S}^2
+\frac{\alpha_k}{2}\cdot \frac{\hat{S}^2}{M^2}.
\end{align}

Sum \eqref{eq:avg-bound-pre} from $k=0$ to $K-1$. The loss term telescopes and $\cL_{L}\ge 0$ gives
\begin{align}
\label{eq:sum-bound}
\frac12\sum_{k=0}^{K-1}\alpha_k\EE_G\!\Big[\|\nabla \cL_{L}(\theta_k)\|_2^2\ \Big]
&\le
\cL_{L}(\theta_0)+ \frac{\mathbf{L}}{2}  \,  \hat{S}^2 \,\sum_{k=0}^{K-1}\alpha_k^2+ \frac{\hat{S}^2}{2 M^2} A_K.
\end{align}
Divide by $A_K$ to obtain the statement.
\end{proof}

\section{Applications}\label{sec:appl}
We implemented   Algorithm \ref{alg:theoretical} directly with the following set of parameters for discretization of the Fourier domain with a uniform grid with nodes $ u_\ell \;=\; -K + (\ell-1)\,\Delta u$ with $
\Delta u \;=\; \frac{2K}{L-1}$, for $ \ell=1,\dots,L$ and a trapezoidal rule for computing the squared difference of the smoothed characteristic functions. For the smoothing weight we choose Gaussian $w(u)\;=\;e^{-\alpha u^{2}}$ and $\alpha$ is a smoothing parameter (usually around $0.05$ in our examples). 

The policy $(\pi_t)$ is implemented as a feed-forward neural network $f(\theta,s_t,R_t,z_t,\frac{T-t}{T}) \sim \pi_t(\cdot|s_t,R_t)$
receiving the four inputs $(s,R,z,\frac{T-t}{T})\in\mathbb{R}^4$.
The inputs are first concatenated and passed through a linear layer of width
$P=256$, followed by layer normalization and a ReLU activation.
This is followed by $K=4$ residual blocks, each consisting of a linear map
$\mathbb{R}^d\!\to\!\mathbb{R}^d$ with layer normalization and ReLU, and a skip
connection.  A final linear layer maps the resulting hidden state to a scalar
output.  The output layer is initialized to zero so that
the initial action distribution is unbiased. No activation function is applied to the final layer.

\begin{remark}
	The choice of our network architecture here is covered by our main theorem, except for the ReLU activation, since the proof technique remains unchanged. We only need to make the following readjustments: (1) the regularity constants for deeper and wider neural networks, (2) residual blocks and skip connections preserve the regularity (3) layer normalization needs to be accounted for in the regularity analysis but poses no challenge. 
	
	The choice of ReLU was mainly empirical as it gave better results than \verb|tanh|. Nevertheless  \verb|tanh| works as well. 
	 
\end{remark}
For most of the examples we run the algorithm for a predefined number of iterations $k=400$, usually enough for convergence. Along with the loss, we also track  a validation error using newly sampled noises. For the LQ example where the set of optimal policies is unknown to us but within the reachable set, we report a confidence interval for the gradient estimator to ensure convergence to a stationary point. For the last example with the compact support we have an exact, unique (up to a truncation level) optimal control distribution against which we compare the learned distribution.
 
The hyperparameters for all examples can be found in Appendix \ref{appx:numerical_examples} and the exact code for all numerical examples can be found \href{https://github.com/ThanosVasileiadis/Markov-Decision-Processes-of-the-Third-Kind-Learning-Distributions-by-Policy-Gradient-Descent}{\emph{here}} \footnote{https://github.com/ThanosVasileiadis/Markov-Decision-Processes-of-the-Third-Kind-Learning-Distributions-by-Policy-Gradient-Descent}. To run our examples we use the HAICORE and HoreKa servers, courtesy of the Helmholtz Association. Smaller values of the parameters could run on an M1 MacBook Air with 16GB of RAM. 

\subsection{Linear Quadratic Control Problem}
\label{subse:LQ}

 We consider a discrete--time stochastic control problem over a finite horizon \(T\in\mathbb{N}\).
The scalar state \(s_t\) evolves under real--valued actions \(a_t\) from a compact set according to
\begin{equation}\label{eq:dynamics}
  s_{t+1} \;=\; s_t + a_t + \sigma_\varepsilon\,\varepsilon_{t+1}, 
  \qquad t=0,1,\dots,T-1,
\end{equation}
where \(\sigma_\varepsilon=0.1\) is fixed and \((\varepsilon_t)\) are i.i.d.\  standard normal random variables, independent of the initial state \(s_0=0\).  At each step, a one-stage reward (negative cost) is incurred,
\begin{equation}\label{eq:stage-cost}
  r(s_t,a_t) \;=\; -\tfrac12\bigl(s_t^2 + a_t^2\bigr),
\end{equation}
and the cumulative reward is updated as
\begin{equation}\label{eq:cumulative-cost}
  R_0 := 0,
  \qquad 
  R_{t+1} := R_t + r(s_t,a_t),
  \qquad 
  R_T = \sum_{t=0}^{T-1} r(s_t,a_t) .
\end{equation}
Note that our Assumptions 1, 2, and 3 are satisfied and Assumption 4 and 5 are satisfied on the good event which is sufficient for our considerations. For example, for Assumption \ref{ass:rFLip} (i) we have on the good event $G$, $|s_t^2-\tilde s_t^2|\le |s_t-\tilde s_t| 2 B^s_t$.
We are now interested in the distribution of $R_T$. We want to match it with a given distribution. Thus, we consider the lifted MDP:
At time \(t\), the controller observes the current state \(s_t\), the cumulative reward \(R_t\), and the normalized time--to--go
\(\tau_t := (T-t)/T\) which we use instead of $t$ for the practical implementation.
A  randomized Markov policy is a sequence \(\pi=(\pi_t)_{t=0}^{T-1}\) with stochastic kernels
\(\pi_t(\cdot\,|\,s_t,R_t)\) on \(\mathbb{R}\), parametrized by $f(\theta,s,R,z,\tau)$. The distribution of $z$ is a standard Gaussian. 

\paragraph{Target Distribution}
In order to test the algorithm we choose  a target characteristic function that can be attained by a deterministic policy. More precisely, we choose the benchmark feedback control $a_t= -0.5 s_t$ and generate $M=100\times 1024$ trajectories to construct  the target characteristic function  \((R_T^{m})_{m=1}^M\) under the dynamics
\eqref{eq:dynamics}--\eqref{eq:cumulative-cost} for $T=10$ and define its empirical characteristic function as
\begin{equation}\label{eq:empirical-policy}
  \widehat{\phi}^*(u)
  \;=\;
  \frac{1}{M}\sum_{m=1}^{M}
  e^{\,iu R_T^{m}} .
\end{equation}
The control objective is to choose \(\pi\) so that the empirical law of \(R_T\) matches the target distribution in the frequency domain as formulated in \eqref{eq: parametrised Loss}.

\paragraph{Numerical Results and Discussion}

The algorithm is stochastic and the optimal policy distribution is not unique. Thus, in order to create a reproducible example and to analyze the behavior of the algorithm we choose to "freeze" randomness with a given seed. We perform $10$ runs of the algorithm with different seeds from $2026-2035$. Each run is independent and we summarize the findings across seeds in  Table \ref{tab:lq_seed_audits}. We report the complete  parameters and hyperparameters of the example in Appendix \ref{appx:numerical_examples}.   

\begin{table}[htbp]
\centering
\small
\begin{tabular}{rrrrrc}
\toprule
Seed & k &
\shortstack{CF loss} &
\shortstack{$W_1$\\distance} &
\shortstack{Gradient\\$5\%$ conf } \\
\midrule
2026 & 290 & $6.293\times10^{-6}$ & 0.001166 & 0.02141  \\
2027 & 260 & $1.264\times10^{-5}$ & 0.001080 & 0.04553  \\
2028 & 380 & $4.109\times10^{-6}$ & 0.001188 & 0.03281  \\
2029 & 310 & $3.584\times10^{-6}$ & 0.001404 & 0.01510  \\
2030 & 399 & $2.236\times10^{-6}$ & 0.001272 & 0.01556  \\
2031 & 390 & $2.333\times10^{-6}$ & 0.001329 & 0.01905  \\
2032 & 390 & $1.929\times10^{-6}$ & 0.001113 & 0.01581  \\
2033 & 360 & $2.464\times10^{-6}$ & 0.001252 & 0.02971  \\
2034 & 270 & $2.909\times10^{-6}$ & 0.001504 & 0.01871  \\
2035 & 399 & $3.882\times10^{-6}$ & 0.001305 & 0.02784  \\
\bottomrule
\end{tabular}
\caption{
Independent results for ten training seeds, using a common
fixed target distribution. }
\label{tab:lq_seed_audits}
\end{table}

The column $k$ denotes the index of
the policy neural network selected by the smallest monitored validation loss. The
$CF \; loss$ is the mean weighted characteristic-function distance, while $W_1$ is the Wasserstein distance between cumulative reward distributions. 

To effectively test for stationary points we monitor the gradient estimator while sampling new noises. For the sake of brevity we demonstrate and explain the computation only for seed $2028$ at $k=380$. The same results and figures hold for the others as well. We obtain the estimated norm of the gradient
$\|\overline g\|_2=0.01426$ and the estimated standard deviation of the estimator is $\widehat s=0.004147$. Using the second-moment screening radius $r=\widehat s/\sqrt{0.05}=0.01855$ gives the confidence interval
$$
\left[
\max\{0,\|\overline g\|_2-r\},
\ \|\overline g\|_2+r
\right]
=
[0,\,0.03281].
$$
This interval contains zero and its upper endpoint is below the gradient-norm tolerance $0.05$, supporting approximate stationarity at that tolerance.

For the same seed, given the stochastic nature of the algorithm training is quite stable and approaches the stationary region as  can be seen in Figure \ref{fig:lq_training10}. Also, judging visually, the distributions of rewards (that are actually costs) and their characteristic functions match very closely in Figure \ref{fig:lq_distribuitons}.

\begin{figure}[H]
    \centering
    
        \includegraphics[width=0.7\textwidth]{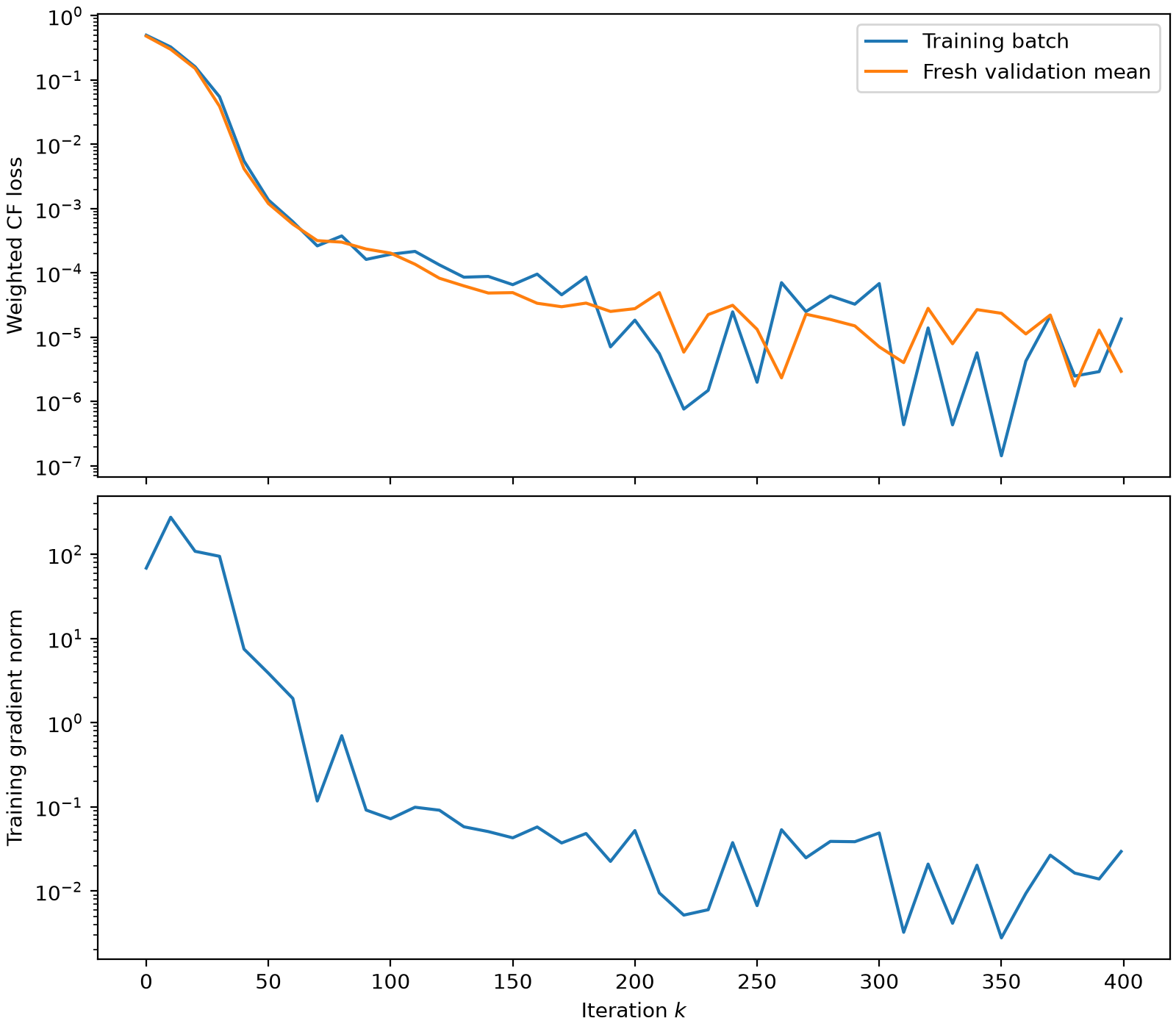}
        \caption{Training history}
        \label{fig:lq_training10}
\end{figure}

\begin{figure}[H]
    \centering
    \begin{subfigure}[b]{0.55\textwidth}
        \includegraphics[width=\textwidth]{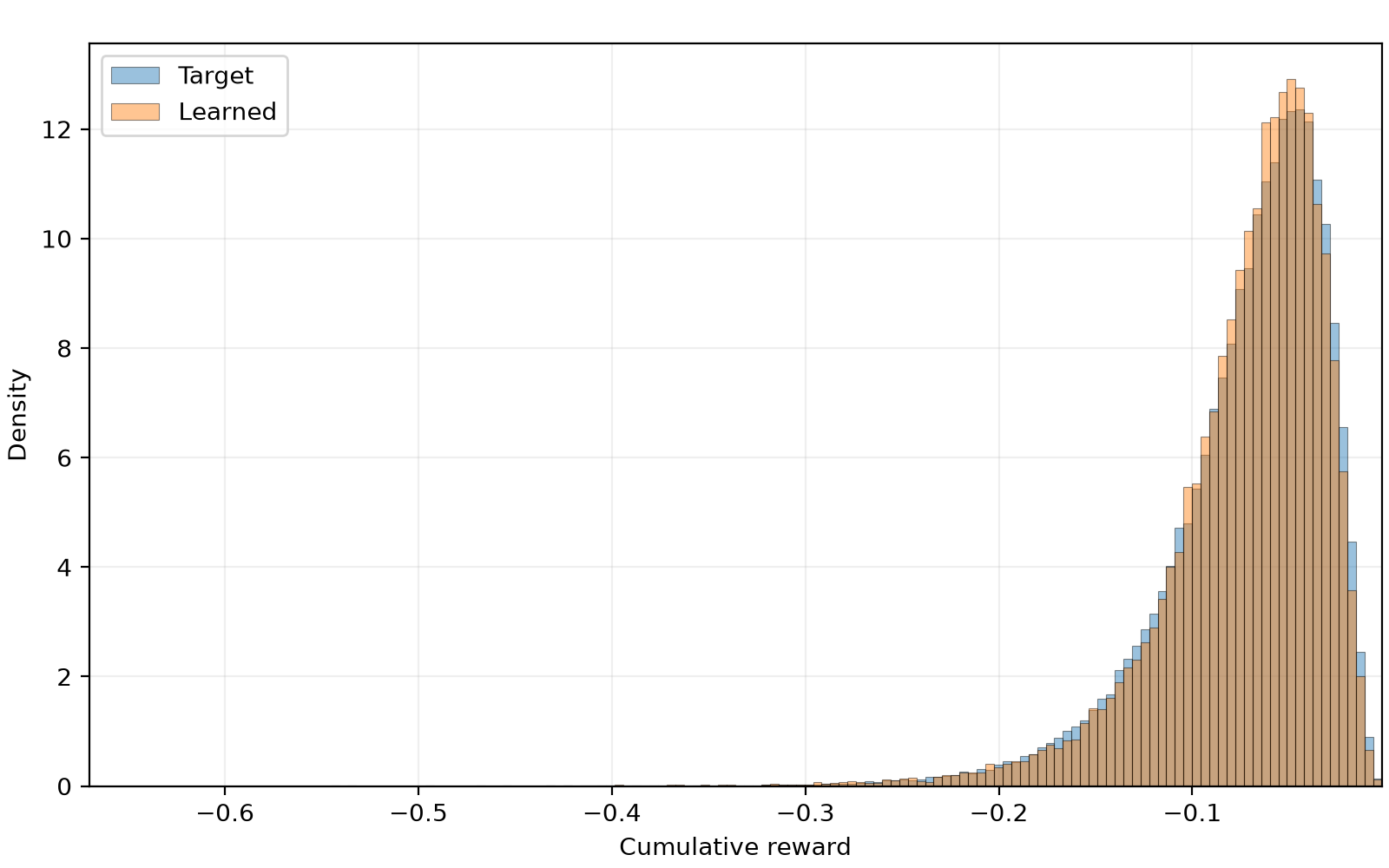}
        \caption{Histogram}
        \label{fig:return_dist2}
    \end{subfigure}
    ~ 
    \begin{subfigure}[b]{0.4\textwidth}
        \includegraphics[width=\textwidth]{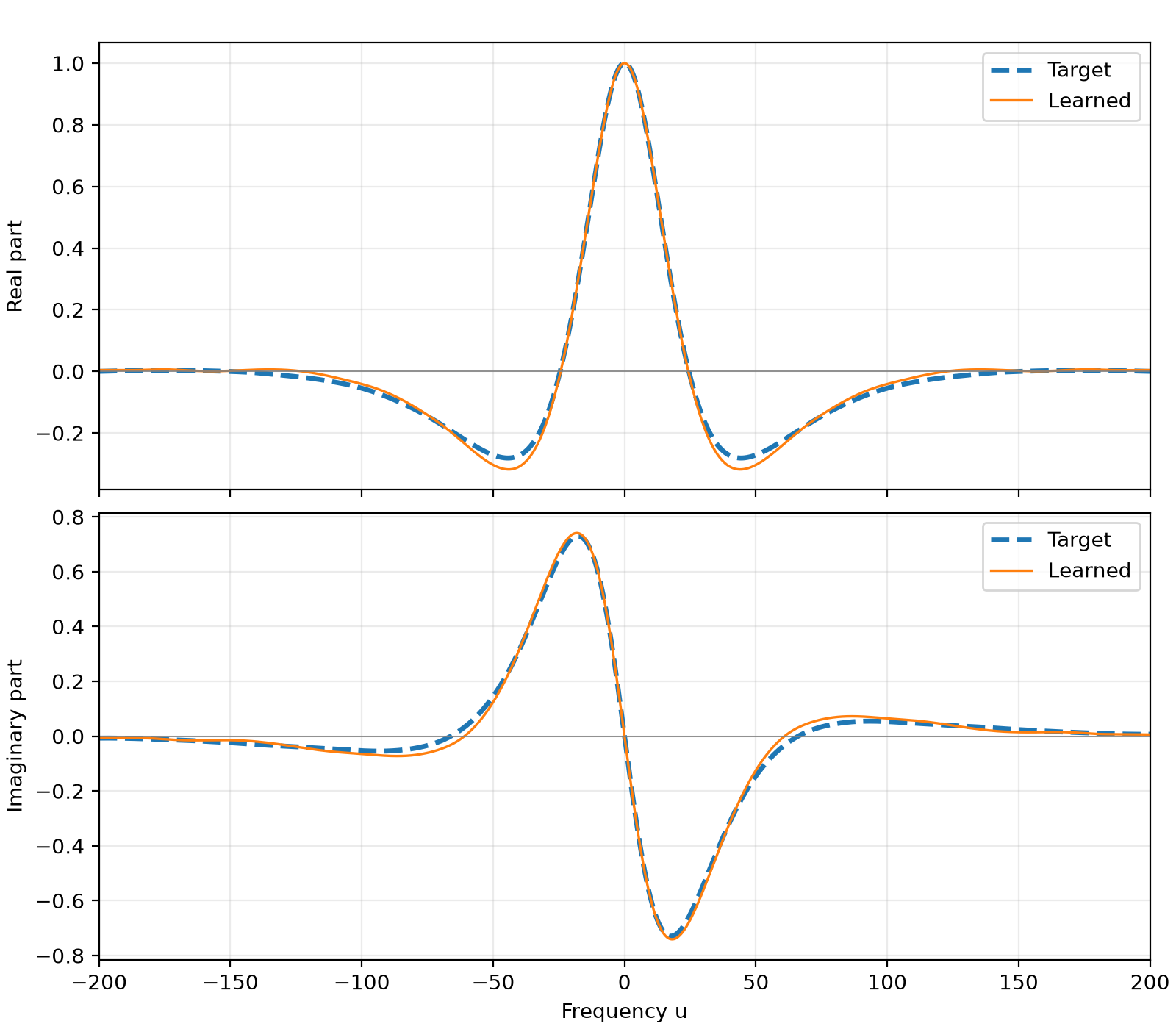}
        \caption{Characteristic Function }
        \label{fig:return_cf2}
    \end{subfigure}
    \caption{Distributions of Rewards}
     \label{fig:lq_distribuitons}
     
\end{figure}
 
 Now we turn to the policy distributions that we learn. All runs reach a region of almost zero (average) gradient with a tolerance of $0.05$. Our findings suggest that there is not a single policy distribution attracting the solutions. We first track the stability of these policy distributions across training in Figure \ref{fig:lq_action_stability} (for seed 2028 throughout). We compute the Wasserstein distance for the terminal $t=T-1$ action distribution over a fixed state-reward grid, $[-0.5 ,0.5 ]\times [-0.5, 0]$ with $21$ points each. Mean refers to the average distance over the grid and max to the maximum distance across all points. We compare with the initial distribution, a $10$-iterations lag and the previous step.

    \begin{figure}[H]
    \centering
    \begin{subfigure}[t]{0.45\textwidth}
        \vspace{0pt}
        \includegraphics[width=\textwidth]{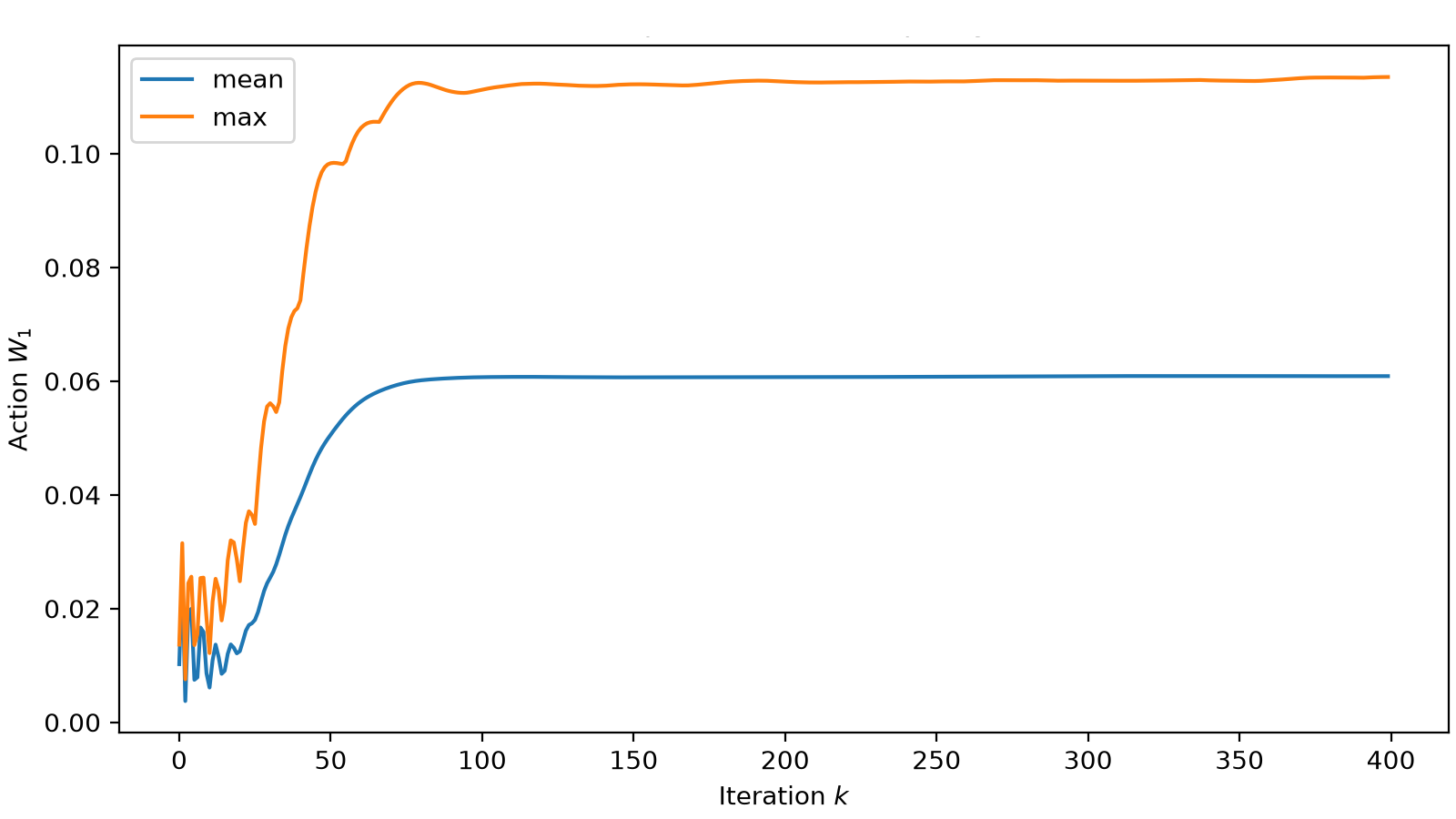}
        \caption{Wasserstein distance between $k$-th iteration and initial action distribution.}
        \label{fig:action_stability_initial}
    \end{subfigure}
    \begin{subfigure}[t]{0.45\textwidth}
        \vspace{0pt}
        \includegraphics[width=\textwidth]{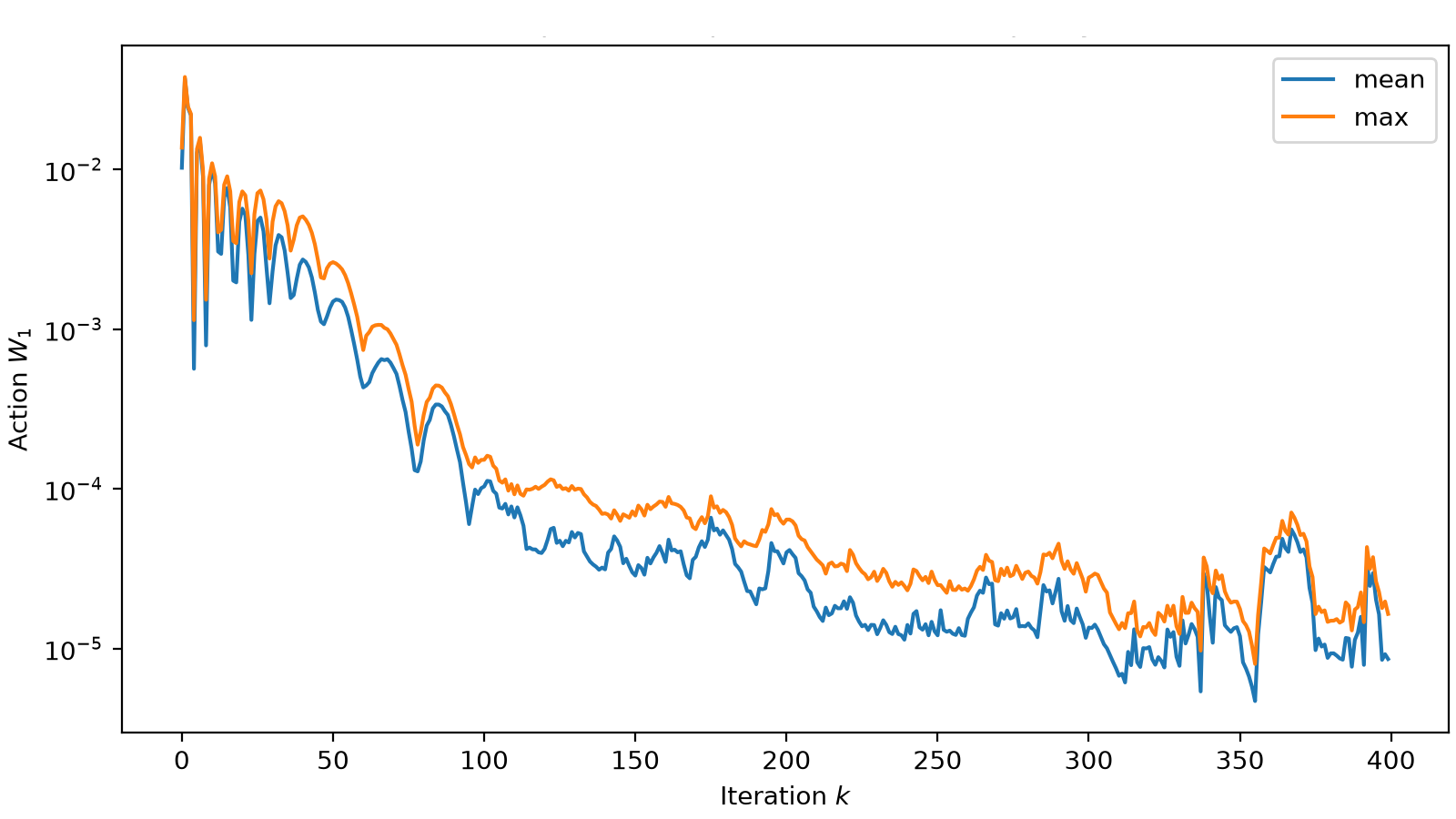}
        \caption{Wasserstein distance between the action distributions of the iteration $k$ and iteration $k-1$.}
        \label{fig:action_stability_previous}
    \end{subfigure}
    \caption{Stability of action distributions throughout training}
    \label{fig:lq_action_stability}
\end{figure}
    Next, in Figure \ref{fig:lq_pairwise_heatmaps} we report the final maximum and mean error between the different action distributions at the end of training. The closest and the farthest action CDF  appear in Figure \ref{fig:lq_cdfs}. To further support  our argument of non-uniqueness and non-identifiability with the policy that generated the target distribution we present in Figure \ref{fig:lq_cdfs_determinitic}  the cumulative distribution functions of all seeds and the reference policy $a_t=-0.5 s_t$ for some combinations of states and rewards in our aforementioned grid. The policy CDFs  differ from one another and are significantly different from the deterministic reference policy. In Appendix \ref{appx:numerical_examples} we provide further plots of the action densities of all seeds for different states and rewards on the grid.     
 
\begin{figure}[H]
    \centering
    \begin{subfigure}[b]{0.45\textwidth}
        \includegraphics[width=\textwidth]{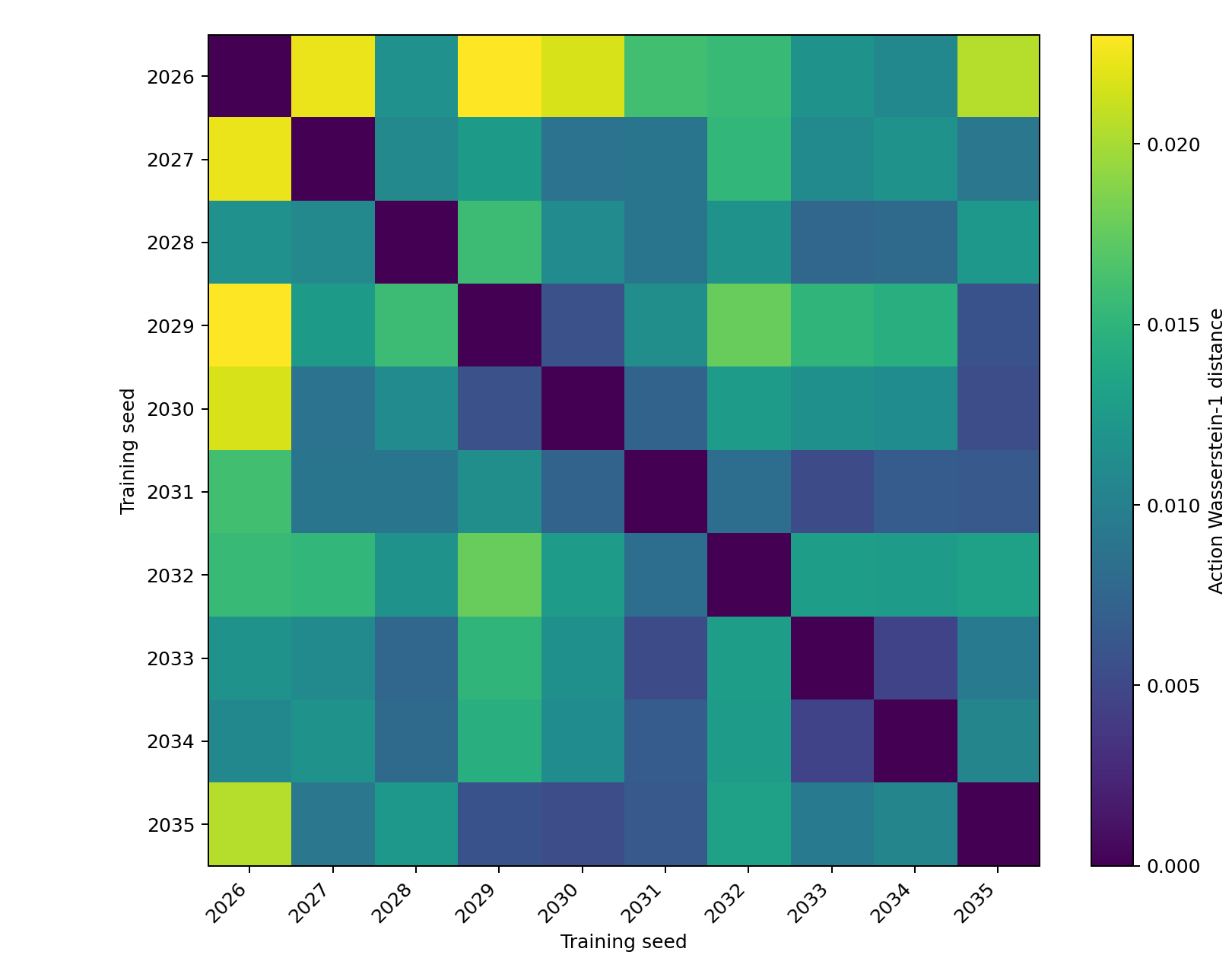}
        \caption{Max Wasserstein distance heatmap}
        \label{fig:pairwise_grid_max}
    \end{subfigure}
         ~ 
    \begin{subfigure}[b]{0.45\textwidth}
        \includegraphics[width=\textwidth]{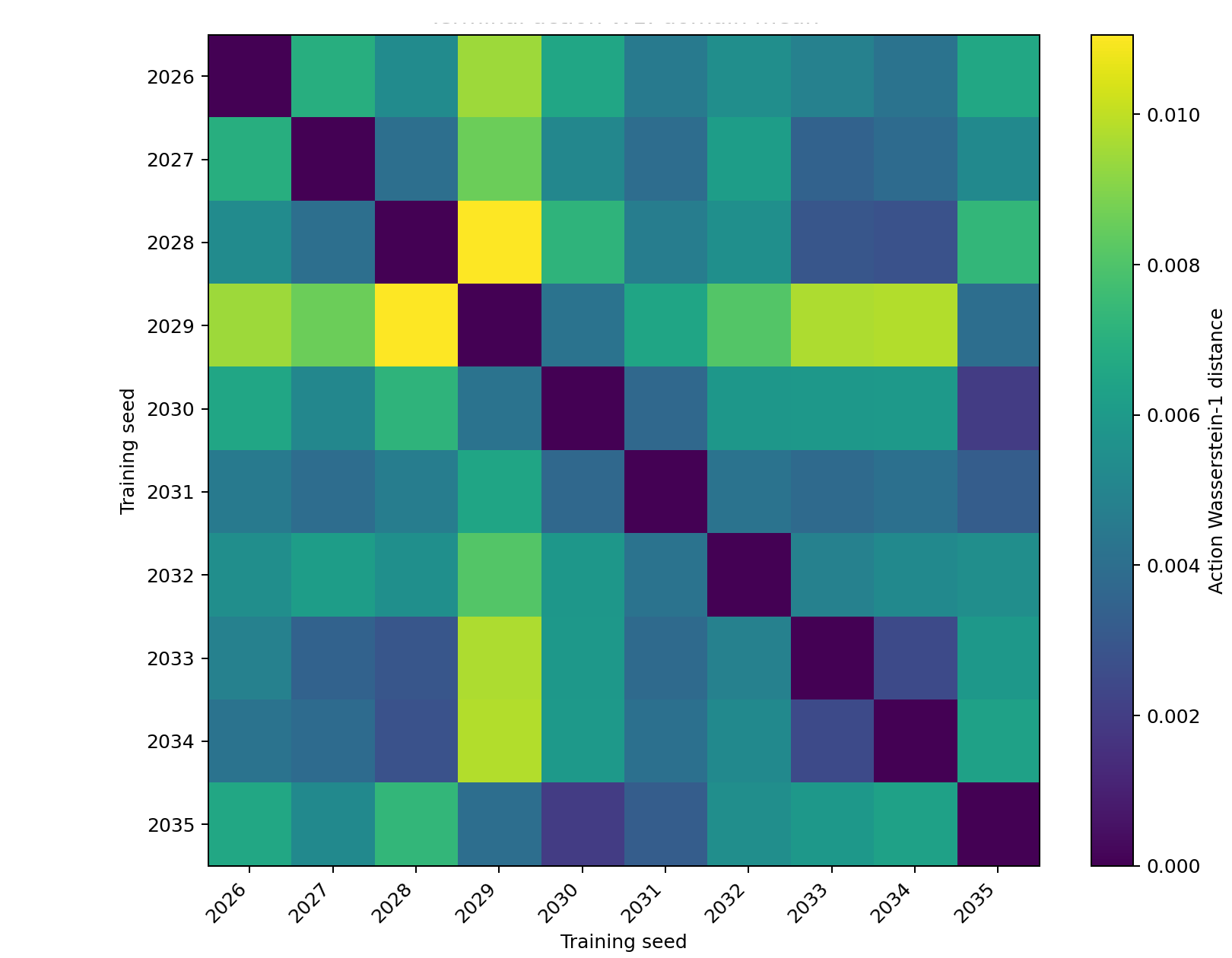}
        \caption{Mean Wasserstein distance heatmap}
        \label{fig:pairwise_domain_mean}
    \end{subfigure}
    \caption{Pairwise heatmaps}
     \label{fig:lq_pairwise_heatmaps}
\end{figure}

Our findings suggest the non-uniqueness of optimal action distributions. Our learning procedure leads to a continuous set of stationary action distributions that effectively minimize the distance between the target and the controlled reward distributions. We leave to future work the investigation of the exact properties of this set.

\begin{figure}[H]
    \centering
    \begin{subfigure}[b]{0.45\textwidth}
        \includegraphics[width=\textwidth]{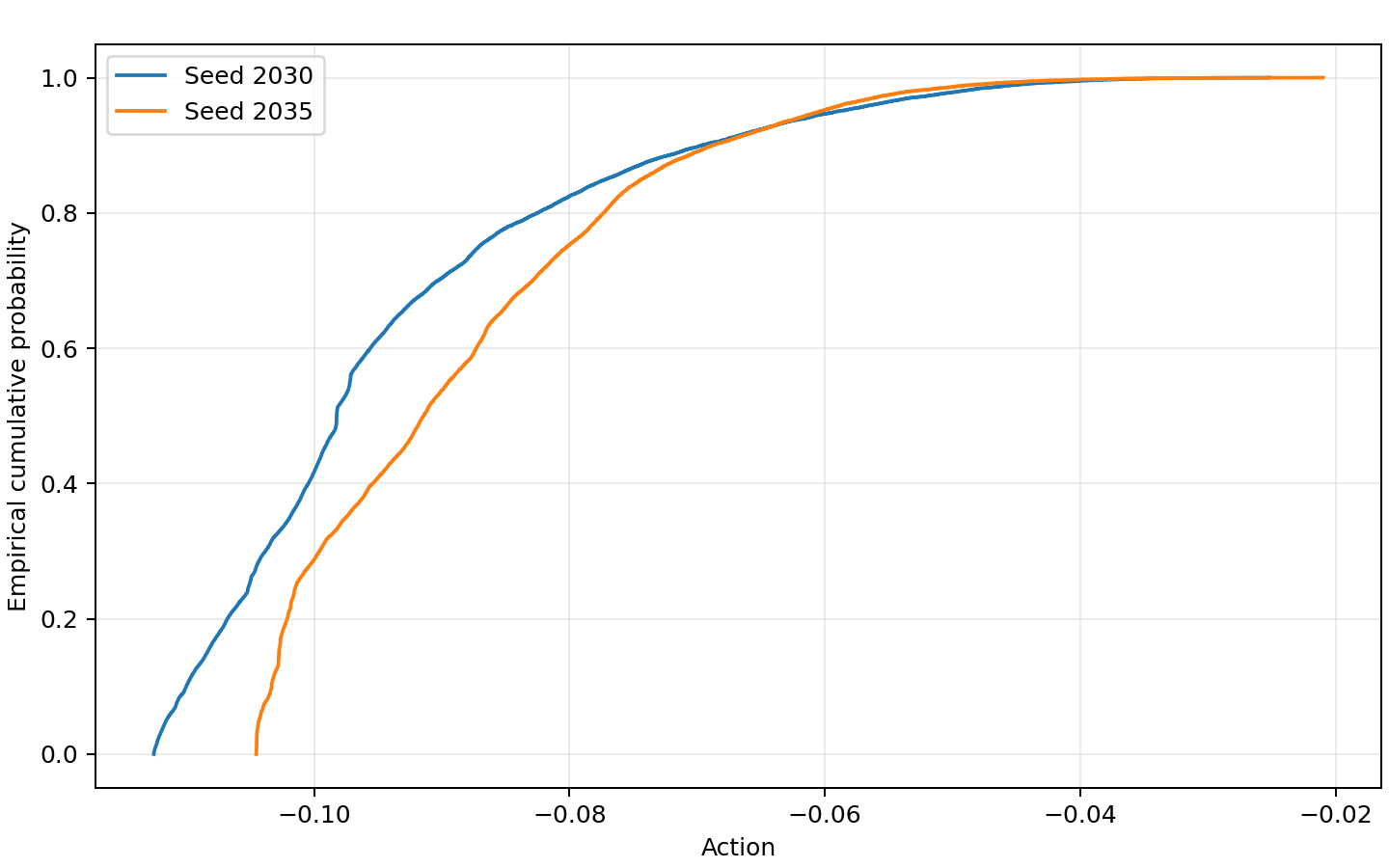}
        \caption{Closest pair at $s=0.2, r=0$}
        \label{fig:closest_pair}
    \end{subfigure}
         ~ 
    \begin{subfigure}[b]{0.45\textwidth}
        \includegraphics[width=\textwidth]{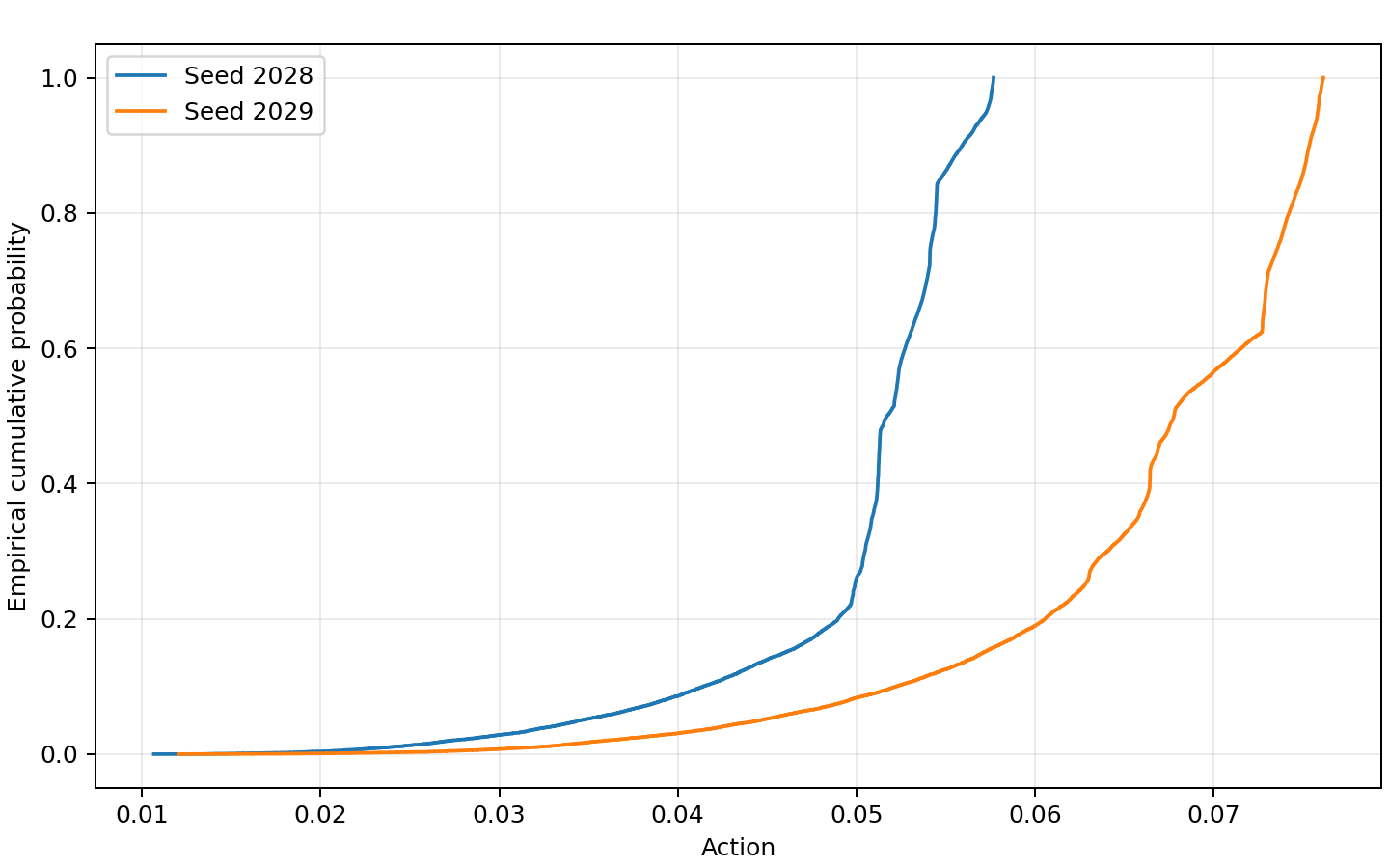}
        \caption{Farthest pair at $s=-0.12, r=-0.15$}
        \label{fig:farthest_pair}
    \end{subfigure}
    \caption{CDFs of closest and farthest action distributions}
     \label{fig:lq_cdfs}
 
\end{figure}

\begin{figure}[H]
    \centering
        \includegraphics[width=\textwidth]{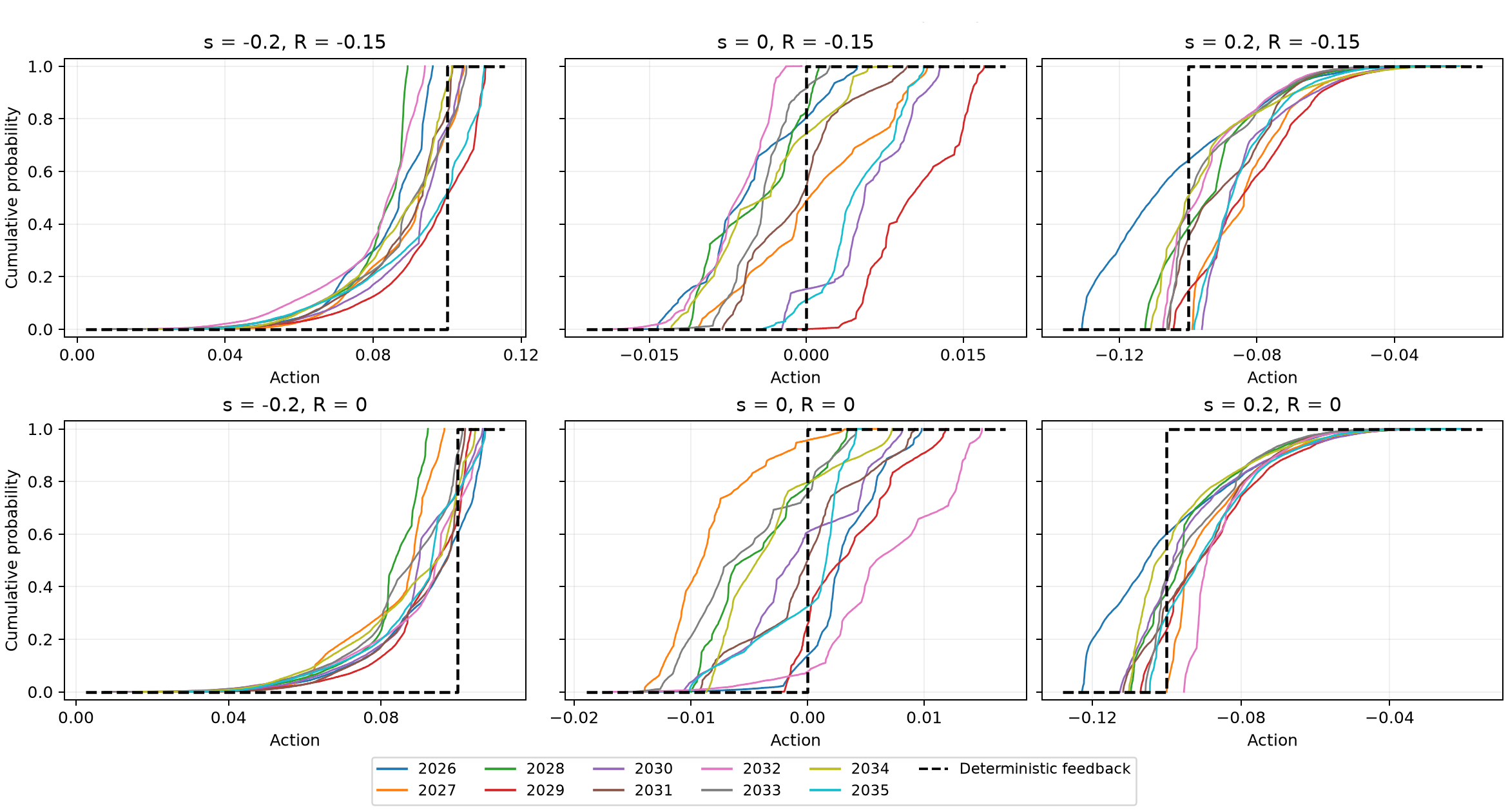}
        \caption{CDFs in comparison with deterministic reference feedback policy}
        \label{fig:lq_cdfs_determinitic}
\end{figure}

 To conclude this example, we present a sensitivity analysis of the hyperparameters, see Section \ref{subse:discretization of the integral}, associated with the computation of the loss. We focus on the number of Fourier nodes $L$, the frequency cutoff $K$, and the Gaussian weight $\alpha$. In Figure \ref{fig:param_sensitivity} we present results for only  two seeds because the influence of the parameters is clear and independent of the randomness. The most influential parameter is the Gaussian weight which has to balanced with the cutoff and the number of nodes because a very small value could make training unstable.   
 
\begin{figure}[H]
    \centering
    \begin{subfigure}[b]{0.3\textwidth}
        \includegraphics[width=\textwidth]{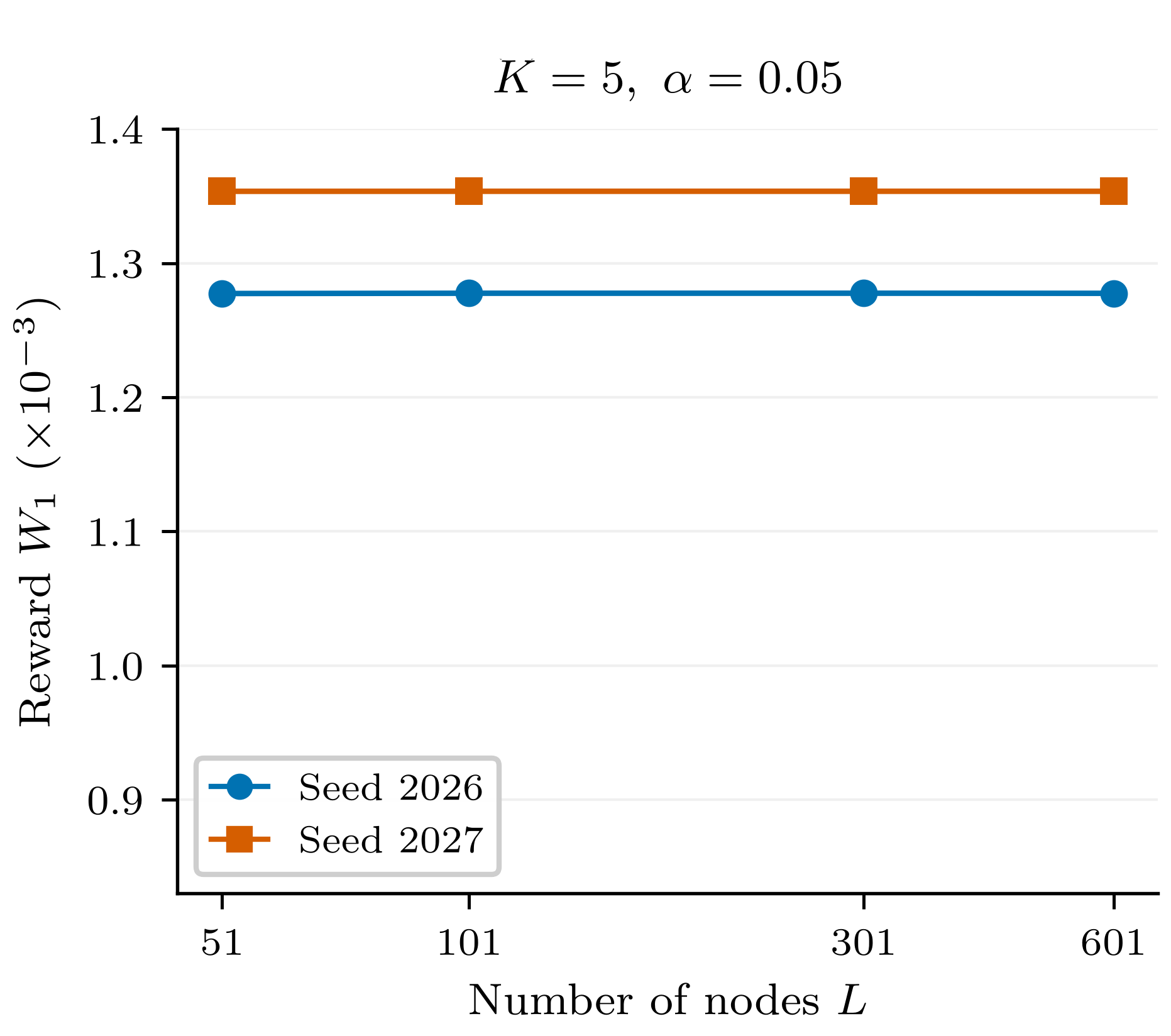}
        \caption{Sensitivity with respect to the number of nodes}
        \label{fig:nodes_sensitivity}
    \end{subfigure}    
     \begin{subfigure}[b]{0.3\textwidth}
        \includegraphics[width=\textwidth]{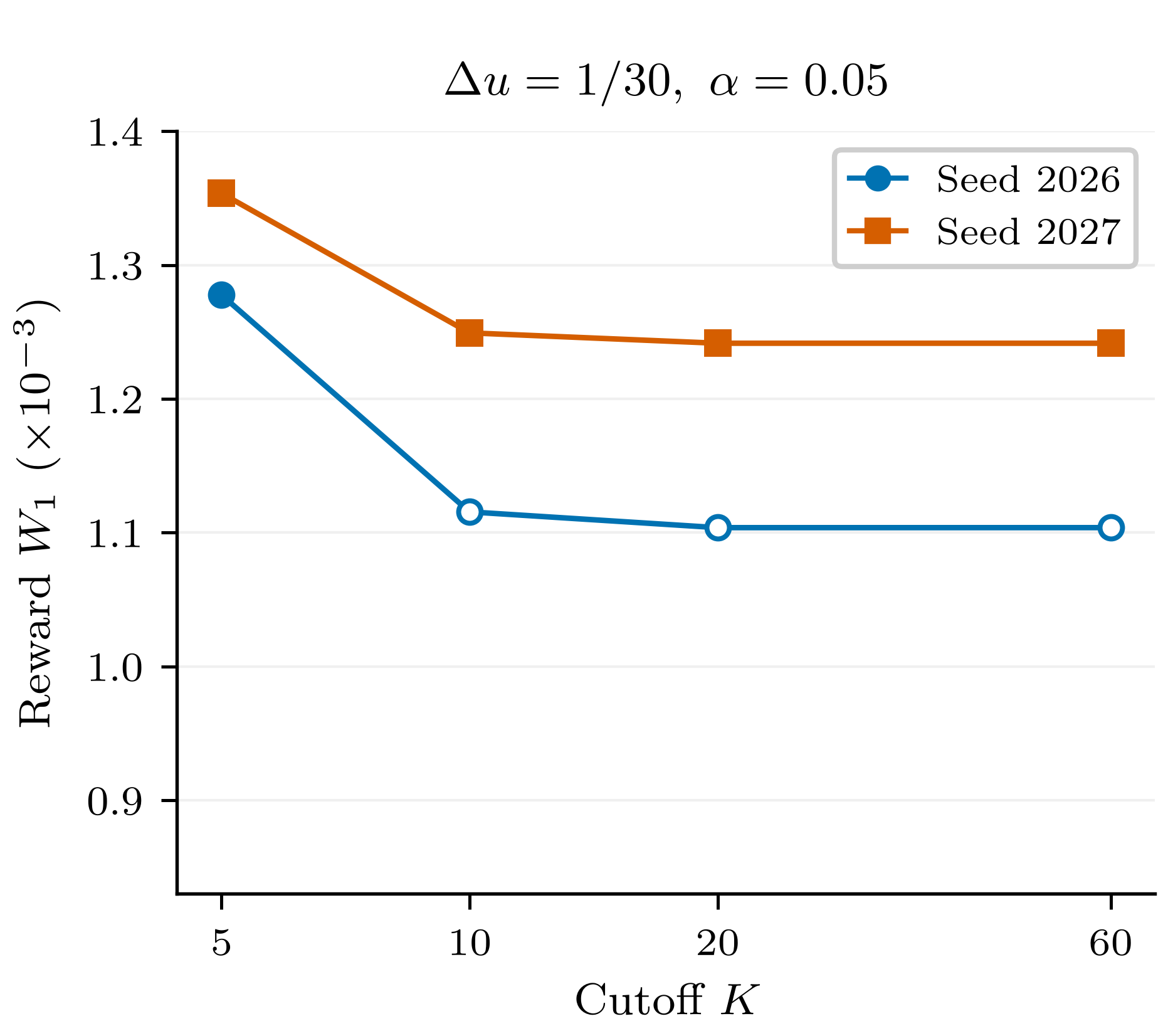}
        \caption{Sensitivity with respect to the cutoff}
        \label{fig:cutoff_sensitivity}
    \end{subfigure}
    ~ 
    \begin{subfigure}[b]{0.3\textwidth}
        \includegraphics[width=\textwidth]{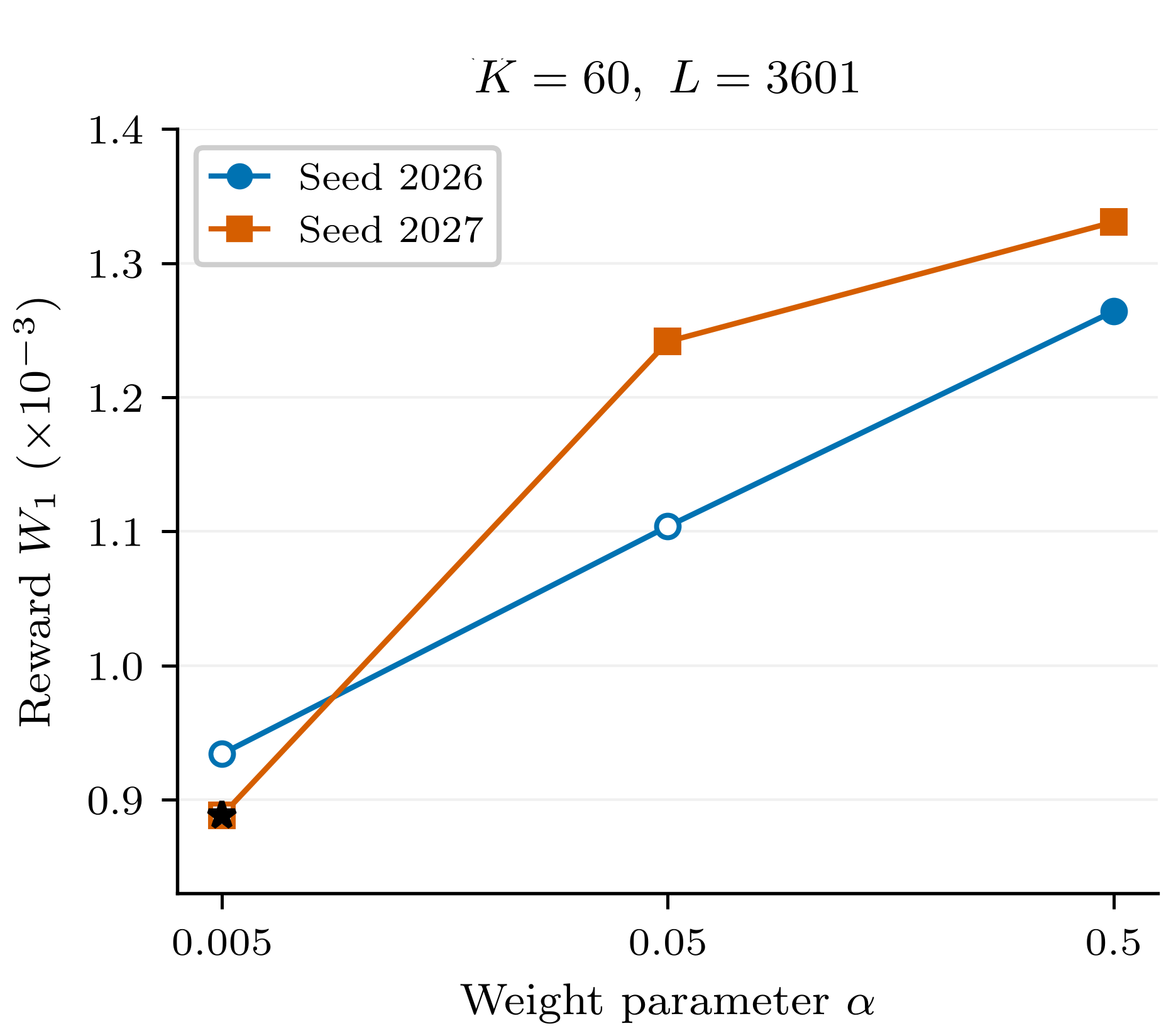}
        \caption{Sensitivity with respect to the Gaussian weight}
        \label{fig:gaussian_sensitivity}
    \end{subfigure}
    \caption{Hyperparameter sensitivity}
     \label{fig:param_sensitivity}
     
\end{figure}

\subsection{Targeting Wealth Distributions by Investment}
\label{subse:wealth_example}

We consider a simple discretized Black-Scholes model with time-step length $\Delta t$, i.e. the bond price evolution is given by $B_t=e^{r t \Delta t}$ for $r>0$ and $t\in\NN_0$ and the stock price is given by $x_0=1$ and 
$$ x_{t+1}=x_t \cdot \exp\Big((\mu-\frac12 \sigma^2)\Delta t+\sigma \sqrt{\Delta t} \;\varepsilon_{t+1}\Big)$$
for $\mu\in \mathbb{R}$ and $\sigma>0$ and i.i.d.\ standard normal $(\varepsilon_t)$. If $a_t$ is the amount of money invested in the stock at time $t$ and it is decomposed as $a_t=\pi_t \;s_t$, where $\pi_t\in[0,1]$ is the fraction of our wealth $s_t$ invested in the stock, then the wealth $s_{t+1}$ at time $t+1$  is given as follows:
$$ s_{t+1}= F(s_t,a_t,\varepsilon_{t+1})  = e^{r  \Delta t} \big( s_t+ a_t y_{t+1}\big)$$
with $$y_{t+1} = e^{-r  \Delta t} \exp\Big((\mu-\frac12 \sigma^2)\Delta t+\sigma \sqrt{\Delta t} \;\varepsilon_{t+1}\Big)-1, \qquad \varepsilon_{t+1} \sim N(0,1),$$
being the excess return over the bond. 
 
As in the motivating example we consider the reward to be non-stationary and set $r_t(s,a)=0$ for $t=0,\ldots,T$ and $r_{T+1}(s,a)=s.$ Thus we are interested in the distribution of the terminal wealth $s_T$ which  is given by 
\begin{equation}
	\label{eq:terminal_wealth}
s_T
= e^{r T \Delta t}
  \Biggl(
        s_0 
      + \sum_{t=1}^{T} 
            a_t \, y_t 
                \exp\!\bigl(-r (t-1)\Delta t \bigr)
  \Biggr).
\end{equation}
Note that when we choose actions from a compact set (which is not restrictive in applications) then Assumptions 1-5 are satisfied on the good event $G$.

\paragraph{Target Distribution}

For our target distribution, as in the LQ example, we generate terminal-wealth samples from \eqref{eq:terminal_wealth} under fixed, known investment policies.
We distinguish the following two benchmark scenarios.

\begin{enumerate}
	
	\item \textbf{100\% investment in the stock ($\pi_t=1$).}  
	Here $a_t=s_t$ for all $t$, so the entire wealth is invested in the risky asset at every time step.
	The wealth process follows a multiplicative recursion and can be written as
	\[
	s_T
	= s_0 \prod_{t=1}^T \exp\!\bigl((\mu-\tfrac12\sigma^2)\Delta t
	+ \sigma \sqrt{\Delta t}\,\varepsilon_t \bigr).
	\]
	Consequently, $\log s_T$ is normally distributed and the terminal wealth $s_T$ follows a log-normal distribution,
	\[
	s_T \sim \mathrm{LN}\!\Bigl(
	\log s_0 + T(\mu-\tfrac12\sigma^2)\Delta t,
	\; T\sigma^2\Delta t
	\Bigr).
	\]

	\item \textbf{Random investment fraction ($\pi_t\sim U(0,1)$).}  
	In this case, the investment proportions $\pi_t$ are sampled independently from the uniform distribution on $[0,1]$ and are independent of the market noise $\varepsilon_t$.
	
	The induced distribution of $s_T$ has no closed-form expression; however, we can approximate	this distribution  numerically via Monte Carlo simulation. 
\end{enumerate}
  
\begin{remark}
\label{re:terminal_wealth}
    Because the terminal wealth depends linearly on the control and multiplicatively on market noise,  randomizing the investment fraction adds negligible variance compared to market fluctuations, making a uniform policy and its mean indistinguishable at the level of the terminal distribution.
\end{remark}

We discretize the horizon $T=1$ into $h=20$ rebalancing dates, hence $\Delta t = T/h=0.05$.
For the Black--Scholes dynamics we use initial wealth $s_0=100$, risk-free rate $r=0.02$, stock drift $\mu=0.06$ and volatility $\sigma=0.40$.
In each epoch, we simulate a Monte Carlo batch of $M=10^5$ independent trajectories by sampling i.i.d.\ standard Gaussians
$\varepsilon_{t+1}\sim\mathcal N(0,1)$ and iterating the recursion for wealth under the current policy.

\paragraph{Numerical Results and Discussion} 

For the first  case we see in Figure \ref{fig:wealth_dist_case_2} and \ref{fig:cf_dist_wealth_case_2}  the histogram of the wealth distributions (target and learned) as well as the real and imaginary parts of the target and learned characteristic functions. Both plots show good fits.  From Figure  \ref{fig:optimal_action_dist_wealth_case_2} it is clear that the optimal action distribution produced by the algorithm is almost deterministic and close to the theoretical one. 

For the second case, we plot first in Figure \ref{fig:wealth_dist_case_3} the comparison of the target density (blue) and the density of the terminal wealth learned by the algorithm (red). In Figure \ref{fig:cf_dist_wealth_case_3} we compare the real and imaginary part of the target and learned characteristic functions. Both plots show a good fit.
The optimal action distribution shown in Figure \ref{fig:optimal_action_dist_wealth_case_3} is, as expected from Remark \ref{re:terminal_wealth}, close to the mean of the uniform $(0,1)$ and almost deterministic, i.e. with a very small variance. This discrepancy could be due to finite-precision error of the calculation. We investigate further the accuracy of the approximation in the next example where we have an exact, analytically traceable, unique distribution for the optimal controls.    

\textbf{Case 1) $\pi_t=1$}\\

\begin{figure}[H]
    \centering
    \begin{subfigure}[b]{0.45\textwidth}
        \includegraphics[width=\textwidth]{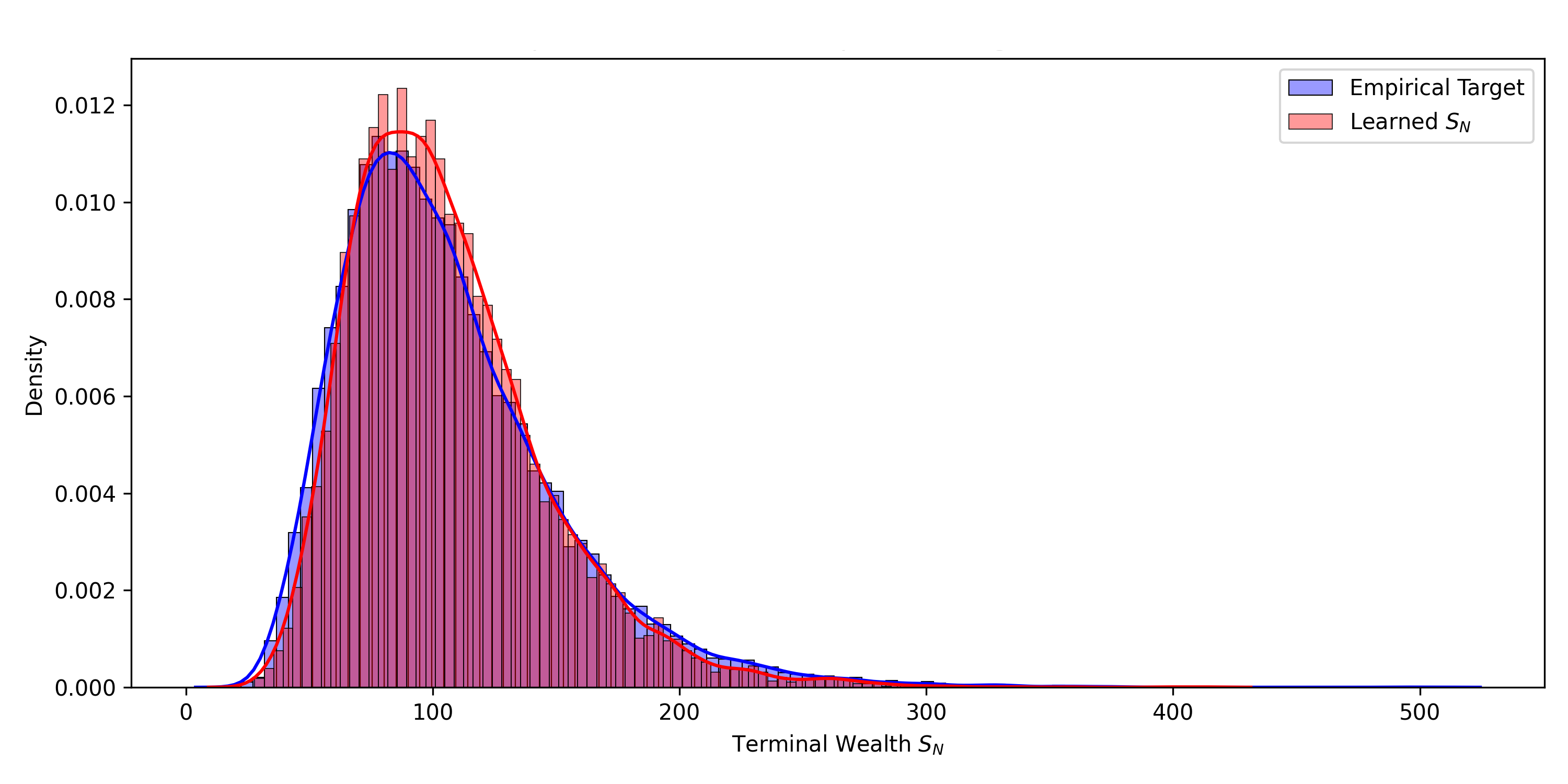}
        \caption{Histogram}
        \label{fig:wealth_dist_case_2}
    \end{subfigure}
    ~ 
    \begin{subfigure}[b]{0.4\textwidth}
        \includegraphics[width=\textwidth]{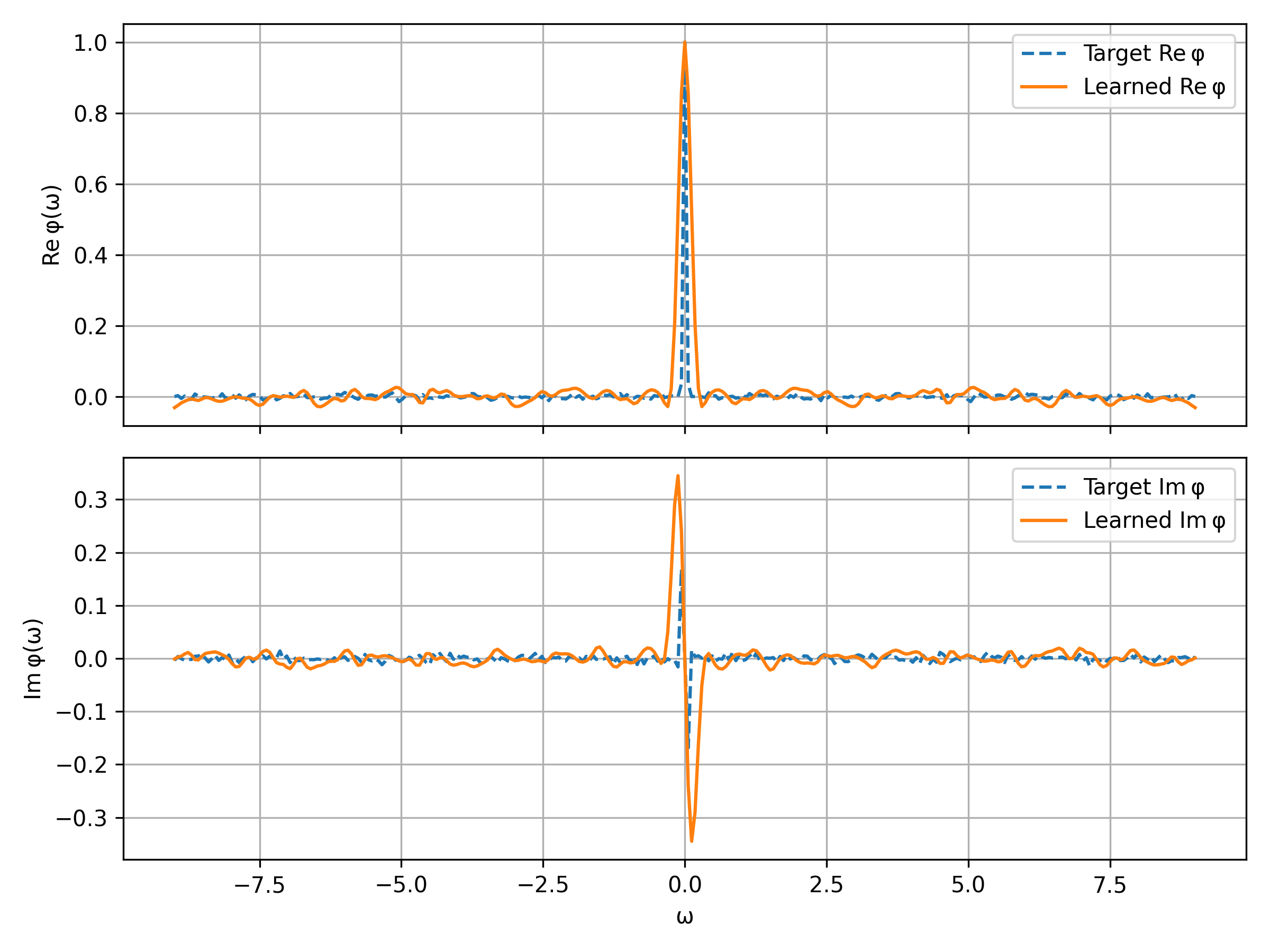}
        \caption{Characteristic Functions}
        \label{fig:cf_dist_wealth_case_2}
    \end{subfigure}
    \caption{Distribution of terminal wealth}
    \end{figure}

\begin{figure}[H]
	\centering
	\includegraphics[width=0.5\textwidth]{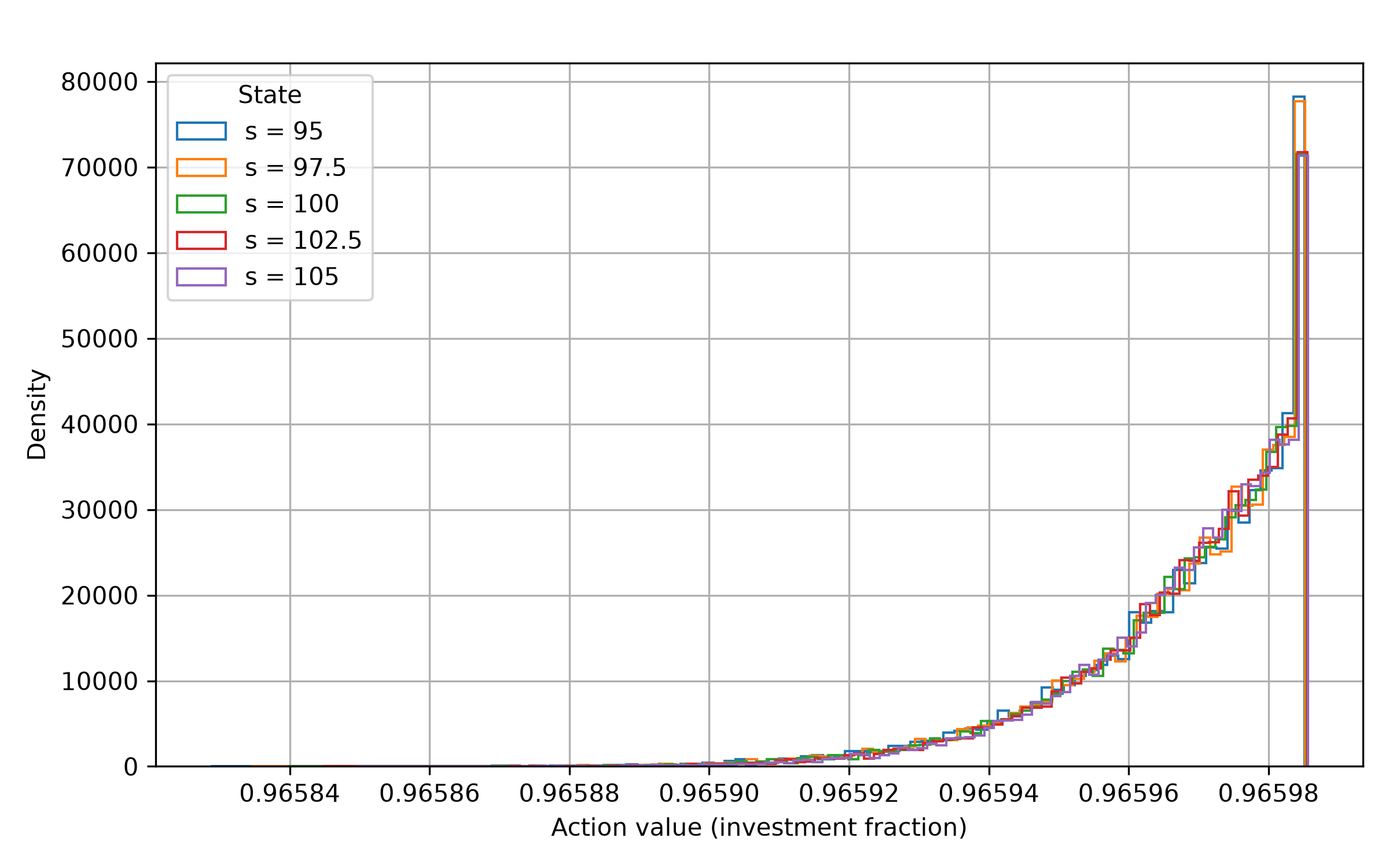}
	\caption{Optimal actions distributions for different wealth states}
	\label{fig:optimal_action_dist_wealth_case_2}
\end{figure}

\textbf{Case 2) $\pi_t\sim U(0,1) $}\\

\begin{figure}[H]
    \centering
    \begin{subfigure}[b]{0.45\textwidth}
        \includegraphics[width=\textwidth]{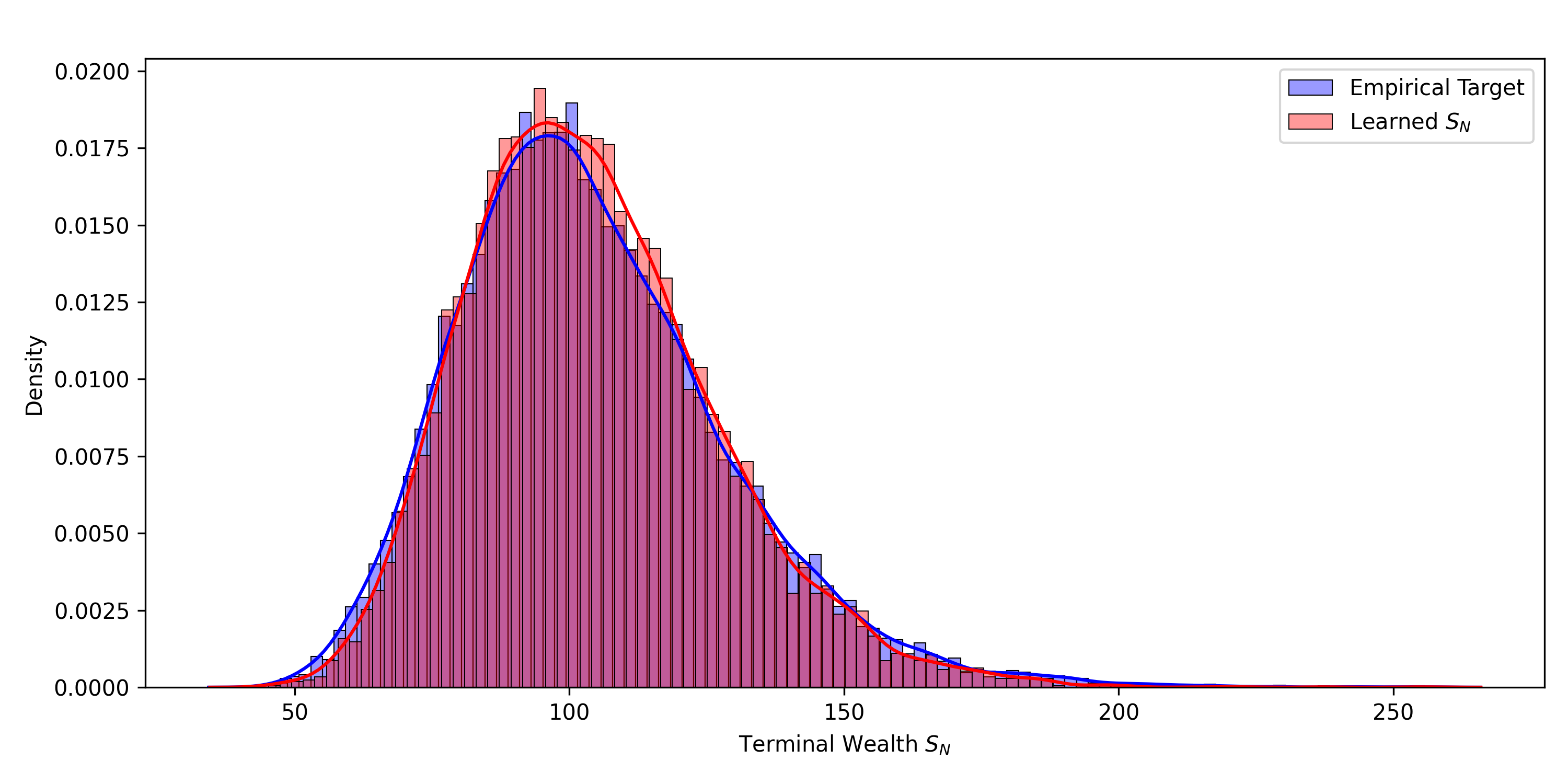}
        \caption{Histogram}
        \label{fig:wealth_dist_case_3}
    \end{subfigure}
    ~ 
    \begin{subfigure}[b]{0.4\textwidth}
        \includegraphics[width=\textwidth]{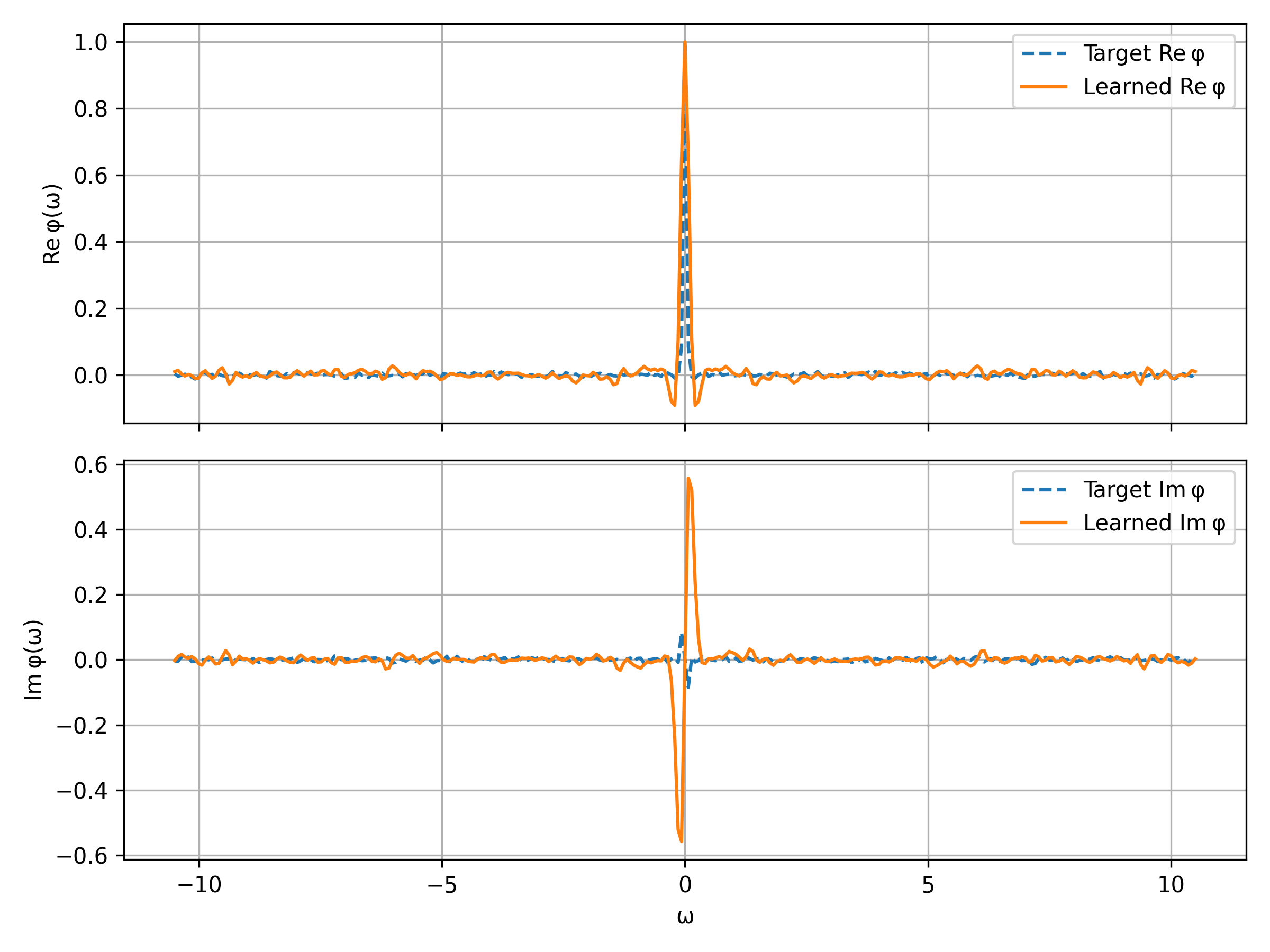}
        \caption{Characteristic Functions}
        \label{fig:cf_dist_wealth_case_3}
    \end{subfigure}
    \caption{Distribution of terminal wealth}
    \end{figure}

\begin{figure}[H]
	\centering
	\includegraphics[width=0.5\textwidth]{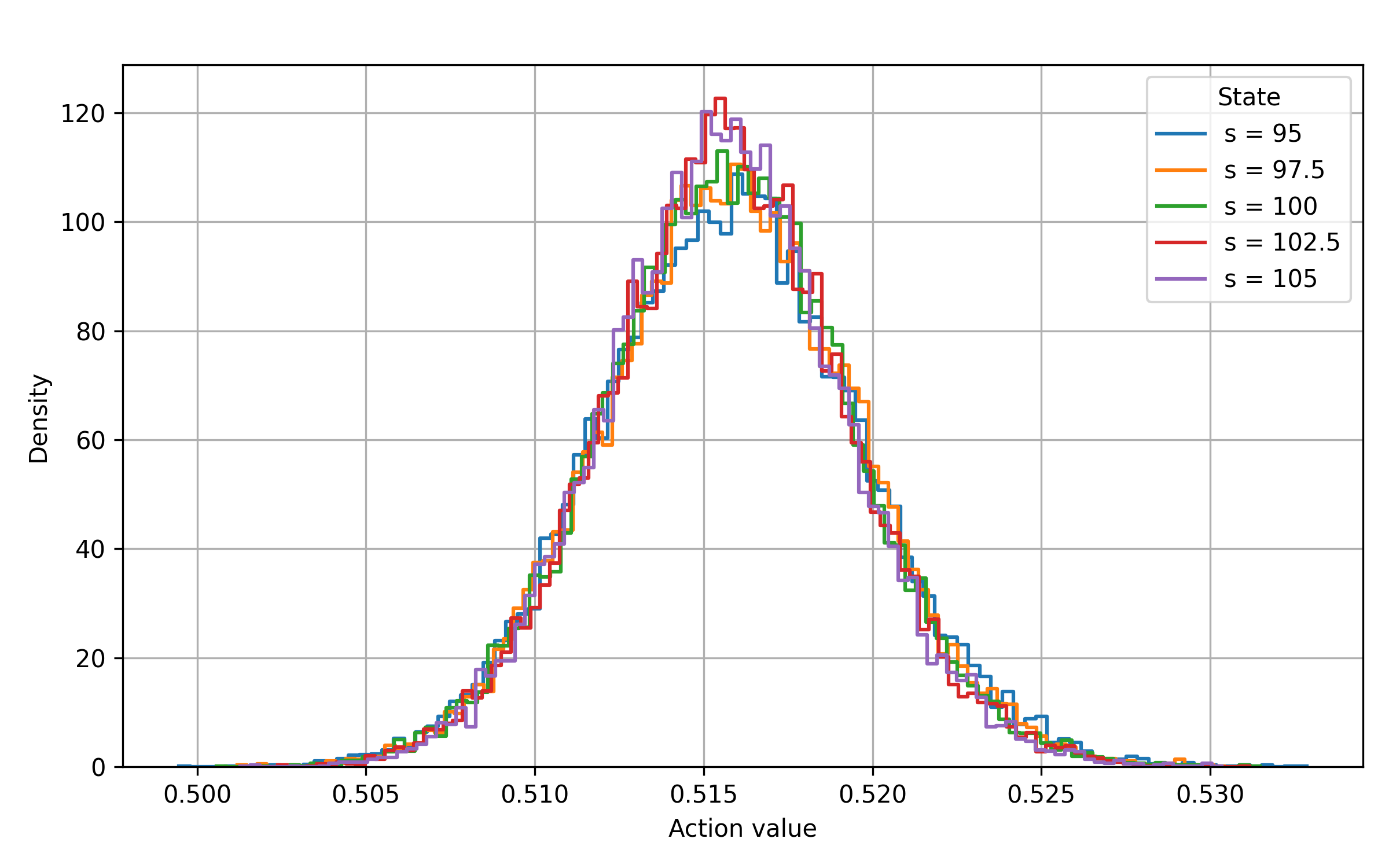}
	\caption{Optimal actions distributions for different wealth states}
	\label{fig:optimal_action_dist_wealth_case_3}
\end{figure}

To conclude this example we present the loss for the uniform case. Here the variance of the target distribution is much larger than in the LQ example of the previous section thus, to make the Wasserstein distance informative we compute the relative errors (i.e., Wasserstein divided by standard deviation) and compare them with those of the LQ which are of similar size.   

\begin{table}[htbp]
\centering
\small
\begin{tabular}{rrrrrc}
\toprule
Seed & k &
\shortstack{CF loss} &
\shortstack{$W_1$\\distance} &
\shortstack{Target s.d.} &
\shortstack{Relative\\error} \\
\midrule
LQ seed 2026 & 290 & $6.293\times10^{-6}$ & 0.001166 & 0.04 & 2.73\% \\
BS-uniform 2026 & 399 & $0.00183511$ & 0.443101 & 24.4945 & 1.81\% \\
\bottomrule
\end{tabular}
\caption{
Loss and Wasserstein distances}
\label{tab:BS_error}
\end{table}

\subsection{An Example with Compact Support}
\label{subse:compact_support}
This example is related to the first one. Our goal is to exemplify a rather unique phenomenon. Due to the structure of the cost function we can identify a unique control distribution that steers the discounted rewards toward the target. In contrast to Example \ref{subse:wealth_example}, here the control enters a nonlinear function and, together with the noise in the dynamics, shapes the reward distribution.

The policy gradient can indeed recover the unique control distribution when the length of the action interval permits unique identifiability of policies.  Otherwise the unique control can be found by enforcing the symmetry of the control distribution a posteriori. 

Concretely, we consider linear dynamics with $T=1$, a randomized (relaxed) control $a_0$ and additive Gaussian noise:
\begin{equation}
\label{eq:dynamics2}
 s_1 = s_0 + a_0 + \varepsilon_1, \qquad \varepsilon_1 \sim \mathcal N(0,\sigma_\varepsilon^2),
\end{equation}
where $s_0\in\mathbb R$ is known and $a_0$ is a controlled random variable, independent of $\varepsilon_1$.
Define the bounded performance variables $r_0\equiv 0$ and
\begin{equation}
\label{eq:r-def}
 r_1(s_1,a_1) := \cos(s_1) \in [-1,1],
\end{equation}
thus, we are interested in the law of $R_2=\cos(s_0+a_0+\varepsilon_1).$ Obviously we can restrict $a_0$ to a compact set $I$ and Assumptions 1-5 are satisfied.
\paragraph{Target Distribution}
Our \emph{target} distribution is the Epanechnikov law on $[-1,1]$ with density
\begin{equation}
\label{eq:epa-density}
 f_E(x) = \frac{3}{4}\Bigl(1 - x^2\Bigr)\,1_{\{|x|\le 1\}},
\end{equation}
its characteristic function admits the closed form
\begin{equation}
\label{eq:phiE}
 \phi_E(u) = \frac{3}{u^3}\Bigl(\sin(u ) - u \cos(u)\Bigr), \qquad \phi_E(0)=1.
\end{equation}

We seek a distribution of the randomized control $a_0$ such that the distribution of $R_2$ matches the Epanechnikov law. 
We assume that $a_0$ takes values in a compact interval $I$ of length $|I|<\pi$, so that its law is uniquely determined by its Fourier coefficients, $\psi_A(k):=\EE\big[e^{ik a_0}\big]$ on $\mathbb Z$. The characteristic function of the distribution of $a_0$ reads as
\begin{equation}
\phi_A(t) = \int_I e^{itx}p_{a_0}(x)\,dx = \sum_{k=-\infty}^{\infty} \psi_A(k)\,\int_I e^{i(t-k)x}\,dx.
\end{equation}
To benchmark  our algorithm, we solve the optimization problem analytically up to a small approximation error and obtain a reference solution.

\paragraph{The matching condition}
With the help of the Jacobi--Anger identity we can compute the characteristic function of $R_2$ as follows:
\begin{align*}
\phi_{R_2}(u) &= \EE[e^{iu R_2}] = \sum_{k\in\mathbb Z} i^k J_k(u )\EE[e^{iks_1}] = \sum_{k\in\mathbb Z} i^k J_k(u  )e^{iks_0}e^{-\frac12 k^2\sigma^2}\psi_A(k)
\end{align*}
where $J_k$ is the $k$-th Bessel function of the first kind. Thus, the matching equation for the characteristic functions is
\begin{equation}
\label{eq:matching-functional}
 \sum_{k\in\mathbb Z} i^kJ_k(u )e^{iks_0}e^{-\tfrac12 k^2\sigma^2}\psi_A(k) = \phi_E(u), \quad \forall u\in\mathbb R.
\end{equation}
We solve this equation approximately.

\paragraph{Truncation and linear system}
Since $J_k$ and $e^{-k^2\sigma^2/2}$ decay rapidly, we truncate to $|k|\le K$ and select a finite frequency grid $\{u_\ell\}_{\ell=1}^L\subset[-K,K]$. The truncated system is
\begin{equation}
\label{eq:trancated_linear_system}
 J_0(u_\ell ) + 2\sum_{k=1}^{K} J_k(u_\ell )e^{-\frac12k^2\sigma^2}x_k = \phi_E(u_\ell), \quad x_k=Re(i^ke^{iks_0}\psi_A(k)).
\end{equation}
Note that the symmetry $\psi_A(-k) = \overline{\psi_A(k)}$ and $J_{-k} = (-1)^k J_k$ makes the imaginary parts cancel. Solving for $x=(x_1,\dots,x_K)$ yields a least-squares solution (or an exact solution). In addition, both left and right-hand sides are even, forcing $x_k=0$ for all odd $k$. Then we reconstruct
\begin{equation}
 \psi_A(k) = i^{-k}e^{-iks_0}x_k, \qquad |\psi_A(k)|\le 1,\quad \psi_A(-k)=\overline{\psi_A(k)}.
\end{equation}

The density on $I$ (of width $<\pi$) follows by Fourier inversion:
\begin{equation}
 p_{a_0}(x) = \frac{1}{2\pi}\Big(1+2\sum_{k=1}^{K}Re(\psi_A(k)e^{-ikx})\Big),\quad x\in I.
\end{equation}

Finally, we notice that $p_{a_0}$ is $\pi$-periodic since all odd Fourier modes $\psi_A(k)$ are zero.

\begin{remark}
	In this example  the notion of matching is approximate rather than exact, unlike in  the previous examples. First, the problem of constructing a probability distribution when marginals satisfy certain conditions in the spirit of Equation \eqref{eq:matching-functional}, is a classical one; see, for example \cite{truncated_matrix_triginometric_problem} and the references therein. Second, the notion of uniqueness should also be understood only up to a certain number of modes related to the truncation level of \eqref{eq:trancated_linear_system}.    
\end{remark}

\paragraph{Numerical Results and Discussion}
First, for the numerical solution of \eqref{eq:trancated_linear_system} we fix $K=16$ and $L>>K$, in particular, $L=8001$ (for increased accuracy in our policy gradient since we used the same grid), and thus implement a least-squares solution. If we write the right-hand side  of \eqref{eq:trancated_linear_system} in matrix form we need the corresponding matrix to have full rank -- a condition easily verifiable in our case -- in order to have a unique solution. As in the previous examples we report in Appendix \ref{appx:numerical_examples} the complete hyperparameters used for our results.    

We run two sets of experiments, one on the interval $I_1=[0,\pi]$ to have uniquely identifiable actions, and one on the interval $I_2=[-\pi,\pi]$ to try and identify uniquely the optimal policy enforcing the symmetry of the distribution.

\noindent
\textbf{Case $I_1=[0,\pi]$}\\

In Table \ref{tab:compact_lossx} we report the loss and the Wasserstein distance for the reference solution and the algorithm, as well as the distance of the action distribution returned by the algorithm to the one obtained by the reference. 

\begin{table}[htbp]
\centering
\small
\begin{tabular}{rrrrrc}
\toprule
 & k &
\shortstack{CF loss} &
\shortstack{$W_1$\\distance} &
\shortstack{$W_1$-dist from\\reference} \\
\midrule
Reference solution & - & $8.6324\times10^{-4}$ & 0.001166 & - \\
Learned solution & 799 & $1.2177\times10^{-3}$ & 0.00390521 & 0.0229202  \\
\bottomrule
\end{tabular}
\caption{
Loss and Wasserstein distances, $I_1=[0,\pi]$}
\label{tab:compact_lossx}
\end{table}

\begin{figure}[H]
    \centering
    \begin{subfigure}[b]{0.45\textwidth}
        \includegraphics[width=\textwidth]{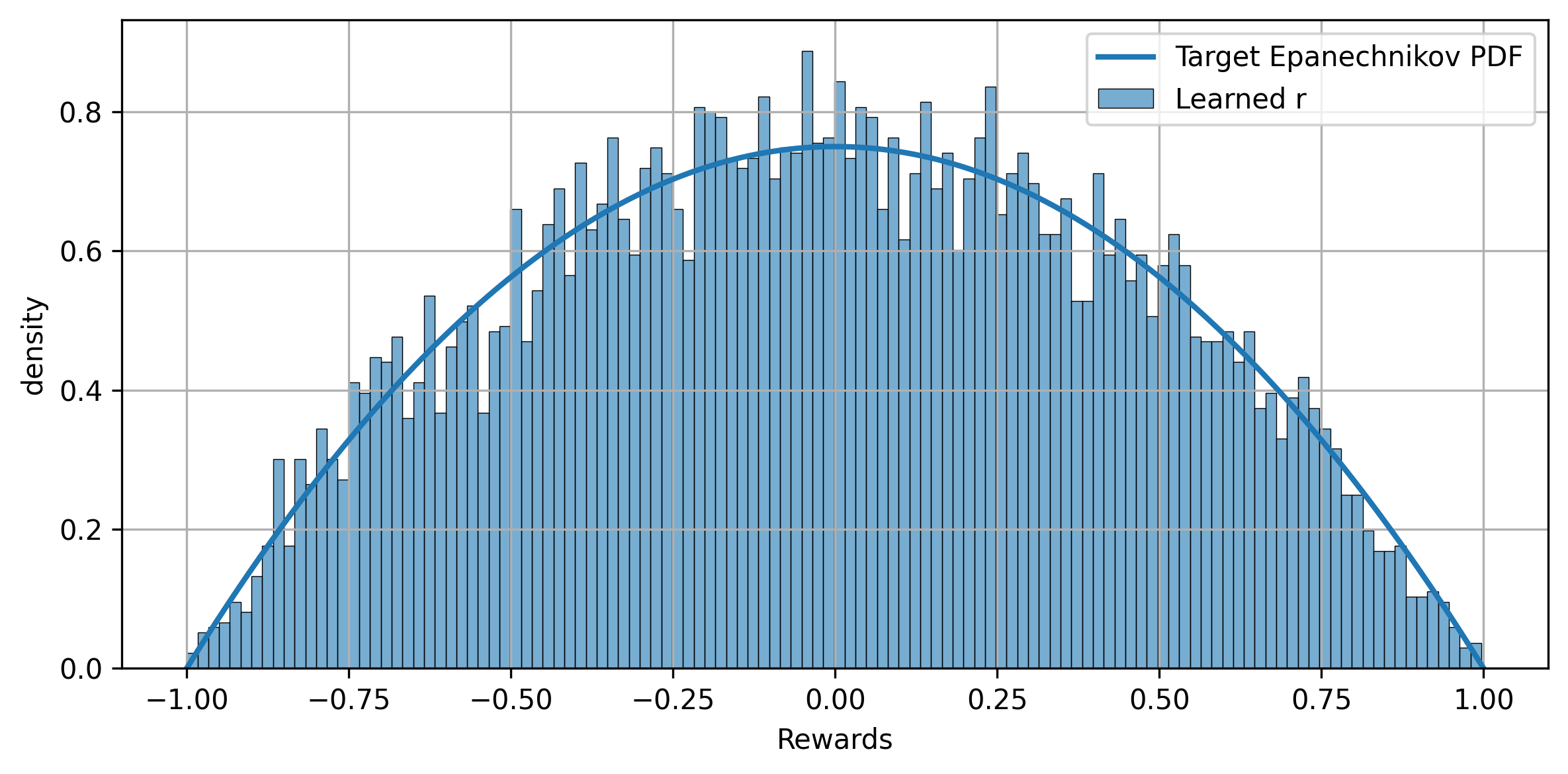}
        \caption{Histogram}
        \label{fig:return_dist2-3}
    \end{subfigure}
    ~ 
    \begin{subfigure}[b]{0.4\textwidth}
        \includegraphics[width=\textwidth]{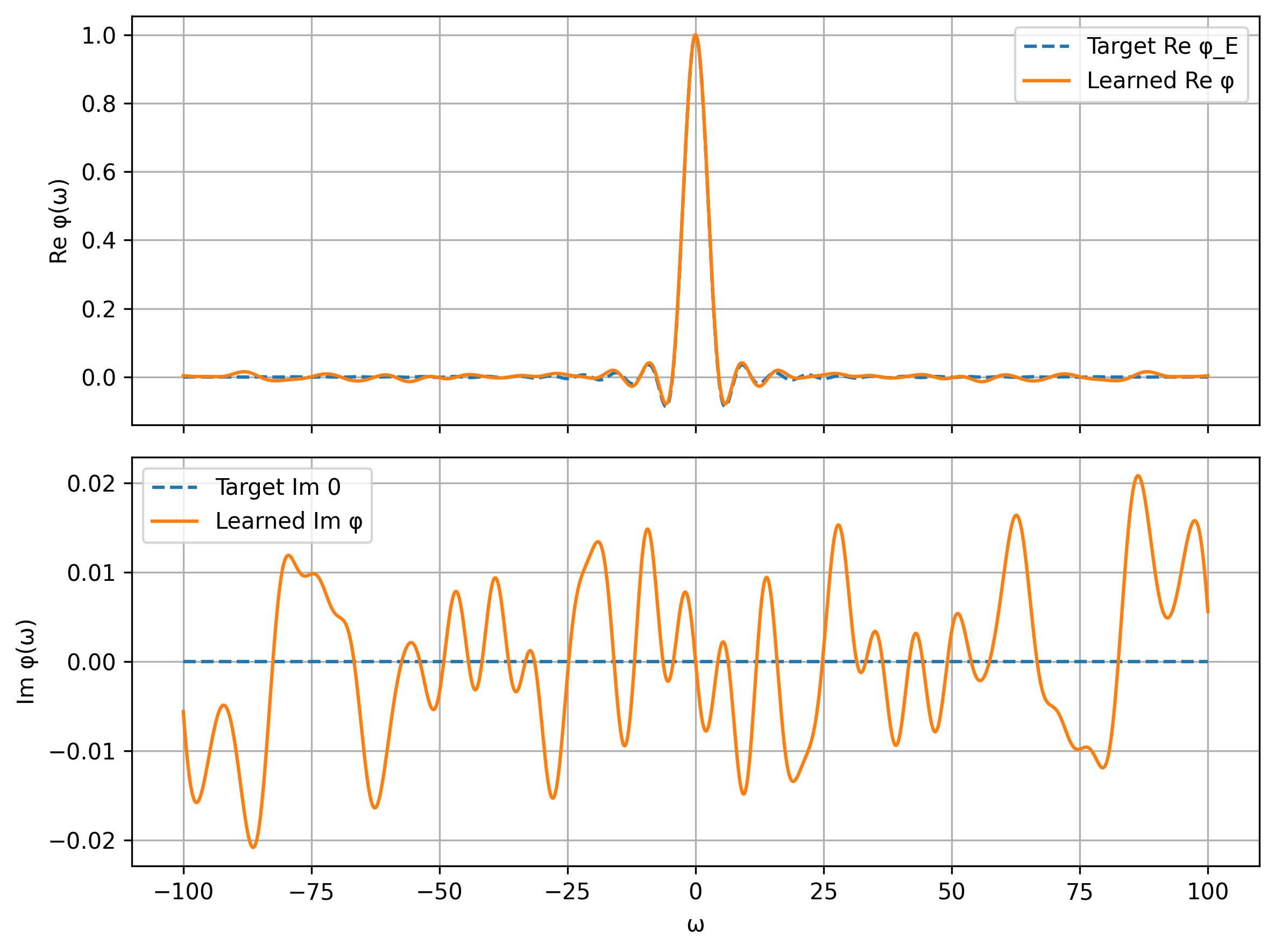}
        \caption{Characteristic Function }
        \label{fig:return_cf2-2}
    \end{subfigure}
    \caption{Distribution of Rewards}
\end{figure}

\begin{figure}[H]
    \centering
    \begin{subfigure}[b]{0.45\textwidth}
        \includegraphics[width=\textwidth]{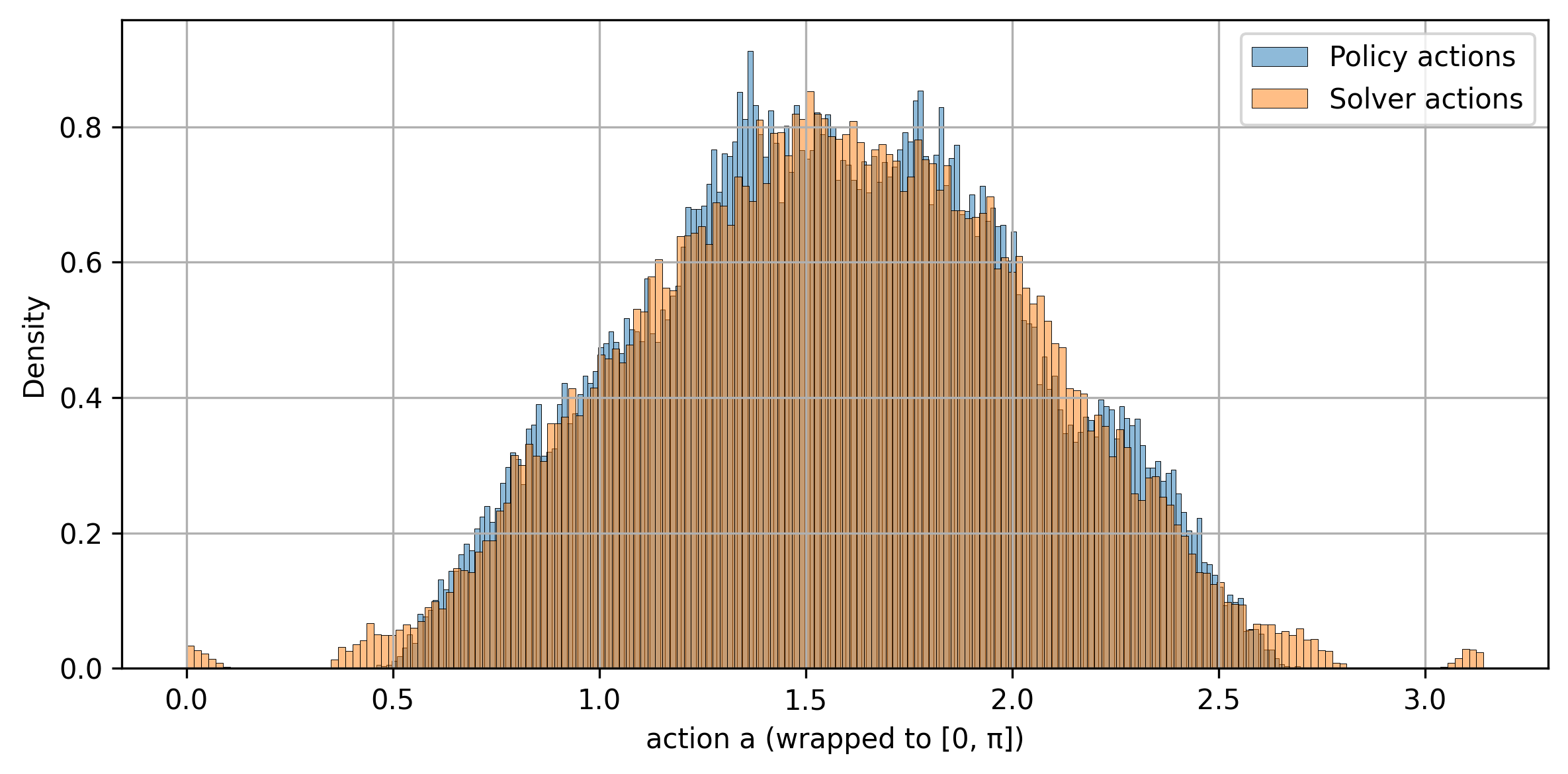}
        \caption{Histogram}
        \label{fig:action_dist2}
    \end{subfigure}
    ~ 
    \begin{subfigure}[b]{0.4\textwidth}
        \includegraphics[width=\textwidth]{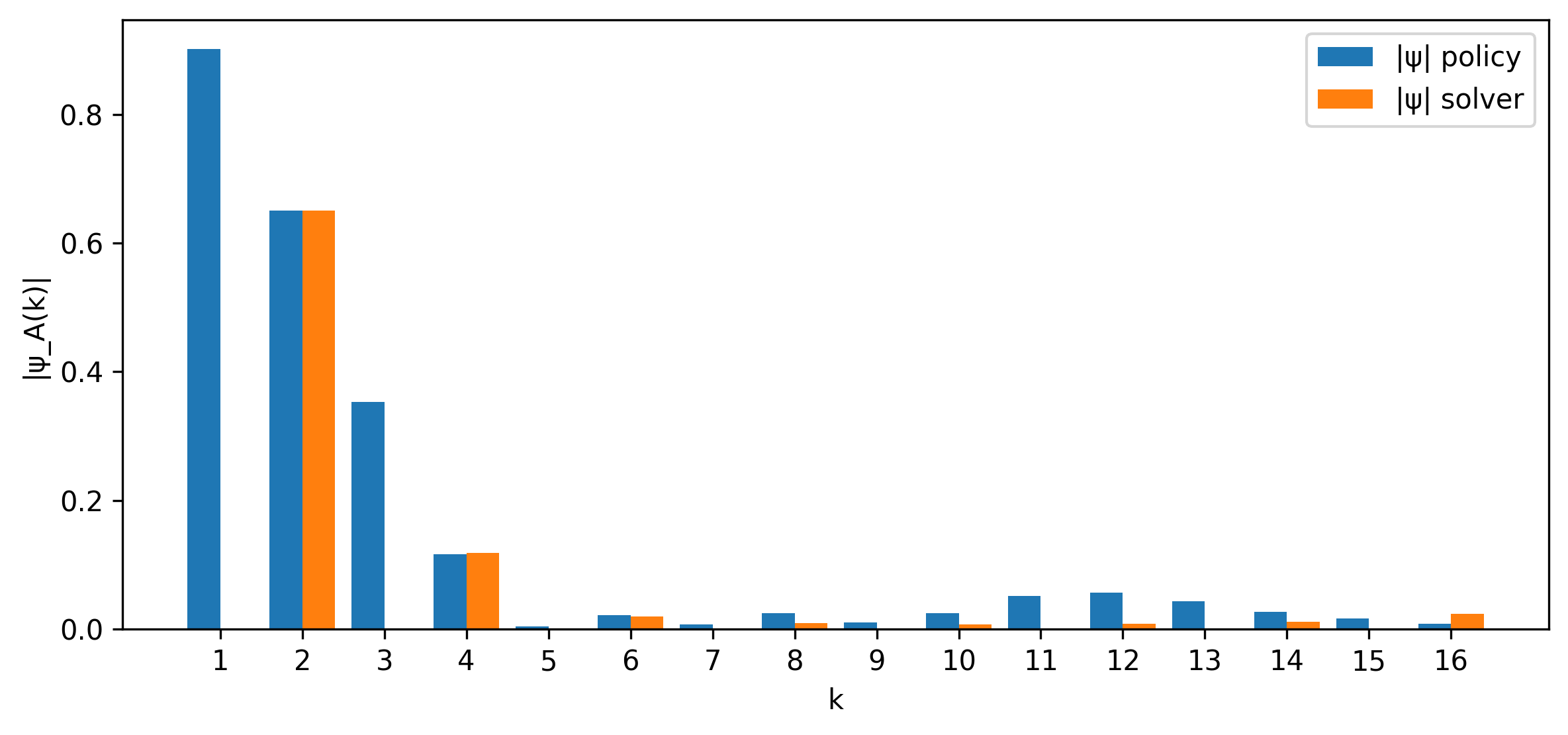}
        \caption{Fourier modes $\psi_A(k)$}
        \label{fig:action_fourier_2}
    \end{subfigure}
    \caption{Action distribution without projecting}
\end{figure}
First note that the density of the learned distribution is rather close to the target distribution, cf. Figure \ref{fig:return_dist2-3}. In Figure \ref{fig:return_cf2-2} we compare the real and imaginary parts of the learned and target characteristic function. Since we restrict to $I_1=[0,\pi]$, the imaginary part is not zero in the learned case, as it is in the target case.
However, here the policy gradient algorithm can indeed find the unique action distribution. One may notice that in Figure \ref{fig:action_fourier_2} the odd modes are not zero as they should be. This is due to the fact that we have a $\pi$-periodic density and we have samples only on one side of the density.  To reveal the true symmetric density we need to consider  the interval $I_2=[-\pi,\pi]$ as we do in the next experiment.

\noindent
\textbf{Case $I_2=[-\pi,\pi]$}\\

\begin{table}[htbp]
\centering
\small
\begin{tabular}{rrrrrc}
\toprule
 & k &
\shortstack{CF loss} &
\shortstack{$W_1$\\distance} &
\shortstack{$W_1$-dist from\\reference} \\
\midrule
Reference solution & - & $8.6324\times10^{-4}$ & 0.001166 & - \\
Learned solution & 799 & $0.00121705$ & 0.00459225 & 0.01526  \\
\bottomrule
\end{tabular}
\caption{
Loss and Wasserstein distances after projection, $I_2=[-\pi,\pi]$}
\label{tab:compact_loss}
\end{table}

Here the solver solution for \eqref{eq:trancated_linear_system} gives two periods of the $\pi$-periodic density. The policy gradient is not capable to identify correctly the complete periodic density since for the reward there is no way to differentiate between periodic actions and the computed policy stays concentrated on  one of the two branches of the density. Since the algorithm is stochastic the outcome may vary from run to run, meaning that sometimes the left lobe is selected and other times the right, cf. Figures \ref{fig:return_dist2-4} and \ref{fig:return_cf2-3}  for the densities and characteristic functions of the rewards and Figures \ref{fig:action_dist2-2} and \ref{fig:action_fourier} for the action densities and Fourier modes of the action distribution.

\begin{figure}[H]
    \centering
    \begin{subfigure}[b]{0.45\textwidth}
        \includegraphics[width=\textwidth]{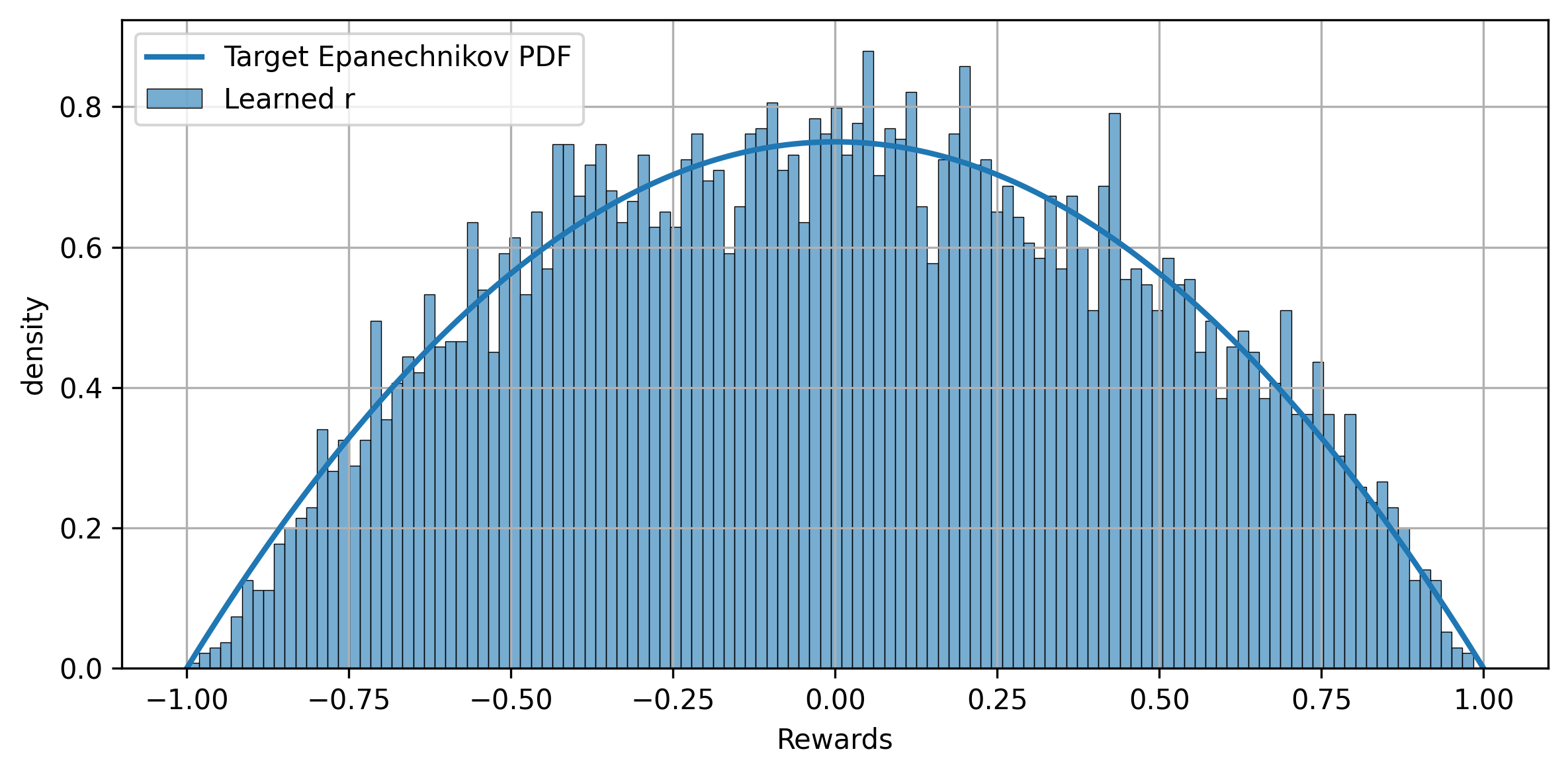}
        \caption{Histogram}
        \label{fig:return_dist2-4}
    \end{subfigure}
    ~ 
    \begin{subfigure}[b]{0.4\textwidth}
        \includegraphics[width=\textwidth]{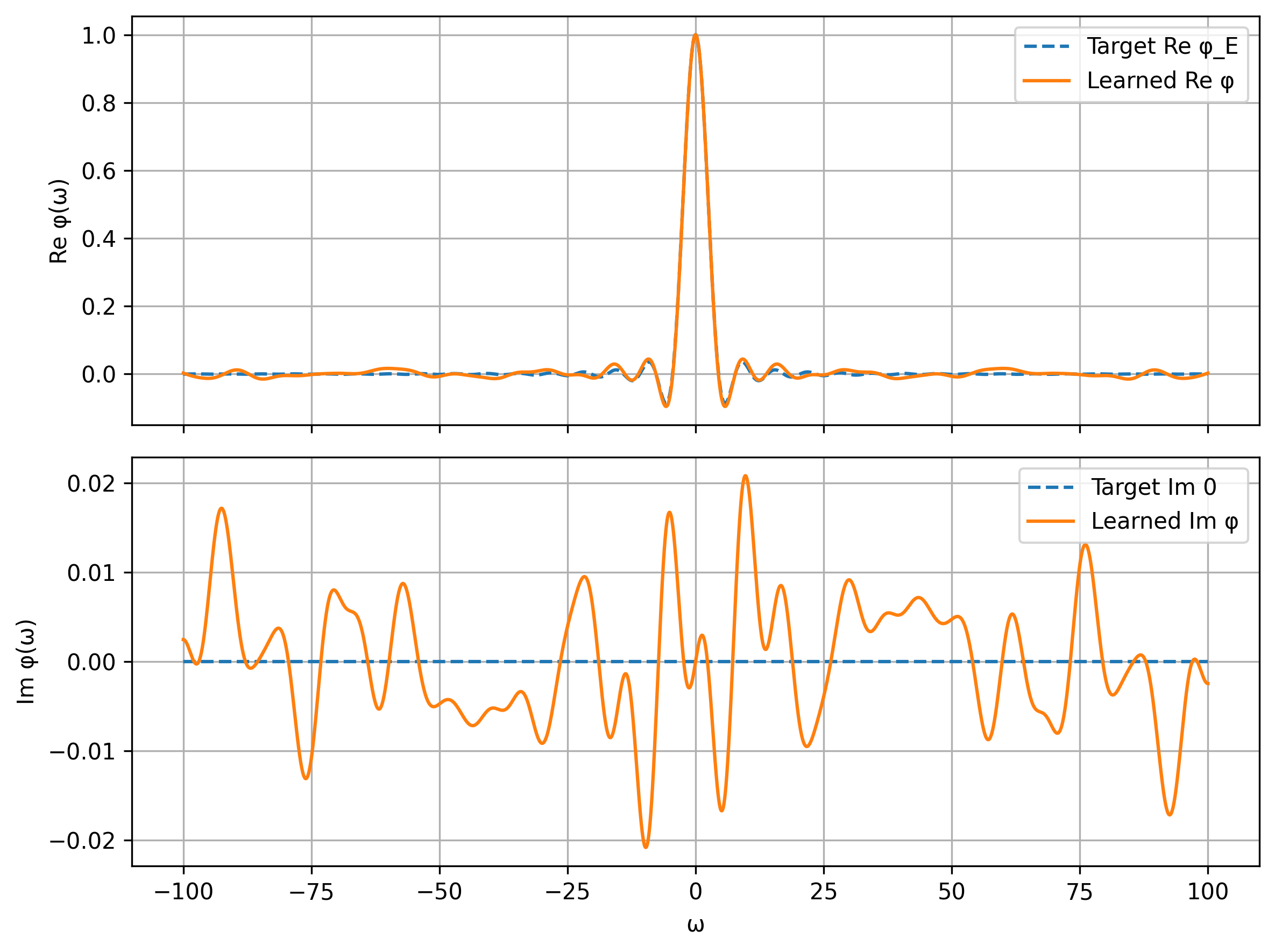}
        \caption{Characteristic Function }
        \label{fig:return_cf2-3}
    \end{subfigure}
    \caption{Distribution of Rewards}
    \end{figure}

\begin{figure}[H]
    \centering
    \begin{subfigure}[b]{0.45\textwidth}
        \includegraphics[width=\textwidth]{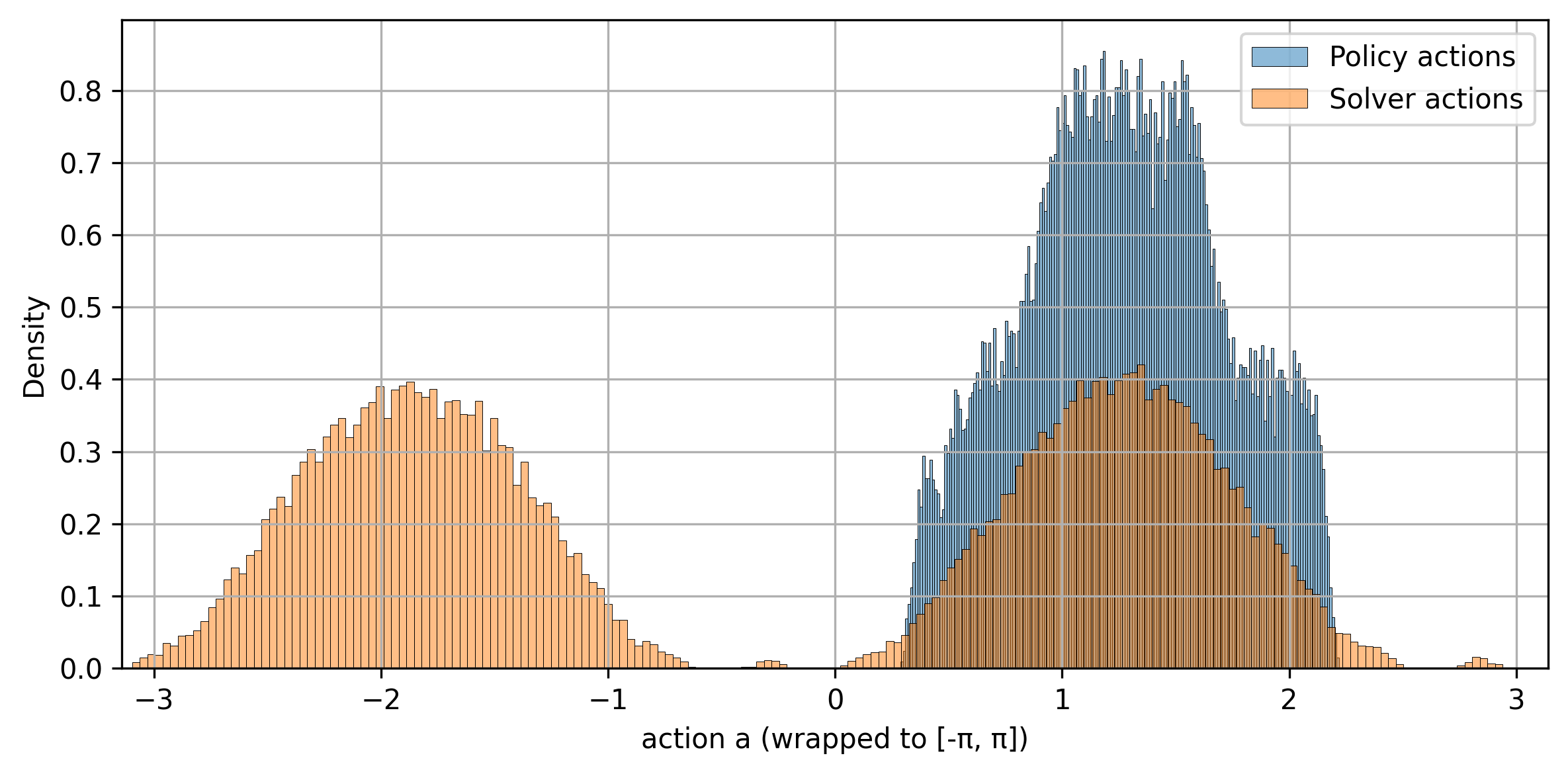}
        \caption{Histogram}
        \label{fig:action_dist2-2}
    \end{subfigure}
    ~ 
    \begin{subfigure}[b]{0.4\textwidth}
        \includegraphics[width=\textwidth]{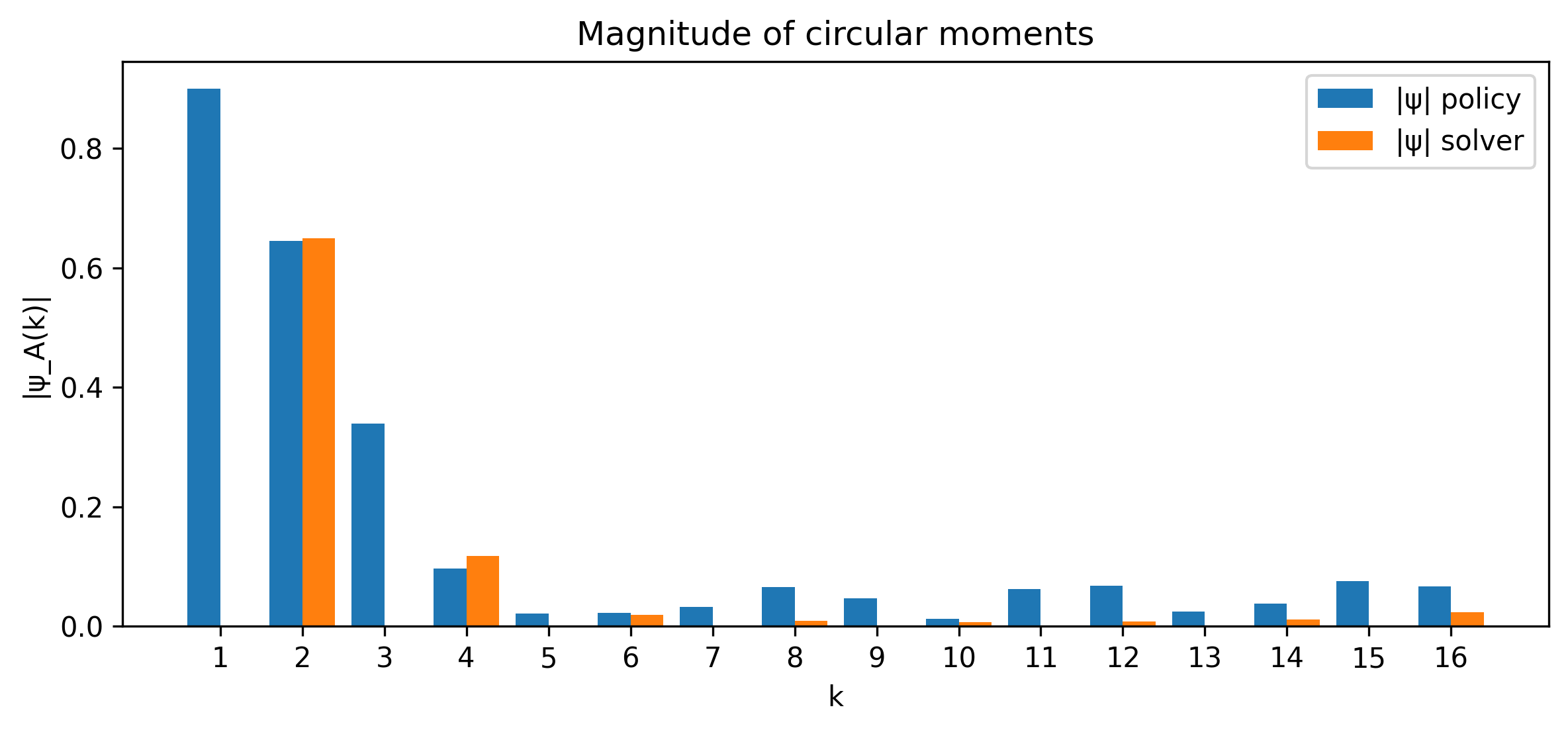}
        \caption{Fourier modes $\psi_A(k)$}
        \label{fig:action_fourier}
    \end{subfigure}
    \caption{Action distribution without projecting}
\end{figure}

Nevertheless, if we split and project the samples equally to the opposite interval after training we can recover the correct periodic density from the learned density in Figure \ref{fig:action_dist2}, see Figures \ref{fig:gull} and \ref{fig:tiger}.

\begin{figure}[H]
    \centering
    \begin{subfigure}[b]{0.45\textwidth}
        \includegraphics[width=\textwidth]{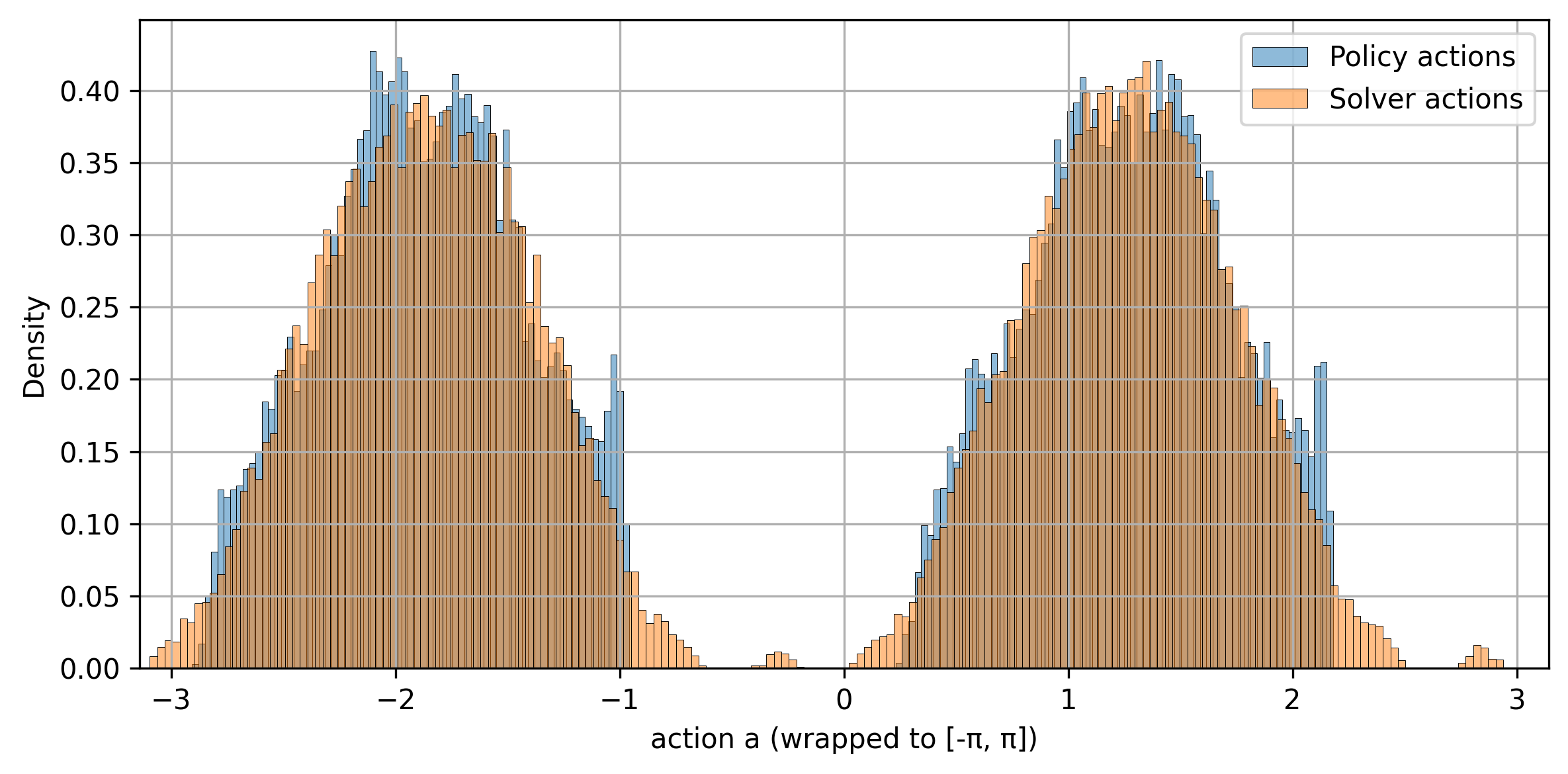}
        \caption{Histogram}
        \label{fig:gull}
    \end{subfigure}
    ~ 
    \begin{subfigure}[b]{0.4\textwidth}
        \includegraphics[width=\textwidth]{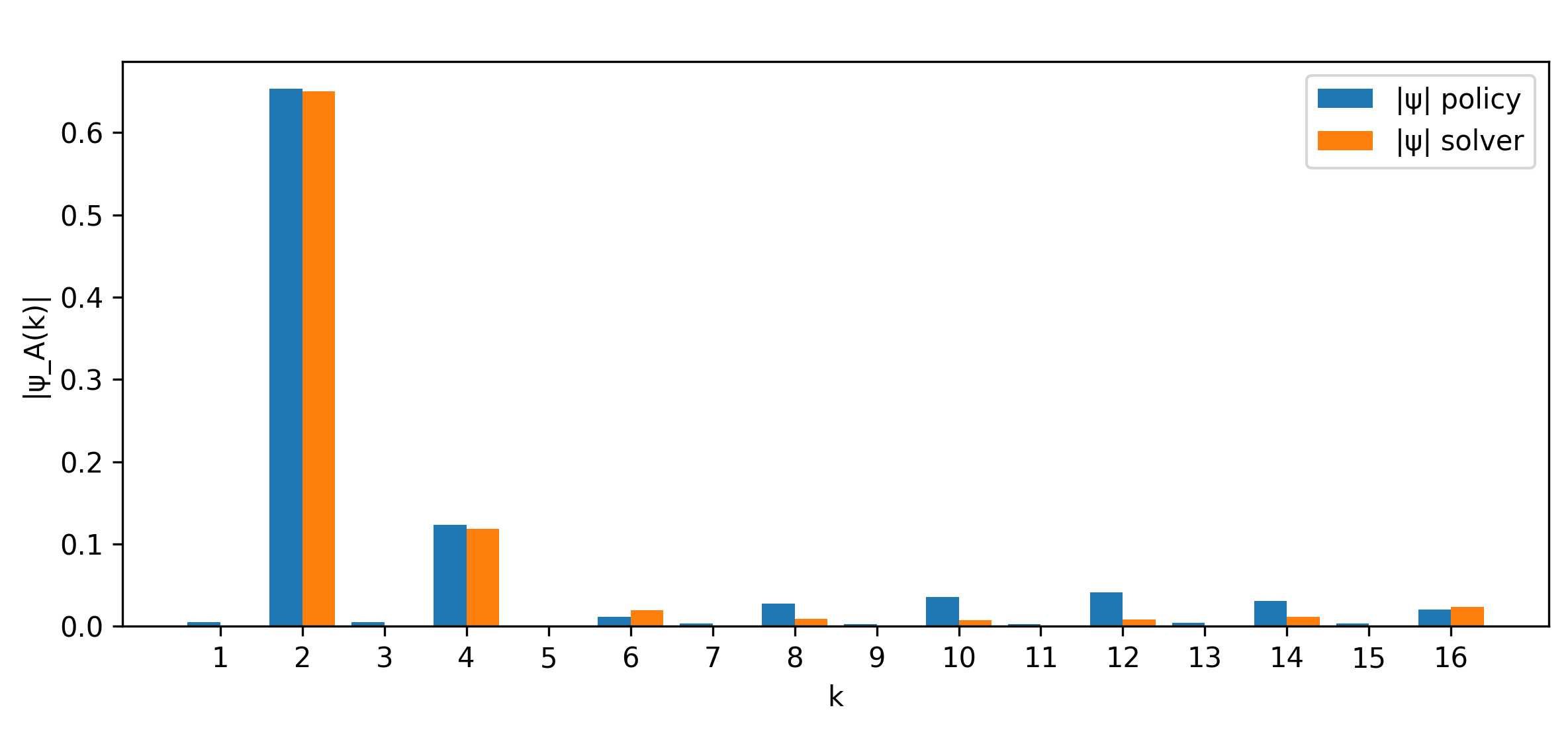}
        \caption{Fourier modes}
        \label{fig:tiger}
    \end{subfigure}
    \caption{Action distribution with projection}
\end{figure}

\begin{remark}[about the fit of the action distributions]
	To test the best possible fit for the action distribution we augment the loss function with an additional matching objective $\sum_{j=0}^{K=16}\vert \psi_A(j)- \hat\psi_A(j)\vert^2 w(u_\ell) $ for the Fourier modes that come from the numerical solution of \eqref{eq:trancated_linear_system}. As one can see in Figure \ref{fig:Policy_vs_solver} the periodic densities are very close and thus the  ANN is sufficiently expressive. We do not attempt here an exhaustive optimization of the architectures and the hyperparameters since the goal is a proof of concept. 
		\begin{figure}[!h]
		\centering 
		\includegraphics[width=0.8\textwidth]{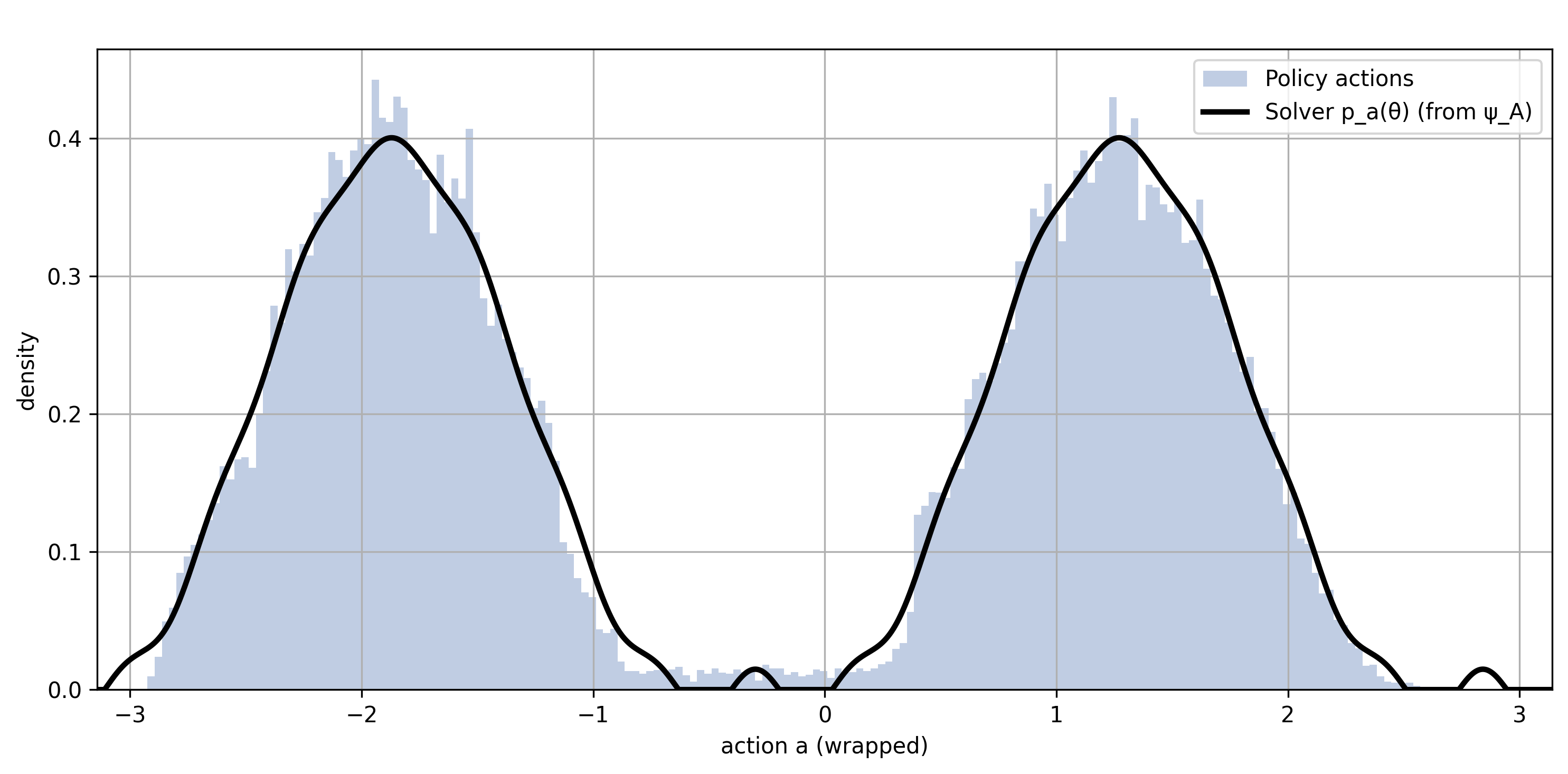}
\caption{Comparison of policy distributions }
		\label{fig:Policy_vs_solver}
	\end{figure}
\end{remark}

\subsection{Writing a Classical MDP as a Special Case of a .Distributional MDP}
\label{se:forzen_lake}
We give another example, the standard \texttt{FrozenLake} MDP from \cite{towers2025gymnasiumstandardinterfacereinforcement}, a pleasant variant of \cite[Example 3.5]{SuttonBarto}. 
The episodic return of the MDP satisfies
\[
R_T \in \{0,1\},
\qquad
R_T = 1 \text{ if the goal is reached,}
\quad
R_T = 0 \text{ otherwise}.
\]
Let
\[
p_\sigma := \mathbb{P}^\sigma(R_T = 1)
\]
denote the success probability under policy $\sigma$. Then
\[
\mathrm{Law}^\sigma(R_T)
    = p_\sigma\,\delta_1 + (1-p_\sigma)\,\delta_0.
\]
Choose the target distribution as the ideal success law,
$\mu_{\mathrm{target}} = \delta_1.$
If we consider the squared $L^2$ distance between characteristic functions,
\[
\cL(\sigma)
:= \int_{\mathbb{R}}
\big|\varphi^\star(u)-\EE^\sigma[e^{iuR_T^\theta}] \big|^2
\, w(u)\, du,
\]
where $w$ is a nonnegative weight function. A direct computation yields
\[
\cL(\sigma)
= (1 - p_\sigma)^2
\int_{\mathbb{R}} |e^{iu} - 1|^2 w(u)\, du
=: C_w (1 - p_\sigma)^2,
\]
with $C_w \in (0,\infty)$ independent of $\sigma$.

Thus, minimizing the characteristic function loss is  equivalent to
\[
\min_\sigma \cL(\sigma)
\quad \Longleftrightarrow \quad
\max_\sigma p_\sigma
\quad \Longleftrightarrow \quad
\max_\sigma \mathbb{E}_\sigma[R_T].
\]

Of course, this theoretical framework has to be adapted to design practical algorithms since matching a Dirac measure is a difficult problem and also the issue of exploration has to be addressed. 

\section{Acknowledgment}{This work is supported by the Helmholtz Association Initiative and Networking Fund on the HAICORE@KIT partition. 
Apart from this, the authors received no financial support for the research, authorship, and/or publication of this article and have no competing interests to declare.

A.V. and N.B. would like to thank the anonymous referees for their comments, which helped improve the manuscript.   

A.V. would like to thank Florian D\"ottling, Harley Wiltzer and the participants of the conference "Stochastic Models and Control 2026" for the fruitful discussions at various stages of this project. 
}

\section{Appendix}
In what follows we use the $L^1$-norm as our default norm i.e., for $x\in\RR^n$ we have $\|x\|=|x_1|+\ldots +|x_n|.$ The $L^2$ norm is denoted by $\|x\|_2$. Note that we always  have $\|x\|_2 \le \|x\|\le \sqrt{n}\|x\|_2. $

\subsection{Bounds}\label{sec:bounds}
In what follows we need bounds on several quantities. We consider everything on $\Theta\times G.$
From \eqref{eq:s_bound}  we obviously have constants $B^s_t$ and $B^z$ such that for all $t$:
\begin{equation}
    |s_t^\theta| \le B^s_t,\quad  |z_t| \le B^z.
\end{equation}
We set  $B^s = \max_{t=0,1,\ldots,T} B_t^s.$
Since $\Theta$ is bounded we can define 
\begin{equation}
    B^\theta = \max_{\theta\in\Theta} \|\theta\|
\end{equation}
Now we can prove by induction:

\begin{lemma}\label{lem:boundssa}
    There exist constants $B_t^R,B_t^a>0, t=0,\ldots,T$ such that
    \begin{equation}
         |R_t^\theta(\omega)| \le B_t^R,\quad  |a_t^\theta(\omega)|\le B_t^a, \quad (\theta,\omega)\in\Theta\times G
    \end{equation}
    and we define $B^R := \max_{t=0,1,\ldots,T} B_t^R, B^a := \max_{t=0,1,\ldots,T} B_t^a.$ 
\end{lemma}

\begin{proof}
    The proof is by induction over $t$. For $t=0$ the statement is obvious. Suppose the bounds are correct up to time $t-1.$ We obtain by the continuity of $r$
   \begin{equation}
    |R_t^\theta| \le \sum_{k=0}^{t-1} |r(s_k^\theta,a_k^\theta)| \le \sum_{k=0}^{t-1} \max_{|s|\le B^s_k, |a|\le B^a_{k}}|r(s,a)| =: B^R_t. 
\end{equation} 

Further, by continuity of $f$ we have 
\begin{equation}
    |a_t^\theta| = |f(\theta,s_t^\theta,R_t^\theta,z_t,t)|\le 
    \max_{\theta\in\Theta, |s|\le B_t^s, |R|\le B_t^R, |z|\le B^z, k\le T} |f(\theta,s,R,z,k)| =: B^a_t
\end{equation}
which proves the statement.
\end{proof}
In particular with Assumption \ref{as:1} and on $\Theta\times G$ we can define 
$$\sigma_{max} :=  \max_{\|x\|\le B^s+B^R+B^z+T,\theta\in\Theta}|\sigma(W^1 x+b^1)|,\quad \sigma'_{max} := \sup_{y\in\RR} |\sigma'(y)|, $$
and $\sigma''_{max}$ is the Lipschitz constant of $\sigma'$
and $B^x := B^s+B^R+B^z+T.$
The proof of the next lemma follows directly from our assumptions.

\begin{lemma}
    Under  Assumptions \ref{as:1}-\ref{ass:rFLipgrad}  and with Lemma \ref{lem:boundssa} we obtain constants $B_F^i, B_r^i,B_\nabla^{f,\theta},B_\nabla^{f,i}  $ such that on $\Theta\times G$ for all $t$:
    \begin{equation}
    |\partial_i F(s^\theta_t(\omega),a^\theta_t(\omega),\varepsilon(\omega))| \le  B_F^i, \quad  |\partial_i r(s^\theta_t(\omega),a^\theta_t(\omega))| \le  B_r^i,\quad i=s,a
\end{equation}
\begin{equation}
    \|f_\theta(\theta,x^\theta_t(\omega))\| \le B_\nabla^{f,\theta},\quad |\partial_i f(\theta,x^\theta_t(\omega))| \le  B_\nabla^{f,i}, \quad i=s,R.
\end{equation}
\end{lemma} 
Finally we need bounds on $\nabla_\theta s_t^\theta, \nabla_\theta a_t^\theta, \nabla R_t^\theta. $ 
\begin{lemma}\label{lem:bound1}
    Under  Assumptions \ref{as:1}-\ref{ass:rFLipgrad}  there exist constants $B_\nabla^{s}, B_\nabla^a, B_\nabla^R$ such that on $\Theta\times G$ for all $t$:
    \begin{equation}
        \|\nabla_\theta s_t^\theta(\omega)\| \le B_\nabla^{s},\;  \|\nabla_\theta a_t^\theta(\omega)\| \le B_\nabla^a,\;  \|\nabla_\theta R_t^\theta(\omega)\| \le B_\nabla^R
    \end{equation}
\end{lemma} 

\begin{proof}
We have that 
\begin{eqnarray}\label{eq:nablas0}
\nabla_\theta s_t^\theta &=& \partial_sF(s_{t-1}^\theta,a_{t-1}^\theta,\varepsilon_t) \nabla_\theta s_{t-1}^\theta + \partial_aF(s_{t-1}^\theta,a_{t-1}^\theta,\varepsilon_t) \nabla_\theta a_{t-1}^\theta \\\label{eq:nablaa}
\nabla_\theta a_t^\theta &=& f_\theta(\theta,s_{t}^\theta,R_{t}^\theta,z_t,t)+ \partial_s f(\theta,s_{t}^\theta,R_{t}^\theta,z_t,t) \nabla_\theta s_{t}^\theta + \partial_R f(\theta,s_{t}^\theta,R_{t}^\theta,z_t,t) \nabla_\theta R_{t}^\theta\\
\nabla_\theta R_t^\theta &=& \sum_{k=0}^{t-1}  \partial_sr(s_k^\theta,a_k^\theta) \nabla_\theta s_k^\theta + \partial_ar(s_k^{\theta},a_k^{\theta}) \nabla_\theta a_k^\theta.
\end{eqnarray}
From this it follows inductively that $\|\nabla_\theta s_t^\theta\|\le B_\nabla^{s,t}, \|\nabla_\theta a_t^\theta\| \le B_\nabla^{a,t}, \|\nabla R_t^\theta\|\le B_\nabla^{R,t}$ and the constants satisfy the following recursion:
\begin{eqnarray*}
    B_\nabla^{s,t} &=& B_F^s B_\nabla^{s,t-1}+ B_F^a B_\nabla^{a,t-1}\\
    B_\nabla^{R,t} &=& \sum_{k=0}^{t-1} B_r^s B_\nabla^{s,k}+ B_r^a B_\nabla^{a,k}\\
    B_\nabla^{a,t} &=& B_\nabla^{f,\theta}+ B_\nabla^{f,s} B_\nabla^{s,t} + B_\nabla^{f,R}B_\nabla^{R,t}.
\end{eqnarray*}
As before, we define $B_\nabla^{s}=\max_t B_\nabla^{s,t}, B_\nabla^{a} =\max_t B_\nabla^{a,t}, B_\nabla^{R}=\max_t B_\nabla^{R,t}$
\end{proof}

\subsection{Lipschitz properties}
The proof of the following result is obvious, see e.g. \cite{hinderer2005lipschitz}, Lemma 2.1. 
\begin{lemma}\label{lem:lipschitzproperties}
Let $M,M'$ and $M''$ be metric spaces.
\begin{itemize}
\item[a)] If $f,g:M\to M'$ are Lipschitz continuous with constant $L_f$ and $L_g$, respectively, then  $f +g$ is Lipschitz continuous with constant  $L_f+L_g$.
\item[b)] If $f,g:M\to \RR$ are Lipschitz continuous with constant $L_f$ and $L_g$ respectively and  $| f|\le B_f$ and $| g|\le B_g$ on $M$ respectively, then $f \cdot g$ is Lipschitz continuous with constant $B_g L_f+B_f L_g. $
\item[c)] Let $f:M\times M'\to M''$. If $f$ is  Lipschitz continuous separately in the components, i.e.
$$ \|f(x,y)-f(\tilde x,y)\| \le L_1 \|x-\tilde x\|, \quad \|f(x,y)-f(x,\tilde y)\| \le L_2 \|y-\tilde y\|,$$
for suitable $L_1,L_2,$ then for $z=(x,y), \tilde z=(\tilde x,\tilde y)$ we have with $L := \max\{L_1,L_2\}$
$$ \|f(z)-f(\tilde z)\| \le L \|z-\tilde z\|$$
\item[d)] If $f:M\to M'$ and $g:M'\to M''$ are Lipschitz continuous with constant $L_f$ and $L_g$ respectively, then $g\circ f$ is again Lipschitz continuous with constant $L_f\cdot  L_g$.
\end{itemize}
\end{lemma}

Next we investigate the Lipschitz properties of our neural net. 
\begin{lemma}\label{lem:NNLipschitz}
Let $f : \RR^{11}\to \RR$ be the neural network defined in \eqref{eq:definition_NN_formula}, i.e.
$$ f(\theta,x) = w^2 \sigma(W^1 x+b^1)+b^2$$
for $\theta=(W^1,w^2,b^1,b^2)\in \RR^7$ and $x=(s,R,z,t)\in\RR^4.$ 
On $\Theta\times G$ we obtain
\begin{itemize}
\item[a)] $f$ is jointly  Lipschitz-continuous in $(\theta,x)$ i.e.,
$$|f(\theta,x)-f(\tilde\theta,\tilde x)| \le L_f \|(\theta,x)-(\tilde\theta,\tilde x)\| \mbox{ for all } \omega \in G \mbox{ and } \theta,\tilde\theta\in\Theta$$ with constant
$L_f := \max\{ 1, \sigma_{max}, B^{\theta}\sigma'_{max} ,B^{\theta} B^{x} \sigma'_{max}, (B^\theta)^2 \sigma'_{max}\}. $
\item[b)] $\nabla_\theta f$ is jointly  Lipschitz-continuous in $(\theta,x)$  i.e.
$$|\nabla_\theta f(\theta,x)-\nabla_\theta f(\tilde\theta,\tilde x)| \le L_f^\nabla \|(\theta,x)-(\tilde\theta,\tilde x)\| \mbox{ for all } \omega \in G \mbox{ and } \theta,\tilde\theta\in\Theta$$
with constant $L_f^\nabla = L_{b^1}^\nabla+L_{w^2}^\nabla+L_{W^1}^\nabla$
\begin{eqnarray*}
    L_{b_1}^\nabla & :=& \max\{\sigma'_{max},  B^{\theta}\sigma''_{max},  (B^\theta)^2\sigma''_{max}, B^{\theta} B^{x} \sigma''_{max}\}\\
     L_{w^2}^\nabla & :=& \sigma'_{max} \max\{ 1,B^x, B^{\theta}\}\\
 L_{W^1}^\nabla & :=& \max\{  B^x  \sigma'_{max}, B^{\theta} (B^{x})^2\sigma''_{max},  B^{\theta} B^{x} \sigma''_{max},  (B^{\theta})^2 B^{x} \sigma''_{max}+\sigma'_{max} B^\theta\}.
\end{eqnarray*}
\end{itemize}
\end{lemma}

\begin{proof}
a) According to Lemma \ref{lem:lipschitzproperties} it is enough to show Lipschitz continuity separately in the components. We obtain
\begin{eqnarray*}
|f(\theta,x)-f(\tilde\theta,x)| &\le& |b^2-\tilde b^2| + |w^2-\tilde w^2| \sigma(W^1x+b^1) + |\tilde w^2| |\sigma(W^1x+b^1)-\sigma(\tilde W^1x+\tilde b^1)|\\
&\le &  |b^2-\tilde b^2| + |w^2-\tilde w^2| \sigma(W^1x+b^1) + |\tilde w^2| |\sigma'(\cdot)| \Big(|b^1-\tilde b^1| +\|(W^1-\tilde W^1)x\| \Big) \\
&\le & \|\theta-\tilde\theta\| \max\{1,\sigma_{max}, B^\theta \sigma'_{max}, B^\theta B^x \sigma'_{max}\}.
\end{eqnarray*}
On the other hand we have
\begin{eqnarray*}
|f(\theta,x)-f(\theta,\tilde x)| &\le&  |w^2|  \sigma'_{max}  \|W^1(x-\tilde x)\| \le (B^\theta)^2 \sigma'_{max} \|x-\tilde x\|.
\end{eqnarray*}
This implies part a).

b) The gradient involves products of parameters of the neural net and $\sigma'$. More precisely, we obtain:
\begin{eqnarray*}
    f_{b^1}(\theta,x) &=& w^2 \sigma'(W^1x+b^1) \\
    f_{b^2}(\theta,x) &=& 1 \\
    \nabla_{W^1} f(\theta,x) &=& w^2 \sigma'(W^1x+b^1) \cdot x \\
    f_{w^2}(\theta,x) &=&  \sigma(W^1x+b^1)
\end{eqnarray*}
and 
$$\nabla_\theta f=\Big(f_{b^1}, f_{b^2}, \nabla_{W^1} f, f_{w^2}  \Big). $$
In particular, if $f_{b^1}, f_{b^2}, \nabla_{W^1} f, f_{w^2}$  are Lipschitz-continuous with constants $L_{b^2}=0$ and  $L_{b^1},L_{W^1},L_{w^2}$, then $\nabla_\theta f$ is  Lipschitz-continuous with constant  $L_f^\nabla =  L_{b^1}+ L_{b^2}+L_{W^1}+L_{w^2}.$ The Lipschitz constants can be obtained similarly to part a).
\end{proof}

\begin{lemma}\label{lem:LR}
Under Assumptions\ref{as:1}-\ref{ass:rFLip}
the mapping $\theta \mapsto R_T^\theta$ is a.s. locally Lipschitz continuous. On the set $\Theta\times G$ the function is globally Lipschitz continuous with a constant $K_{R,T}$ which can be obtained recursively, i.e.
$$| R_T^\theta(\omega) - R_T^{\tilde \theta}(\omega)| \le K_{R,T} \|\theta-\tilde\theta\| \mbox{ for all } \omega \in G \mbox{ and } \theta,\tilde\theta\in\Theta.$$
\end{lemma}

\begin{proof}    
We have
\begin{eqnarray*}
| R_T^\theta - R_T^{\tilde\theta}| \le \sum_{k=0}^{T-1} | r(s_k^\theta,a_k^\theta) - r(s_k^{\tilde\theta},a_k^{\tilde\theta})| \le \sum_{k=0}^{T-1} L_{r} \big( | s_k^\theta- s_k^{\tilde\theta}| +|a_k^\theta -a_k^{\tilde\theta}| \big)
\end{eqnarray*}
In what follows, denote
\begin{eqnarray*}
\Delta s_t &:=& s_t^\theta- s_t^{\tilde\theta},\\
\Delta a_t &:=& a_t^\theta- a_t^{\tilde\theta},\\
\Delta R_t &:=& R_t^\theta- R_t^{\tilde\theta}.
\end{eqnarray*}
We prove now that there exist constants $K_{s,t}, K_{a,t}, K_{R,t}$ such that for $t=0,1,\ldots ,T-1$
\begin{equation}\label{eq:deltasaR} | \Delta s_t| \le K_{s,t} \|\theta-\tilde\theta\|,\quad | \Delta a_t| \le K_{a,t}\|\theta-\tilde\theta\|, \quad | \Delta R_t| \le K_{R,t}\|\theta-\tilde\theta\|
\end{equation}
which implies the statement.
We do this by induction on $t$. For $t=0$ we have with Lemma \ref{lem:NNLipschitz} that
$$ \Delta s_0=0, \quad \Delta R_0 =0, \quad \Delta a_0\le L_{f} \|\theta-\tilde\theta\|. $$
Now suppose the statement is true up to time $t-1$.
From the state dynamics, the parametrization of the action and the definition of $R_t^\theta$ we obtain
\begin{itemize}
\item[(1)] $| \Delta s_t| \le \tilde L_{F} \big( |\Delta s_{t-1}|+  |\Delta a_{t-1}|\big)$
\item[(2)] $| \Delta R_t| \le \sum_{k=0}^{t-1}  L_{r} \big( | \Delta s_k| + | \Delta a_k| \big)$.
\item[(3)] $| \Delta a_t| \le L_{f} \big( \|\theta-\tilde\theta\| + |\Delta s_{t}|+ |\Delta R_{t}|\big)$
\end{itemize}
These relations obviously imply that the statement in \eqref{eq:deltasaR} is also true for time  $t$. It is also possible to derive the following recursive equations for the Lipschitz constants for $t=1,\ldots, T$:
\begin{eqnarray*}
    K_{s,t} = \tilde L_F \big( K_{s,t-1} +K_{a,t-1}\big),\\
    K_{R,t} = L_r \sum_{k=0}^{t-1}\big( K_{s,k} +K_{a,k}\big),\\
    K_{a,t} = L_f  ( 1+ K_{s,t} +K_{R,t}).
\end{eqnarray*}
The recursion may be simplified when we use larger constants. Outside  $\Theta\times G$, the same proof applies, however the Lipschitz constants may depend on  $\omega$ and $\theta$.
\end{proof}

\begin{lemma}\label{lem:LRgrad}
Under Assumptions\ref{as:1}-\ref{ass:rFLipgrad}
the mapping $\theta \mapsto \nabla_\theta R_T^\theta$ is a.s. locally Lipschitz continuous.  On the set $\Theta\times G$ the function is globally Lipschitz continuous with a constant $M_{R,T}$ which can be obtained recursively, i.e.
$$| \nabla_\theta R_T^\theta(\omega) - \nabla_\theta R_T^{\tilde\theta}(\omega)| \le M_{R,T} \|\theta-\tilde\theta\| \mbox{ for all } \omega \in G \mbox{ and } \theta,\tilde\theta\in\Theta.$$
\end{lemma}

\begin{proof}
We have
\begin{eqnarray}\label{eq:nablaR}
\nabla_\theta R_T^\theta = \sum_{t=0}^{T-1}  \partial_sr(s_t^\theta,a_t^\theta) \nabla_\theta s_t^\theta + \partial_ar(s_t^{\theta},a_t^{\theta}) \nabla_\theta a_t^\theta
\end{eqnarray}
In view of Lemma \ref{lem:lipschitzproperties} we have to show that $\partial_sr, \partial_ar,s_t^\theta, a_t^\theta, \nabla_\theta s_t^\theta , \nabla_\theta a_t^\theta$ are all locally Lipschitz.  The Lipschitz properties of $\partial_sr, \partial_ar$ follow from  Assumption \ref{ass:rFLipgrad}. The Lipschitz properties of $s_t^\theta, a_t^\theta$ follow from the previous proof. It remains to show the Lipschitz property of $\nabla_\theta s_t^\theta , \nabla_\theta a_t^\theta.$ 
In what follows, denote
\begin{eqnarray*}
D_{s,t}  &:=& \nabla_\theta s_{t}^\theta- \nabla_\theta s_{t}^{\tilde\theta},\\
D_{a,t} &:=& \nabla_\theta a_t^\theta- \nabla_\theta a_t^{\tilde\theta},\\
D_{R,t}  &:=& \nabla_\theta R_t^\theta- \nabla_\theta R_t^{\tilde\theta}.
\end{eqnarray*}
We prove now that there exist constants $M_{s,t}, M_{a,t}, M_{R,t}$ such that for $t=0,1,\ldots ,T-1$
\begin{equation}\label{eq:deltanablasaR} \| D_{s,t}\| \le M_{s,t} \|\theta-\tilde\theta\|,\quad \| D_{a,t}\| \le M_{a,t}\|\theta-\tilde\theta\|, \quad \| D_{R,t} \| \le M_{R,t}\|\theta-\tilde\theta\|.
\end{equation}
As in the previous proof, we proceed by induction.
For $t=0$ we have with Lemma \ref{lem:NNLipschitz} that
$$ D_{s,0} =0, \quad D_{R,0} =0, \quad D_{a,0}\le L_{f}^\nabla \|\theta-\tilde\theta\|. $$
Now suppose the statement is true up to time $t-1$. 
Before we proceed with the induction step, let us recall that $x=(s,R,z,t)$ and $\partial_sf := \frac{\partial}{\partial s}f$ and $\partial_Rf := \frac{\partial}{\partial R}f.$   Recall the bounds that we introduced on 
 $\Theta\times G$ in Section \ref{sec:bounds}.

We further have (cp. Section \ref{sec:bounds}):
\begin{eqnarray}\label{eq:nablas}
\nabla_\theta s_t^\theta &=& \partial_sF(s_{t-1}^\theta,a_{t-1}^\theta,\varepsilon_t) \nabla_\theta s_{t-1}^\theta + \partial_aF(s_{t-1}^\theta,a_{t-1}^\theta,\varepsilon_t) \nabla_\theta a_{t-1}^\theta \\\label{eq:nablaa2}
\nabla_\theta a_t^\theta &=& f_\theta(\theta,s_{t}^\theta,R_{t}^\theta,z_t,t)+ \partial_s f(\theta,s_{t}^\theta,R_{t}^\theta,z_t,t) \nabla_\theta s_{t}^\theta + \partial_R f(\theta,s_{t}^\theta,R_{t}^\theta,z_t,t) \nabla_\theta R_{t}^\theta.
\end{eqnarray}
From the first equation \eqref{eq:nablas} we obtain 
\begin{eqnarray*}
    M_{s,t} &=& B_F^s M_{s,t-1} + B_\nabla^s L_\nabla^{F,s}(K_{s,t-1} +K_{a,t-1})\\
    && + B_F^a M_{a,t-1} + B_\nabla^a L_\nabla^{F,a}(K_{s,t-1} +K_{a,t-1}).
\end{eqnarray*}

From the equation for $\nabla_\theta R_t^\theta$ in \eqref{eq:nablaR} we obtain:

$$M_{R,t} = \sum_{k=0}^{t-1} B_r^s M_{s,k} + B_r^a M_{a,k} + \big( B_\nabla^s L_\nabla^{r,s}+B_\nabla^a L_\nabla^{r,a} \big)\big(K_{s,k}+K_{a,k} \big)$$
It remains to discuss $M_{a,t}.$ To do so, we examine  $\partial_sf, \partial_R f$; cf. \eqref{eq:definition_NN_formula}:
\begin{eqnarray*}
    \partial_sf(\theta,x) &=& w^2 \sigma'(w^1_s s + w^1_R R+ w^1_z z + w^1_t t+ b^1) w^1_s\\ 
    \partial_Rf(\theta,x) &=& w^2 \sigma'(w^1_s s + w^1_R R+ w^1_z z + w^1_t t+ b^1) w^1_R.
\end{eqnarray*}
Using these formulas we obtain for the first expression (the second expression is similar):
\begin{eqnarray*}
  &&  |\partial_sf(\theta,s_t^\theta,R_t^\theta,z_t,t)-\partial_sf(\tilde\theta,s_t^{\tilde\theta},R_t^{\tilde\theta},z_t,t)| \\
&\le& |\partial_sf(\theta,s_t^\theta,R_t^\theta,z_t,t)-\partial_sf(\tilde\theta,s_t^{\theta},R_t^{\theta},z_t,t)| + |\partial_sf(\tilde\theta,s_t^\theta,R_t^\theta,z_t,t)-\partial_sf(\tilde\theta,s_t^{\tilde\theta},R_t^{\tilde\theta},z_t,t)|
\end{eqnarray*}
For the first expression we obtain:
\begin{eqnarray*}
  &&  |\partial_sf(\theta,s_t^\theta,R_t^\theta,z_t,t)-\partial_sf(\tilde\theta,s_t^{\theta},R_t^{\theta},z_t,t)| \le  \|\theta-\tilde\theta\|\;  \max\{B^\theta \sigma'_{max}, (B^\theta)^2B^x \sigma''_{max}  \}
\end{eqnarray*}
For the second expression we obtain:
\begin{eqnarray*}
  &&  |\partial_sf(\tilde\theta,s_t^\theta,R_t^\theta,z_t,t)-\partial_sf(\tilde\theta,s_t^{\tilde\theta},R_t^{\tilde\theta},z_t,t)| \le  \|\theta-\tilde\theta\|\;   (B^\theta)^3 \sigma''_{max} (K_{s,t}+K_{R,t})
\end{eqnarray*}
In total this yields 
\begin{eqnarray*}
  &&  |\partial_sf(\theta,s_t^\theta,R_t^\theta,z_t,t)-\partial_sf(\tilde\theta,s_t^{\tilde\theta},R_t^{\tilde\theta},z_t,t)| \\
&\le& \|\theta-\tilde\theta\|\; (B^\theta)^3 \sigma''_{max} (K_{s,t}+K_{R,t})+\|\theta-\tilde\theta\|\;\max\{ B^\theta \sigma'_{max}, (B^\theta)^2B^x \sigma''_{max} \  \}\\
&=:& \|\theta-\tilde\theta\|\;  L_{f,t} ^{\nabla_x} 
\end{eqnarray*}

The Lipschitz constant $L_{R,t} ^{\nabla_x}$ for $\theta\mapsto\partial_Rf(\theta,x^\theta_t)$ can be derived similarly.
Thus, we obtain from \eqref{eq:nablaa2} that
$$ M_{a,t} = L_f^\nabla(1+K_{s,t}+K_{R,t})+ B_\nabla^{f,s} M_{s,t} + B_\nabla^s L_{f,t} ^{\nabla_x} +
 B_\nabla^{f,R} M_{R,t} + B_\nabla^R L_{R,t} ^{\nabla_x} $$
Finally, the induction step is complete and the statement follows.
\end{proof}

\subsection{Further Proofs}
\begin{proof} of Theorem \ref{lem:Lipschitz_gradient}:
First, note that we can write the gradient more explicitly as
\begin{eqnarray*}
     \nabla_\theta \mathcal{L}_L(\theta) &=&
   2 \sum_{\ell=1}^L \beta_\ell u_\ell  \left[ \left(Re(\varphi^\star)(u_\ell)-\EE_G[\cos(u_\ell R_T^\theta)]\right)\cdot \mathbb{E}_G\left[ \sin(u_\ell R_T^\theta) \nabla_\theta R_T^\theta \right] \right] \\
   && - 2 \sum_{\ell=1}^L  \beta_\ell u_\ell \left[ \left(Im(\varphi^\star)(u_\ell)-\EE_G[\sin(u_\ell R_T^\theta)]\right)\cdot \mathbb{E}_G\left[ \cos(u_\ell R_T^\theta) \nabla_\theta R_T^\theta \right] \right] 
\end{eqnarray*}
Now from Lemma \ref{lem:LR} it follows that
$$\theta \mapsto \EE_G[\cos(u R_T^\theta)] \quad\mbox{   and    } \quad\theta \mapsto \EE_G[\sin(u R_T^\theta)]$$
are both Lipschitz-continuous with constant $|u| K_{R,T}.$
Lemma \ref{lem:LR} together with Lemma \ref{lem:LRgrad} imply that 
$$\theta \mapsto \EE_G[\sin(u R_T^\theta) \nabla_\theta R_T^\theta ] \quad \mbox{   and    }\quad \theta \mapsto \EE_G[\cos(u R_T^\theta) \nabla_\theta R_T^\theta]$$
are both Lipschitz-continuous with constant $M_{R,T} + |u| K_{R,T} B_\nabla^R.$ Using the rules for the Lipschitz constants from Lemma \ref{lem:lipschitzproperties} we obtain 
the following Lipschitz constant for $ \nabla_\theta \mathcal{L}_L$ using the $L_1$-norm:
\begin{eqnarray*}
    \tilde L &=& 8 M_{R,T} \sum_{\ell=1}^L |u_\ell|\beta_\ell  + 12 B_\nabla^R K_{R,T} \sum_{\ell=1}^L u_\ell^2\beta_\ell.
\end{eqnarray*}
Noting that for $x\in\RR^n$ we have $\|x\|_2 \le \|x\|_1\le \sqrt{n} \|x\|_2$ we obtain 
 the result with $\mathbf{L} =\sqrt{7} \tilde L$.
\end{proof}

\begin{proof}
of Lemma \ref{lem:bias}: We fix an iteration index $k\in\mathbb N$ and work throughout \emph{conditionally} on the
$\sigma$--algebra $\cF_k^G$.
Under this conditioning, the parameter $\theta_k$ is deterministic, and all bounds derived
on the good event $G$ hold pathwise.
In particular, by Lemma \ref{lem:bound1}
\[
\|\nabla_\theta R_T^{\theta_k}\|\;\le\; B_\nabla^R
\qquad\text{almost surely on } G.
\]

Let $(R_m,\nabla_\theta R_m)_{m=1}^M$ be i.i.d.\ copies of
$(R_T^{\theta_k},\nabla_\theta R_T^{\theta_k})$ under the current parameter $\theta_k$. Thus, these Monte Carlo samples are conditionally i.i.d.\ given $\cF_k^G$. All expectations below are therefore taken with respect to these samples only, while $\theta_k$ is treated as fixed.

\medskip
\noindent\textbf{Step 1: Bias.}

Fix a frequency node $u=u_\ell$ and, to simplify notation, suppress the index $\ell$
throughout this step.
Recall that for this node the empirical characteristic function is given by
\[
\widehat\varphi(u)
:=
\frac1M\sum_{j=1}^M e^{iu R_j}.
\]

For each sample index $j\in\{1,\dots,M\}$, define the pathwise gradient contribution
\[
U_j(u)
:=
iu\,e^{iu R_j}\,\nabla_\theta R_j
\;\in\;\mathbb{C}^7.
\]
With this notation, the contribution of the frequency node $u_\ell$ to the full estimator
$g(\theta_k)$ can be written as
\[
g_\ell
=
\frac{2\beta_\ell}{M}\sum_{j=1}^M
\text{Re}\Big(
\overline{\widehat\varphi(u)-\varphi^\star(u)}\;
U_j(u)
\Big).
\]

We emphasize that the same Monte Carlo batch
$(R_m)_{m=1}^M$ is used both to construct the empirical characteristic function
$\widehat\varphi(u)$ and to evaluate the gradient terms $U_j(u)$.
This coupling is the source of the finite-sample bias analyzed below. Taking conditional expectation and using linearity of $\text{Re}(\cdot)$,
\[
\EE[g_\ell\mid \cF_k^G]
=
\frac{2\beta_\ell}{M}\sum_{j=1}^M
\text{Re}\Big(\EE[\overline{\widehat\varphi(u)}\,U_j(u)\mid \cF_k^G]
-\overline{\varphi^\star(u)}\,\EE[U_j(u)\mid \cF_k^G]\Big).
\]
Because the samples are i.i.d., it suffices to compute $\EE[\overline{\widehat\varphi(u)}\,U_1(u)\mid \cF_k^G]$:
\[
\EE[\overline{\widehat\varphi(u)}\,U_1(u)\mid \cF_k^G]
=
\frac1M\EE[\overline{e^{iu R_1}}\,U_1(u)\mid \cF_k^G]
+\frac{M-1}{M}\EE[\overline{e^{iu R_2}}\mid \cF_k^G]\EE[U_1(u)\mid \cF_k^G].
\]
Now $\overline{e^{iu R_1}}\,U_1(u)= iu \nabla_\theta R_1$,
and we denote $\EE[{e^{iu R_2}}\mid \cF_k^G]=\EE_G[{e^{iu R_T^{\theta_k}}}] =:{\varphi_{\theta_k}(u)}$.
Thus $\EE[U_1(u)\mid \cF_k^G]=\nabla_\theta\varphi_{\theta_k}(u)$
and we obtain
\[
\EE[\overline{\widehat\varphi(u)}\,U_1(u)\mid \cF_k^G]
=
\frac1M\,iu\,\EE[\nabla_\theta R_1\mid \cF_k^G]
+\frac{M-1}{M}\,\overline{\varphi_{\theta_k}(u)}\,\nabla_\theta\varphi_{\theta_k}(u).
\]
Plugging back yields
\[
\EE[g_\ell\mid \cF_k^G]
=
2\beta_\ell\,\text{Re}\Big(\overline{\varphi_{\theta_k}(u)-\varphi^\star(u)}\,\nabla_\theta\varphi_{\theta_k}(u)\Big)
+\frac{2\beta_\ell}{M}\text{Re}\Big(iu\,\EE[\nabla_\theta R_1\mid \cF_k^G]-\overline{\varphi_{\theta_k}(u)}\,\nabla_\theta\varphi_{\theta_k}(u)\Big).
\]
Summing over $\ell$ gives \eqref{eq:bias-decomposition}--\eqref{eq:bias-term}.

To bound the bias, use $\|\nabla_\theta R_1\|_2\le \|\nabla_\theta R_1\|\le B_\nabla^R$ and
$\|\nabla_\theta\varphi_{\theta_k}(u)\|_2
\le \|\nabla_\theta\varphi_{\theta_k}(u)\|
\le \EE[|u|\,\|\nabla_\theta R_1\|\mid \cF_k^G ]\le |u| B_\nabla^R$,
and $|\varphi_{\theta_k}(u)|\le 1$, obtaining \eqref{eq:bias-bound}.

\medskip
\noindent\textbf{Step 2: Second moment.}
For each $(\ell,j)$ define
\[
Y_{\ell,j}
:=\mathrm{Re}\Big(\overline{\hat{\varphi}(u_\ell)-\varphi^*(u_\ell)}\,iu_\ell e^{iu_\ell R_j}\nabla_\theta R_j\Big)\in\RR^7,
\qquad
g(\theta_k)=\frac{2}{M}\sum_{\ell=1}^{L}\beta_\ell\sum_{j=1}^M Y_{\ell,j}.
\]
We have
$|\hat{\varphi}(u_\ell)-\varphi^*(u_\ell)|\le  2$
and $|e^{i\cdot}|=1$. Hence, for all $(\ell,j)$ we obtain on  $G$
\begin{equation}\label{eq:Y-bound}
\|Y_{\ell,j}\|
\le 2|u_\ell|\,\|\nabla_\theta R_j\|
\le 2|u_\ell|\,B_\nabla^R.
\end{equation}

Next, by the triangle inequality,
\[
\Big\|\sum_{j=1}^M Y_{\ell,j}\Big\|
\le \sum_{j=1}^M \|Y_{\ell,j}\|
\le M \max_{1\le j\le M}\|Y_{\ell,j}\|.
\]
Thus, we obtain
\[
\|g(\theta_k)\|
\le \frac{2}{M}\sum_{\ell=1}^{L}|\beta_\ell|
\Big\|\sum_{j=1}^M Y_{\ell,j}\Big\|
\le \frac{2}{M}\sum_{\ell=1}^{L}|\beta_\ell|\,
\Big(2M|u_\ell|B_\nabla^R\Big)
=4B_\nabla^R\sum_{\ell=1}^{L}|\beta_\ell u_\ell|.
\]
Therefore, on $G$,
\begin{equation}\label{eq:g-det-bound}
\|g(\theta_k)\|_2^2
\le 16\big(B_\nabla^R\big)^2
\Big(\sum_{\ell=1}^{L}|\beta_\ell u_\ell|\Big)^2.
\end{equation}
Taking conditional expectations preserves the inequality, so we obtain the statement.
\end{proof}

\subsection{Supplementary material for numerical examples}
\label{appx:numerical_examples}

\subsubsection{Hyperparameters}
Hyperparameters for our examples are given in Tables \ref{tab:algorithm_parameters}, \ref{tab:BS_algorithm_parameters},\ref{tab:compact_algorithm_parameters}.
\begin{table}[htbp]
    \centering
    \small
    \renewcommand{\arraystretch}{1.1}
    \begin{tabular}{@{}lc@{}}
        \toprule
        Algorithm hyperparameters & Values \\
        \midrule
        Training trajectories per iteration $M$     & 8192 \\
        Target samples $N_{\mathrm{target}}$        & 102400 \\
        Validation trajectories per batch $M_{\mathrm{val}}$
                                                    & 8192 \\
        Validation batches per evaluation           & 3 \\
        Training seeds                              & 2026,..,2035 \\
        Optimizer                                   & Adam \\
        Learning rate                               & $10^{-4}$ \\
        Fourier cutoff $K_{\mathrm{freq}}$           & 5 \\
        Number of Fourier nodes $L$                 & 301 \\
        Smoothing parameter $\alpha$                & 0.05 \\
        Quadrature rule                             & Composite trapezoidal \\  
        \bottomrule
    \end{tabular}
    \caption{Parameters of Example \ref{subse:LQ}}
    \label{tab:algorithm_parameters}
\end{table}

\begin{table}[htbp]
    \centering
    \small
    \renewcommand{\arraystretch}{1.1}
    \begin{tabular}{@{}lc@{}}
        \toprule
        Algorithm hyperparameters & Values \\
        \midrule
        Training trajectories per iteration $M$     & 102400 \\
        Target samples $N_{\mathrm{target}}$        & 20000 \\
        Validation trajectories per batch $M_{\mathrm{val}}$
                                                    & 4096 \\
        Validation batches per evaluation           & 3 \\
        Training seed                               & 2026 \\
        Target generation seed                      & 42 \\
        Training iterations                         & 400 \\
        Optimizer                                   & Adam \\
        Learning rate                               & $10^{-4}$ \\
        Fourier cutoff $K_{\mathrm{freq}}$           & 9 \\
        Number of Fourier nodes $L$                 & 301 \\
        Smoothing parameter $\alpha$                & 0.025 \\
        Quadrature rule                             & Composite trapezoidal \\       \bottomrule
    \end{tabular}
    \caption{Parameters of Example \ref{subse:wealth_example}}
    \label{tab:BS_algorithm_parameters}
\end{table}

\begin{table}[htbp]
    \centering
    \small
    \renewcommand{\arraystretch}{1.1}
    \begin{tabular}{@{}lc@{}}
        \toprule
        Algorithm hyperparameters & Values \\
        \midrule
        Training trajectories per iteration $M$      & 131072\\
        Target samples $N_{\mathrm{target}}$         & 102400 \\
        Validation trajectories $M_{\mathrm{val}}$   & 65536 \\
        Training and validation seed(s)              & 1 \\
        Number of training iterations $k$              & 2000 \\
        Optimizer                                   & Adam \\
        Learning rate                   & $10^{-4}$  \\
        Fourier cutoff $K_{\mathrm{freq}}$           & 100.0 \\
        Number of Fourier nodes $L$                  & 8001 \\
        Smoothing parameter $\alpha$                 & 0.015 \\
        Quadrature rule                              & Composite trapezoidal \\
        \bottomrule
    \end{tabular}
    \caption{Parameters of Example \ref{subse:compact_support}}
    \label{tab:compact_algorithm_parameters}
\end{table}


\subsubsection{Linear Quadratic Example}
We provide further action densities in the LQ example in Figure \ref{fig:lq_training}
\begin{figure}[H]
    \centering   
        \includegraphics[width=\textwidth]{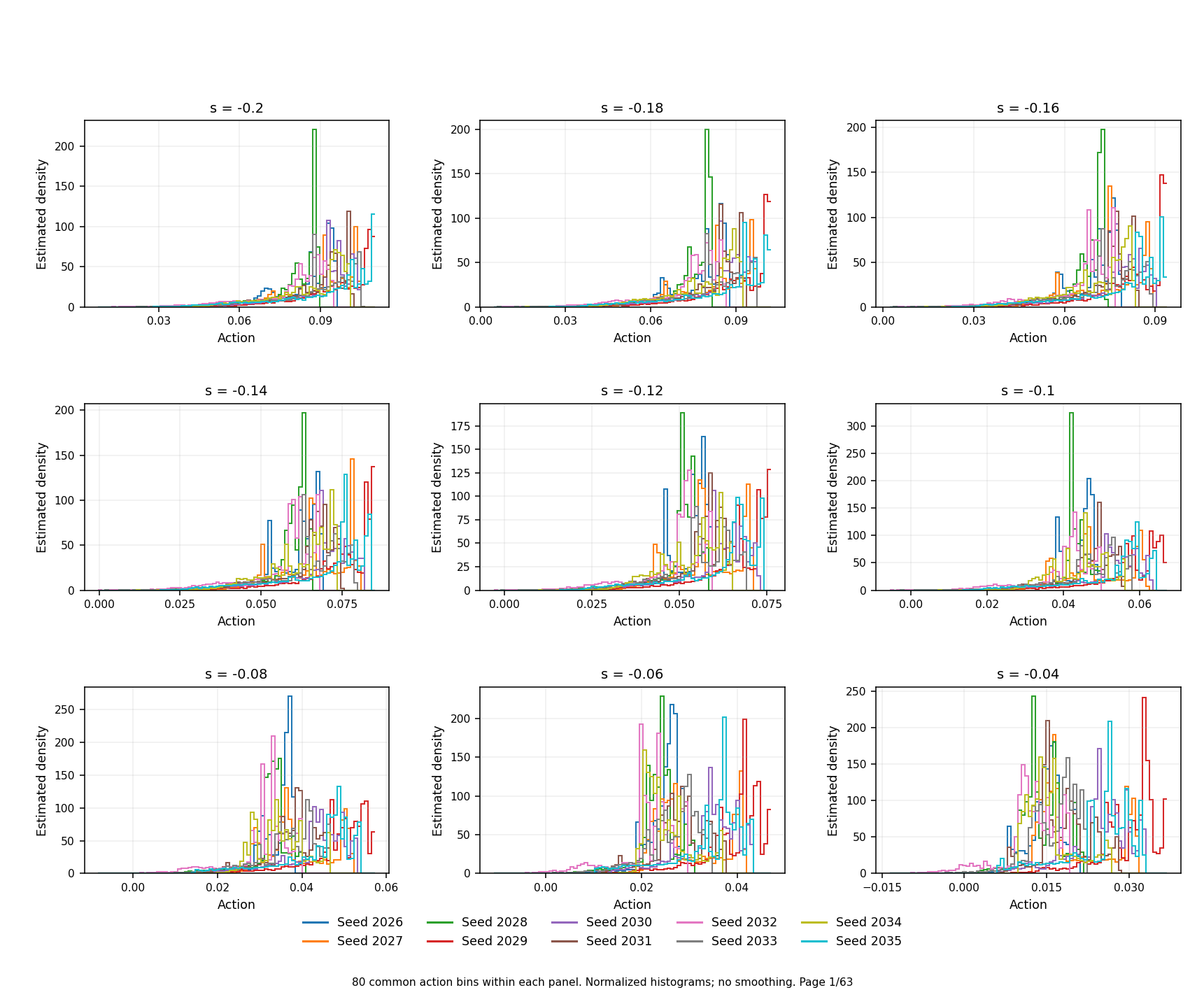}
        \caption{Action densities across seeds for $t=9$ and $r=-0.15$}
        \label{fig:lq_training}
\end{figure}

\bibliographystyle{apalike}
\bibliography{sample}
 
\end{document}